\pdfoutput=1 

\RequirePackage{fix-cm}
\documentclass[smallextended]{svjour3}
\usepackage[utf8]{inputenc}
\usepackage{booktabs}

\usepackage{amssymb,amsmath,epsfig,graphics,psfrag,graphicx,color,cite,arydshln,xcolor,caption,subcaption,url}
\definecolor{darkred}{rgb}{0.82,0.15,0.20}
\definecolor{lightblue}{rgb}{0.22,0.45,0.70}
\definecolor{cgray}{rgb}{0.7,0.7,0.7}
\usepackage[colorlinks=true,breaklinks=true,linkcolor=lightblue,citecolor=lightblue,urlcolor=lightblue]{hyperref}

\numberwithin{equation}{section}
\numberwithin{figure}{section}
\numberwithin{table}{section}
\numberwithin{lemma}{section}
\numberwithin{corollary}{section}
\numberwithin{theorem}{section}
\numberwithin{remark}{section}
\newcommand\bb{\boldsymbol{b}}

\newcommand\bu{\boldsymbol{u}}
\newcommand\bv{\boldsymbol{v}}
\newcommand\bw{\boldsymbol{w}}
\newcommand\bx{\boldsymbol{x}}

\newcommand\nn{\boldsymbol{n}}

\newcommand\bt{\boldsymbol{t}}

\newcommand\beps{\boldsymbol{\epsilon}}
\newcommand\bnabla{\boldsymbol{\nabla}}
\newcommand\bkappa{\boldsymbol{\kappa}}
\newcommand\bI{\mathbf{I}}
\newcommand\bC{\mathbf{C}}
\newcommand\bD{\mathbf{D}}

\newcommand\bF{\mathbf{F}}
\newcommand\bP{\mathbf{P}}
\newcommand\bR{\mathbb{R}}
\newcommand\bS{\mathbf{S}}
\newcommand\bH{\mathbf{H}}
\newcommand\bV{\mathbf{V}}

\newcommand\rQ{\mathrm{Q}}
\newcommand\rW{\mathrm{W}}
\newcommand\rH{\mathrm{H}}
\newcommand\rL{\mathrm{L}}

\newcommand\cT{\mathcal{T}}
\newcommand\cA{\mathcal{A}}
\newcommand\cB{\mathcal{B}}
\newcommand\cC{\mathcal{C}}
\newcommand\cF{\mathcal{F}}

\newcommand\cR{\mathcal{R}}
\newcommand\bcR{\boldsymbol{\mathcal{R}}}
\newcommand{\bigchi}{\mbox{\large$\chi$}}
\renewenvironment{proof}{\noindent{\it Proof.}}{\hfill$\square$}

\newcommand\fo{{\boldsymbol{f}_{\!0}}}
\newcommand\so{{\boldsymbol{s}_0}}

\newcommand\bdiv{\mathop{\mathbf{div}}\nolimits}
\newcommand\bDiv{\mathop{\mathbf{Div}}\nolimits}
\newcommand\vDiv{\mathop{\mathrm{Div}}\nolimits}
\newcommand\vdiv{\mathop{\mathrm{div}}\nolimits}
\newcommand\tr{\mathop{\mathrm{tr}}\nolimits}

\newcommand\cero{\boldsymbol{0}}
\allowdisplaybreaks

\journalname{Journal of Scientific Computing}
\begin{document}

\title{{Finite element} methods for large-strain poroelasticity/chemotaxis models simulating the formation of myocardial oedema}
\titlerunning{Poro-hyperelasticity/chemotaxis for myocardial oedema}

\author{N.A. Barnafi \and 
B. G\'omez-Vargas \and
W. J. Louren\c{c}o \and
R. F. Reis \and 
B. M. Rocha \and
M. Lobosco \and 
R. Ruiz-Baier \and
R. Weber dos Santos}

\authorrunning{Barnafi \textit{et al.}}

\institute{
Nicol\'as A. Barnafi \at Department of Mathematics ``Federigo Enriques'', 
Universit\`a degli Studi di Milano, Via Saldini 50, 
20133 Milano, Italy.\\
\email{nicolas.barnafi@unimi.it}.
\and
Bryan G\'omez-Vargas \at
Secci\'on de Matem\' atica,
Sede de Occidente, Universidad de Costa Rica, San Ram\'on
de Alajuela, Costa Rica.\\ 
\email{bryan.gomezvargas@ucr.ac.cr}.
\and 
Wesley de Jesus Louren\c{c}o, Ruy Freitas Reis, Bernardo Martins Rocha, Marcelo Lobosco, Rodrigo Weber dos Santos \at
Graduate Program on Computational Modeling, Federal University of Juiz de Fora, Jos\'e Louren\c{c}o Kelmer - Martelos, Juiz de Fora, Minas Gerais, Brazil.\\ 
\email{wesleydejesuspearl@ice.ufjf.br, ruyfreitas@ice.ufjf.br, bernardomartinsrocha@ice.ufjf.br, marcelo.lobosco@ice.ufjf.br, rodrigo.weber@ufjf.edu.br}.
\and 
Ricardo Ruiz-Baier (corresponding author) \at 
School of Mathematical Sciences and Victorian Heart Institute, Monash University, 9 Rainforest Walk, Melbourne 3800 VIC, Australia; and 
 World-Class Research Center ``Digital biodesign and personalized healthcare", Sechenov First Moscow State Medical University, Moscow, Russia; 
 and Universidad Adventista de Chile, Casilla 7-D, Chill\'an, Chile. \\
  \email{ricardo.ruizbaier@monash.edu}.
  }

\date{\today}
\maketitle

\begin{abstract}
In this paper we propose a novel coupled poroelasticity-diffusion model for the formation of extracellular oedema and infectious myocarditis valid in large deformations, manifested as an interaction between interstitial flow and the immune-dri\-ven dynamics between leukocytes and pathogens. The governing partial differential equations are formulated in terms of skeleton displacement, fluid pressure, Lagrangian porosity, and the concentrations of pathogens and leukocytes. A {five-field 
finite element} scheme is proposed for the numerical approximation of the problem, and we provide the stability analysis for a simplified system emanating from linearisation. We also discuss the construction of an adequate, Schur complement based,  nested preconditioner. The pro\-du\-ced computational tests exemplify the properties of the new model and of the {finite element} schemes.
\end{abstract}




\section{Introduction}
Poroelastic structures are found in many applications of industrial and scientific relevance. Examples include the interaction between soft permeable tissue and blood flow, or the study of biofilm growth and distribution near fluids (see \cite{ehret17,showalter05}). Recent applications of poroelastic consolidation theory to the poromechanical characterisation of soft living tissues include oxygen diffusivity in cartilage \cite{mauk03},  swelling of hydrogels \cite{yu19}, feather and scale development \cite{deoliveira20b}, tumour localisation and biomass growth \cite{sacco17}, cardiac perfusion \cite{chapelle10, cookson2012novel,barnafi2021multiscale}, chemically-controlled cell motion \cite{moee13}, lung characterisation \cite{berger16}, traumatic brain injury \cite{deoliveira20}, the formation of inflammatory oedema in the context of blood-brain barrier failure \cite{lang16}, and immune systems for small intestine \cite{young12,thompson19} as well as for myocarditis \cite{reis19b}.

In this work we are concerned with the latter application involving the formation and evolution of oedema, a build up of excess of fluid content in the intercellular space, in myocardial tissue. Local infection of tissue by pathogens contribute to an inflammatory reaction driven by the immune system with the aim of protecting the compromised region against invading organisms. Driven by oncotic and hydrostatic pressure gradients, there is a transvascular water flux: capillaries transport fluid into and out from the interstitium (the intercellular region), and lymphatic nodes contribute to returning fluid out from that space. {This work also considers that, before infection occurs, the body is under osmotic equilibrium since the capillary barrier is permeable to ions. In this case, ions have no significant effect on the interstitial flow under normal conditions, otherwise, a previous disease may exist. The presence of pathogens inside a cell as well as tissue damage caused by them may change this scenario, i.e., ion-concentration may change, impacting local pressure. However, in this work, we assume, for simplification purposes, that the main cause of oedema formation is the presence of the pathogen in the tissue and the immune response to it.} In the process of inflammatory response, local vasodilation occurs, which not only allows leukocytes to leave the bloodstream and to access the infection site but also increases the accumulation of fluid in the intercellular space, see Figure~\ref{fig:oedema}. However such abnormal accumulation of interstitial fluid content may lead to local stress generation, tissue deformation and swelling, eventually producing a number of complications that can impair the normal function of the tissue. Oedema may be found in several types of tissue. Heart oedemas are  commonly observed in myocarditis and they constitute one of the main criteria used in medical practice to its diagnosis~\cite{friedrich2009, puntmann2018}. Deriving refined continuum-based models for cardiac oedema formation valid in the large strain regime and that include 3D personalised geometries have the potential to describe mechanistic processes that are observable by clinical evidence and to provide deeper insight into the complex multiscale mechanisms that are inherent to the impairing of the healthy cardiovascular function. 

We extend the preliminary results from \cite{reis19,reis19b} and develop a simplified phenomenological model for the dynamic interaction between poroelastic finite strain deformations and the chemotaxis of leukocytes towards pathogens. The present framework assumes that the porous skeleton admits heterogeneities at the macro scale, and that the pores are all interconnected. We also consider that inertial forces are negligible. The solid-fluid interaction describing the distribution of flow within the deformable skeleton is then cast in Lagrangian coordinates using the solid displacement, the fluid pressure, and the nominal porosity (ratio between the pore volume and the total volume), whereas the immune system dynamics are represented by the concentrations of leukocytes and a pathogen. The model for the immune system has also been modified to better represent the dynamics that occur after the pathogen has been eliminated. Furthermore, the mechanical feedback into pathogens and leukocyte populations are now taken into account, {and we have explored these aspects further in our recent work \cite{lourenco22}, which uses the present formulation for the coupled poromechanical/chemotaxis equations}. 

Even if many numerical methods for the approximation of linear poroelasticity and their convergence analysis can be found in the recent literature (see \cite{Mikelic2013convergence,Both2017,both20,barnafi20,Both2019} and the references therein), much less attention has been given to formulations addressing large-strain poromechanics. In this regard, recent works include an enriched Galerkin framework \cite{choo19}, stabilised finite elements \cite{berger17,zheng20} and hybrid finite elements as well \cite{yu19}. Although most contributions address the problem in a Lagrangian frame of reference, some alternative formulations include descriptions in Eulerian \cite{rohan17} and ALE \cite{burtschell2017} coordinates. Besides the stress response, nonlinear models differ from the linear ones in that pressure is not a primary variable, but instead it is given by constitutive modelling \cite{Chapelle201482,Coussy2004}. Instead, we follow the framework of \cite{macminn16,cookson2012novel} (and other similar models) where the issue is circumvented through solid phase incompressibility, which yields a Lagrange multiplier acting as hydrostatic pressure.

The extension of such solution techniques for large scale simulations is highly non-trivial due to the requirement of efficient preconditioned iterative methods for the poroelasticity problem. Preconditioning strategies are mainly given by operator preconditioning \cite{baerland2017weakly,hong2019conservative} and block factorisations \cite{White201655,frigo2019relaxed}. In this regard, block preconditioners are more easily generalisable, but an efficient implementation requires an adequate approximation of the Schur complement, usually approximated through a fixed-stress formulation in poroelasticity \cite{franceschini2021approximate, frigo2019relaxed,White201655}. Despite all of this, an efficient numerical solver for a linearisation of large-strain poromechanics with nominal porosity as a primary variable remains largely unaddressed.

In the present case we formulate the set of nonlinear equations using as primary variables the solid displacement, the fluid pressure, the Lagrangian porosity, and the concentrations of leukocytes and pathogens. A 
{finite} element method is employed for the space discretisation \cite{ruiz15}, combined with a first-order backward Euler time advancing scheme. The method uses a MINI element for the displacement-porosity pair \cite{Arnold1984} (see also \cite{Auricchio2005,Chamberland2010} for the case of hyperelasticity), and linear Lagrangian elements for the remaining unknowns.  As the stability of the fully nonlinear system is still a paramount task (and an open problem even for the segregated nonlinear poroelasticity without the reaction-diffusion coupling, see, e.g., \cite{barnafi20}), we address the solvability of the linearised set of equations, focusing on the semidiscrete in-time formulation. For the corresponding analysis, we propose a fixed-point strategy, and then, Schauder fixed point theorem, together with the Fredholm alternative, the Babu\v ska-Brezzi theory, the Lax-Milgram lemma, and classical theory of quasi-linear equations, allow us to assert unique solvability robustly with respect to the material parameters of the linearised poroelastic solid. Then, we address the numerical approximation of the linearised problem through a Krylov subspace method with a Schur complement preconditioner.

The main advantages of the proposed mathematical model and the associated computational methods are i) a consistent theoretical framework valid for finite strains that stems as the natural generalisation of three-field linear poroelasticity from \cite{oyarzua16} (which is written in terms of displacement, fluid pressure, and \emph{total pressure}), ii) the versatility of the formulation to accommodate 2D or 3D geometries, iii) the accuracy of the numerical scheme and efficiency of the Krylov subspace iterative methods under consideration, and iv) the potential of the methodology in replacing current invasive methods for the detection of interstitial fibrosis and myocarditis (such as endo-myocardial biopsy) by techniques hinging only on MRI data.  

We have organised the contents of this paper in the following manner. Section~\ref{sec:model} outlines the main details of the model problem, motivating each component in the balance equations and stating main assumptions together with typical boundary conditions. In Section~\ref{sec:weak} we present the weak formulation of the nonlinear coupled problem, as well as the derivation of the linearised form of the weak formulation. Then, in Section~\ref{sec:stability}, we address the solvability of the linearised set of equations. Section~\ref{sec:fem} contains the derivation of {a} finite element scheme, including the fully discrete problem in matrix form, and we address the construction  of efficient solvers tailored for the multiphysics coupled system. In Section~\ref{sec:results} we present a number of computational results, consisting in a simple numerical study concerning the sensitivity analysis of relevant model parameters, the verification of spatio-temporal convergence, and an investigation of different cases on simplified and physiologically  accurate geometries. We close with some remarks and a discussion on model extensions collected in Section~\ref{sec:concl}. 

\section{Continuum model and proposed set of field equations}\label{sec:model}
The different scales and structural components in the tissue suggest to employ a continuum model. We assume that the tissue is a poroelastic medium with interstitial fluid completely filling the void spaces. Adopting usual kinematic conventions and notation, let us consider a domain $\Omega\subset \mathbb{R}^d$, $d=2,3$ representing the volume occupied by a deformable porous structure in its reference configuration. We will denote by $\nn$ the outward unit normal vector on the boundary $\partial\Omega$. We also disjointly partition the boundary into the sub-boundaries $\Gamma$ and $\Sigma$ (also in the reference undeformed state) where different types of boundary conditions, associated with the equations of motion, will be applied. A material particle in the undeformed domain $\Omega$ is identified by its position $\bx$ (and, unless specified otherwise, all differential operators such as gradients and divergences will be evaluated with respect to these undeformed coordinates), whereas $\bu:\Omega\to\bR^d$ will denote the displacement field defining its new position $\bx+\bu$ in the deformed configuration.  The tensor $\bF:=\bI+{\bnabla}\bu$ is the gradient (applied with respect to the fixed material coordinates) of the deformation map and $\bI$ denotes the identity tensor; its Jacobian determinant, denoted by $J=\det\bF = \det(\bI + {\bnabla}\bu)$, measures the dilation (solid volume change during the deformation); and $\bC=\bF^{\tt t}\bF$ is the right Cauchy-Green deformation tensor, where the superscript $()^{\tt t}$ denotes the transpose operator.

The non-viscous filtration flow through the non-deformed porous skeleton can be described by Darcy's law \cite{alves16,alves2019}. Adopting its parabolic formulation solely in terms of pressure head (which condenses momentum and mass conservation equations for the fluid) and combining it  with the conservation of linear momentum in the solid, we obtain the extension of Biot consolidation equations to the regime of large strains. This can be carried out using the descriptions already available in, e.g.,  \cite{macminn16,berger17,zheng20}. In this paper we choose to write them using the  first {Piola--Kirchhoff} stress tensor $\bP = \bP_{\mathrm{eff}}-\alpha pJ\bF^{-\tt t}$ and the effective (hyperelastic)  stress 
$\bP_{\mathrm{eff}} = \frac{\partial\Psi_s}{\partial\bF}$. 
For incompressible neo-Hookean and {Holzapfel--Ogden} solids one has the following forms (see, e.g., \cite{holzapfel09,quarteroni17}) 
\begin{align}
\nonumber\Psi_s^{\mathrm{NH}}(\bC) & = \frac{\mu_s}{2}(I_1(\bC)-d),\\ 
\Psi_s^{\mathrm{HO}}(\bC)  & = \frac{a}{2b}e^{b(I_{1}(\bC)-d)}
+\dfrac{a_{fs}}{2b_{fs}}\bigl[e^{b_{fs}I_{8,fs}(\bC)^{2}}-1\bigr]+ 
\sum_{i\in\{ f,s\} } \dfrac{a_{i}}{2b_{i}}\bigl[e^{b_{i}(I_{4,i}(\bC)-1)_+^{2}}-1\bigr],\label{eq:HO}\end{align}
respectively, where $\mu_s$, $a,b,a_i,b_i$ with $i\in\{f,s,fs\}$ are material parameters and we have used the notation $(\zeta)_+:= \zeta$ if $\zeta>0$ or  zero otherwise, and the invariants and pseudo-invariants of $\bC$ measuring isotropic and direction-specific stretches $I_{1}(\bC)=\tr\bC$, $I_{4,f}(\bC)=\fo\cdot(\bC\fo)$, $I_{4,s}(\bC)=\so\cdot(\bC\so)$, $I_{8,fs} (\bC)=\fo\cdot(\bC\so)$.  Here $\fo$ and $\so$ denote local alignment of fibres and sheetlets which is inherent to the local microstructure of the myocardial tissue \cite{nash00}. Histological studies based on diffusion tensor imaging indicate a fibre architecture that rotates 120$^\circ$ from the epicardial to the endocardial surface. A volumetric part of the energy is added for the case of nearly incompressible materials, for example,  the term $U(J) = \frac{1}{2} \lambda_s (\mathrm{ln} J)^2$ (with $\lambda_s$ another material parameter) which is designed to enforce convexity or polyconvexity of the strain density \cite{zheng19}. 

Provided that the volumetric stiffness of the solid is substantially smaller than that of the constituent (which is reasonable for soft living tissues, see, e.g., \cite{mauk03,lang16}), volume changes of solid constituent are negligible, implying that
\[J = 1 + \phi_f - \phi_0,\]
where $\phi_f - \phi_0$ is the nominal porosity change in the reference configuration \cite{macminn16} (here we are assuming that the mixture is saturated, that is, $\phi_f + \phi_s = 1$, where $\phi_s$ is the volume fraction of the solid). Recall that the Lagrangian porosity is the natural state variable being work-conjugate to the fluid pressure $p$ (see, e.g., \cite{nedjar13}). Should pre-stress be considered, we need to take instead $J = J_0 + \phi_f - \phi_0$, where $J_0$ is the volume associated with the pre-stress. Note that the specific form of the strain energy density $\Psi_s$ is considered irrespective of the fluid saturating the solid. It can also be derived from subtracting the volumetric free energy of the fluid from the total Helmholtz free energy of the mixture $\Psi_s = \Psi - \phi_f \Psi_f$, where $\phi_f$ is  the nominal (Lagrangian) porosity measuring the current fluid volume per unit reference total volume \cite{zheng20,nedjar13}. Note also that rearrangement of tissue components as a result of change in fluid content will imply an additional solid deformation as well as a stress modification \cite{lang16,choo19,phipps11}. 
  
The problem is then  stated in a Lagrangian framework and in terms of the solid displacement $\bu$, the fluid pressure $p$, and the nominal porosity $\phi_f$, as 
\begin{subequations}
\label{eq:poroe}
\begin{align}
\label{eq:u}
-\bDiv[\bP_{\mathrm{eff}} - \alpha pJ\bF^{-\tt t}]& = \rho_s\bb & \text{in $\Omega\times(0,t_{\text{final}}]$},\\
\label{eq:p}
\rho_f D_t \phi_f - \frac{1}{J}\vDiv \biggl({\phi_f\rho_f} J\bF^{-1}\frac{\bkappa(J,\phi_f)}{\mu_f}\bF^{-\tt t} \nabla p \biggr) & = \rho_f\ell & \text{in $\Omega\times(0,t_{\text{final}}]$},\\
  \label{eq:phi}
  J -\phi_f & = 1 - \phi_0  & \text{in $\Omega\times(0,t_{\text{final}}]$},
\end{align}\end{subequations}
{where $D_t$ stands for the material time derivative, and $\vDiv$ and $\bDiv$ are the divergence operators (with respect to the coordinates in the reference configuration)  for vector and tensor fields, respectively}. 
Equation \eqref{eq:u} is the balance of linear momentum for the multiphase system expressed in terms of the true total stress; \eqref{eq:p} states the mass conservation (or continuity) of fluid content in a material framework, written in terms of the nominal flux that includes the movement of the fluid phase due to pressure gradient under Darcy's law (which in particular discards shear flow effects and is valid for small Reynolds numbers); and \eqref{eq:phi} represents the aforementioned kinematic  constraint relating volume and porosity. {It states 
the material incompressibility of the constituents, but it still allows for compression of the medium since the pore volume can change locally in connection with modifications in the macroscopic mass density of the poroelastic material. A more complete discussion can be found in \cite{macminn16}.}
 This description is independent of the material constitutive model for the solid matrix, and  the balance of angular momentum simply furnishes the condition $\bF\bP^{\tt t} = \bP\bF^{\tt t}$. Here $D_t$ indicates the material time derivative (with respect to skeleton particles), $\bb$ is a vector of external body loads, $\ell$ is a distributed source/sink strength term (local source volume of fluid added per volume of tissue per unit time) that accounts for fluid being added or removed from the interstitium by the blood and lymph capillaries and its specific form will be made precise later on, $\rho_s$ is the reference density of the porous matrix, $\rho_f$ is the reference density of the interstitial fluid, $\alpha$ is the {Biot--Willis modulus assuming values between 0 and 1}, and $\mu_f$ is the dynamic viscosity of the interstitial liquid. 

The porosity-dependent heterogeneous tensor of intrinsic permeability (or hydraulic conductivity) $\bkappa$ can be defined by using normalised {Kozeny--Carman} (KC), exponential (E), or by power-law (PL) relations defined, respectively, as
\begin{gather*} \bkappa^{\mathrm{KC}}_{\mathrm{iso}}(J,\phi_f) = \kappa_0\bI+\frac{\kappa_0}{\phi_0^3} (1-\phi_0)^2J\phi_f^3(J-\phi_f)^2\bI, \\ \bkappa^{\mathrm{E}}_{\mathrm{iso}}(J,\phi_f) = \kappa_0 \biggl(\frac{J-\phi_f^3}{1-\phi_f^3}\biggr)^3\exp(M_0(J^2-1)/2)\bI,\quad 
  \bkappa^{\mathrm{PL}}_{\mathrm{iso}}(J,\phi_f) = \kappa_0 \biggl(\frac{J\phi_f}{\phi_0}\biggr)^{2/3}\bI,
  \end{gather*}
where $\kappa_0 = d_m^2\phi_0^3/180(1-\phi_0)^2$, $d_m$ is the typical  diameter of pores or of solid grains, and $M_0$ is the slope of the normalised permeability curve \cite{zheng20,ateshian10,berger16}. As porosity is a volumetric quantity, these forms of permeability are  isotropic and therefore indifferent to rotation and shear effects. A relatively simple phenomenological extension that accommodates anisotropy due to tissue microstructure is given as follows (using as example the {Kozeny--Carman} law) 
\begin{equation}\label{eq:kappa}
  \bkappa^{\mathrm{KC}}(\phi_f) = \kappa_0\bigl(1+\frac{1}{\phi_0^3} (1-\phi_0)^2J\phi_f^3(J-\phi_f)^2\bigr) \biggl(\frac{\kappa_{\fo}}{I_{4,f}} \bF\fo\otimes \bF\fo + \frac{\kappa_{\so}}{I_{4,s}} \bF\so\otimes \bF\so\biggr),
  \end{equation}
where $\kappa_{\fo},\kappa_{\so}$ are weights for the principal hydraulic conductivities such that  the ratio $\kappa_{\fo}/\kappa_{\so}\approx 2.5$ coincides with the anisotropy ratio typically used for the conductivity tensor in the extracellular domain of bidomain models \cite{colli14,sundnes07} (and much milder than the ratio observed in the intracellular domain, or in other porous skeletons such as sand or cartilage \cite{burger97,federico07}). This allows to reproduce the preferential perfusion direction of the micro-vascular system following the orientation of cardiac fibres $\fo,\so$. Other transformations that go beyond the kinematic nonlinearity induced by change of frame can be found in models exploiting strain-dependence \cite{ateshian10} or stress-assistance mechanisms \cite{cherubini17}. Unless otherwise specified, we will use \eqref{eq:kappa}. 

\begin{figure}
    \centering
    \includegraphics[width=.65\textwidth]{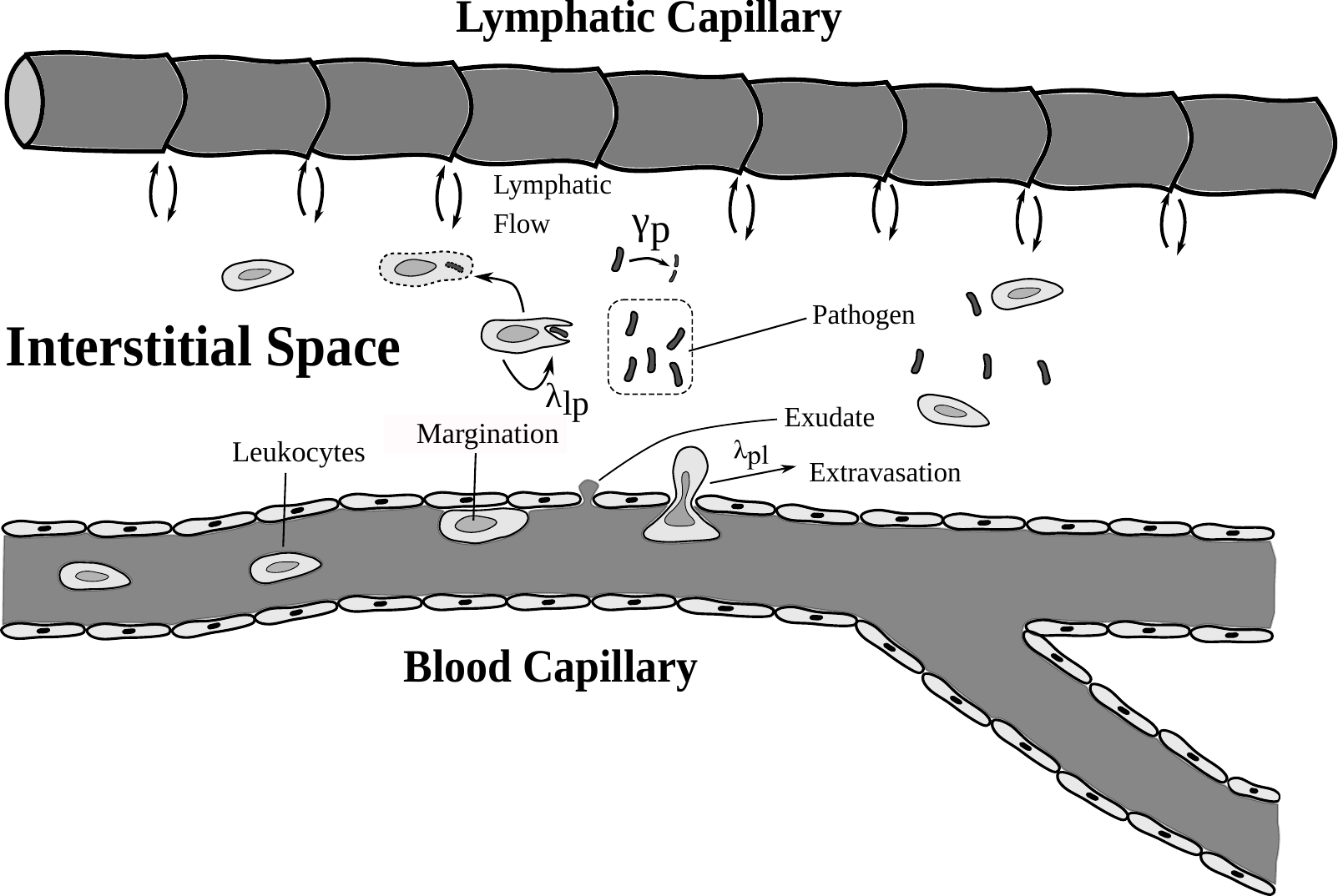}
    \caption{Sketch of local infection of tissue by pathogens, triggering inflammatory and immune responses. In the inflammatory response process, local increment vascular permeability occurs, which allows leukocytes to leave the bloodstream to access the infection site which also leads to fluid accumulation in the interstitial space. Driven by oncotic and hydrostatic pressure gradients, blood capillaries transport fluid into and out of the interstitium, and lymphatic capillaries contribute to the resolution of the oedema.}\label{fig:oedema}
\end{figure}

The infection process starts when an invader (pathogen) enters the poroelastic tissue. Once there, it starts to propagate. The exact mechanism adopted for increasing the pathogen population (replication, in the case of viruses, or reproduction, otherwise) is not essential to our model. In our simplification, $\gamma_p$ represents the rate at which the pathogen propagates. The innate immune cells (leukocytes) use their receptors to sense the presence of non-self molecules characteristic of pathogens. After the presence of pathogens is identified, leukocytes start to produce pro-inflammatory cytokines, which increase vascular permeability and recruit more leukocytes to the site of infection. The model does not include cytokines because the concentration of pathogens and leukocytes can indirectly represent their effect on vascular permeability and recruitment of leukocytes. {Additionally, we assumed a simplified reaction term for the leukocytes since among these cells some have a very long life span when compared to the timescale of the simulations considered in this work. For example, macrophages can live from months to years\cite{vanFurth1968}. Evidently, leukocyte death from apoptosis and pathogen phagocytosis create a sink that was not considered in the current work.} Leukocytes adhere to the endothelial cells that line the vessels, a phenomenon called margination (Figure~\ref{fig:oedema}). The increase in vascular permeability facilitates leukocyte extravasation to the tissue. The model considers that leukocytes leave the bloodstream to the tissue with an extravasation rate equals to $\lambda_{pl}$. The leukocytes move towards the site of infection due to the presence of chemoattractants. The chemotactic rate is given by $\chi_{cl}$. Finally, leukocytes can start the phagocytosis of pathogens, eliminating them at a rate equal to $\lambda_{lp}$ (Figure~\ref{fig:oedema}). We denote by $c_p$ and $c_l$ the concentrations of pathogens and leukocytes, respectively. Both populations can also diffuse in the interstitial fluid. Their dynamics, in the current configuration and with derivatives in terms of the deformed coordinates, are governed by 
\begin{align*}
\partial_t (\phi_fc_p) + (\partial_t \bu  - \frac{\bkappa}{\mu_f} \nabla p)\cdot \nabla c_p 
- \vdiv(\bD_p \nabla c_p) & = r_p(\phi_f,c_p,c_l) \ \text{in $\Omega_t\!\times\!(0,t_{\text{final}}]$},\\
\partial_t (\phi_fc_l) + (\partial_t \bu  - \frac{\bkappa}{\mu_f} \nabla p)\cdot \nabla c_l 
- \vdiv(\bD_l \nabla c_l - \chi c_l\nabla c_p) & = r_l(\phi_f,c_p,c_l) \ \text{in $\Omega_t\!\times\!(0,t_{\text{final}}]$},
\end{align*}
where $\bD_l,\bD_p$ are positive-definite diffusion tensors for the species within the interstitial fluid,  and $r_l,r_p$ are reaction terms adopting the following specific forms (see, e.g., \cite{reis19b}) 
\begin{gather*}
r_p(\phi_f,c_p,c_l) = \phi_f(\bar{\gamma_p} - \bar{\lambda}_{lp}{c_l})c_p, \quad
r_l(\phi_f,c_p,c_l) = \bar{\lambda}_{pl}{\phi_f}c_pc_l,
\end{gather*}  
with $\bar{\gamma}_p,\,\bar{\lambda}_{lp}$, and $\bar{\lambda}_{pl}$ being constant coefficients (pathogen growth rate, leukocyte phagocytosis rate, and leukocyte migration rate, respectively) \cite{reis19b}. Note that the advecting velocity corresponds to the filtration velocity (that between the solid and fluid velocities, as the species are supposed to be transported in the interstitial fluid). Pulling back these advection-reaction-diffusion equations to the reference frame we obtain a modified system with kinematic nonlinearity 
\begin{subequations}
\label{eq:immune}
\begin{align}
\label{eq:cp}
D_t (\phi_fc_p) - \frac{1}{J}\vDiv({\phi_f}J\bF^{-1}\bD_p\bF^{-\tt t} \nabla c_p) & =  r_p(\phi_f,c_p,c_l) \quad \text{in $\Omega\times\!(0,t_{\text{final}}]$},\\
  \label{eq:cl}
  D_t(\phi_f c_l) - \frac{1}{J}\!\vDiv({\phi_f}J\bF^{-1}\bD_l\bF^{-\tt t} \nabla c_l- \chi {\phi_f} c_l J\bF^{-1}\bF^{-\tt t}\nabla c_p ) &= r_l(\phi_f,c_p,c_l) \quad \text{in $\Omega\times\!(0,t_{\text{final}}]$}.\end{align}\end{subequations}
Note that, in order to make the advection completely absorbed by the total derivative we have considered the total fluid-solid velocity. In turn, this implies that we are assuming that pathogens and leukocytes are able to move in the solid also. 

Next we return to the continuity equation \eqref{eq:p} to specify the source (rate of fluid movement) $\ell$. This term will encode the immune response feedback into the hydromechanics by setting a Starling-Hill function of the type
\[ \ell(p,c_p) =  c_f(c_p)[p_c-p-\sigma(c_p)(\pi_c-\pi_i)]-\ell_0\biggl[1+\frac{v_{\max}(p-p_0)^n}{k_m^n + (p-p_0)^n}\biggr],\]
specified by the nonlinear functions of the pathogen concentration (capillary conductivity and  reflection coefficient, respectively) 
\[c_f(c_p) =(S/V)L_{p0}(1+L_{bp}c_p), \qquad \sigma(c_p)=\sigma_0(1+L_{br}c_p)^{-1}.\]

The remaining coefficients $\ell_0,p_c,\pi_c,\pi_i,v_{\max},k_m,n,p_0,S/V, L_{p0},L_{bp}, L_{br}$ (normal lymph flow, capillary basal pressure, capillary oncotic pressure, interstitial oncotic pressure due to plasma protein, maximal lymph flow, half-live of pressure to lymph flow, Hill coefficient, normal interstitial fluid pressure, vessel area per volume unit, healthy tissue hydraulic permeability of the capillary wall, intensity of the pathogen infection into permeability, and influence of pathogens in the reflection coefficient; respectively) are positive model constants \cite{reis19b}.

To close the system composed by \eqref{eq:u}-\eqref{eq:phi} and\eqref{eq:cp}-\eqref{eq:cl}, we need to provide suitable initial data and boundary conditions. We suppose that the system is initially at rest and with an homogeneous distribution of the leukocytes and pathogens in the domain
\begin{equation}
\label{eq:initial}
\phi_f(0) = \phi_0,\quad p(0)=0, \quad \bu(0) = \cero, \quad c_l(0) = c_{l,0},\quad c_p(0)=c_{p,0} \quad \text{in $\Omega\times\{0\}$.}
\end{equation}
In addition, on the boundaries  $\Gamma$ and $\Sigma$ we prescribe traction and displacement,  respectively; and zero flux everywhere on the boundary for the species interacting in the immune system. That is, we consider the following set of boundary conditions 
\begin{subequations}
\begin{align}
\label{bc:Gamma}
\bu = \cero, \quad \text{and} \quad {\phi_f} J\bF^{-1}\frac{\bkappa(\phi_f)}{\mu_f}\bF^{-\tt t} \nabla p \cdot\nn &= 0 &\text{on $\Gamma\times(0,t_{\text{final}}]$},\\
\label{bc:Sigma}
\bP\nn = J\bF^{-\tt t} \bt_\Sigma, \quad\text{and}\quad p& =p_{0}  &\text{on $\Sigma\times(0,t_{\text{final}}]$},\\
  {\phi_f} J\bF^{-1}\bD_p\bF^{-\tt t} \nabla c_p \cdot\nn = [ {\phi_f} J\bF^{-1}\bD_l\bF^{-\tt t} \nabla c_l- \chi {\phi_f} c_lJ\bF^{-1}\bF^{-\tt t}\nabla c_p]\cdot\nn & =0  &\text{on $\partial\Omega\times(0,t_{\text{final}}]$},\label{bc:immune}
\end{align}\end{subequations}
where $\bt_\Sigma$ is a given boundary traction applied in the current configuration.  
 
\section{Weak formulation and consistent linearisation}\label{sec:weak}
We derive a weak form for the coupled system, carry out a {Newton--Raphson} linearisation, and then discuss the solvability of a simplified system arising from the linearisation with respect to conditions at rest. That resulting problem coincides in structure with the coupled Biot/reaction-diffusion equations. 

\subsection{General weak form and in-time discretisation}\label{sec:linearisation}
In view of the essential boundary conditions for displacement and fluid pressure \eqref{bc:Gamma},\eqref{bc:Sigma}, we  define the Hilbert spaces 
\begin{gather*}
\bH^1_\Gamma(\Omega) = \{ \bv\in \bH^1(\Omega): \bv|_{\Gamma} = \cero\},\quad 
 \rH^1_\Sigma(\Omega) = \{ q\in \rH^1(\Omega): q|_{\Sigma} = 0\},
\end{gather*}
associated with the classical norms in $\bH^1(\Omega)$, and $\rH^1(\Omega)$, respectively. The nonlinear weak form of the equations of motion and of the transport for the immune system can be established following a standard approach, that is, multiplying each field equation by a suitable test function, integrating over $\Omega$, and invoking the divergence theorem whenever appropriate. To lighten the presentation, we drop the measures $d\bx$ and $d\mathrm{s}$ from the integrals, which are to be understood throughout the manuscript as the Lebesgue measure. This leads to the following continuous, five-field weak   formulation: For almost all $t>0$, find $(c_p,c_l)\in\rH^1(\Omega)^2$, $\bu\in\bH^1_{\Gamma}(\Omega)$, $p\in\rH^1_{\Sigma}(\Omega)$, and $\phi_f\in \rL^2(\Omega)$,  such that 
\begin{align}
  \nonumber
  \int_{\Omega} J D_t(\phi_f c_p) w_p + \int_{\Omega} {\phi_f}J\bF^{-1}\bD_p\bF^{-\tt t}\nabla c_p\cdot \nabla w_p & =
  \int_\Omega J r_p(\phi_f,c_p,c_l)w_p &\forall w_p \in\rH^1(\Omega),\\
  \nonumber
  \int_{\Omega} J D_t(\phi_f c_l) w_l + \int_{\Omega} {\phi_f}J\bF^{-1}(\bD_l\bF^{-\tt t}\nabla c_l
  -  \chi c_l \bF^{-\tt t}\nabla c_p)\cdot \nabla w_l & =
  \int_\Omega J r_l(\phi_f,c_p,c_l)w_l &\forall w_l \in\rH^1(\Omega),\\
  \label{eq:weak}
  \int_{\Omega} (\bP_{\mathrm{eff}} - \alpha p J\bF^{-\tt t}):{\bnabla}\bv - \int_{\Sigma} J\bF^{-\tt t}\bt_\Sigma\cdot\bv & = \int_{\Omega} \rho_s\bb\cdot\bv & \forall \bv \in \bH^1_{\Gamma}(\Omega),\\
  \nonumber
  \int_{\Omega} \rho_fJ D_t\phi_f q +  \int_{\Omega} {\rho_f\phi_f}J\bF^{-1}\frac{\bkappa(\phi_f)}{\mu_f}\bF^{-\tt t}\nabla p \cdot \nabla q & = \int_\Omega \rho_fJ\ell(p,c_p)\,q &\forall q \in\rH^1_{\Sigma}(\Omega),\\
 \nonumber
\int_\Omega (J - \phi_f)\psi  & = \int_\Omega (1-\phi_0)\psi  & \forall \psi\in \rL^2(\Omega).
\end{align}
We do not address time regularity in this work, and simply assume that the solution is weakly differentiable and continuous in time, the latter in order to obtain well-defined initial conditions.
Next we partition the interval $[0,t_{\mathrm{final}}]$ into $N$ evenly spaced non-overlapping sub-intervals of fixed length $\Delta t$ and apply a time semi-discretisation of \eqref{eq:weak} using the unconditionally stable, backward Euler's method. Starting from initial data, at each time iteration $n=0,\ldots, N-1$ we have a nonlinear  problem with solution $(c_p^{n+1},c_l^{n+1})\in\rH^1(\Omega)^2$, $\bu^{n+1}\in\bH^1_{\Gamma}(\Omega)$, $p^{n+1}\in\rH^1_{\Sigma}(\Omega)$, and $\phi_f^{n+1}\in \rL^2(\Omega)$. Then we perform a Newton-Raphson method linearisation. This means that at each time step $n+1$ we start the Newton iterates using as initial guess the solution at the previous time step $c_p^{k=0} = c_p^n$, $c_l^{k=0} = c_l^n$, $\bu^{k=0} = \bu^n$, $p^{k=0} = p^n$,  $\phi_f^{k=0} = \phi_f^n$, and for $k=0,1,\ldots$ we look for the solution increments $(\delta c_p^{k+1},\delta c_l^{k+1})\in\rH^1(\Omega)^2$, $\delta\bu^{k+1}\in\bH^1_{\Gamma}(\Omega)$, $\delta p^{k+1}\in\rH^1_{\Sigma}(\Omega)$, $\delta\phi_f^{k+1}\in \rL^2(\Omega)$ satisfying 
\begin{alignat}{5}
&  a_1(\delta c_p^{n+1},w_p)   &+ b_1^*(\delta c_l^{n+1},w_p) + & b_2^* (\delta \bu^{n+1}, w_p)& &+ b^*_3(\delta \phi_f^{n+1},w_p) &=& F_1(w_p)\ &\forall w_p \in \rH^1(\Omega),\nonumber\\
&  b_1(w_l, \delta c_p^{n+1}) &+ a_2(\delta c_l^{n+1},w_l) + & b^*_4(\delta \bu^{n+1}, w_l)& &+ b_5^*(\delta \phi_f^{n+1},w_l) &=& F_2(w_l) &\forall w_l \in \rH^1(\Omega),\nonumber\\
  \label{eq:weak-linear}
&&&a_3(\delta\bu^{n+1},\bv) &+  c_1^*(\delta p^{n+1},\bv) &  &=& F_3(\bv) &\forall \bv \in\bH^1_{\Gamma}(\Omega),\\
  \nonumber
&c_2(q,\delta c_p^{n+1})&+ & c_1(q,\delta\bu^{n+1}) &+  a_4(\delta p^{n+1},q) &+ c_3^*(\delta \phi_f^{n+1},q)  &= &F_4(q)  &\forall q \in \rH^1_{\Sigma}(\Omega),\\
  \nonumber
& &  &   c_4(\psi,\delta\bu^{n+1}) & &- a_5(\delta\phi_f^{n+1},\psi)  &= &F_5(\psi) & \forall \psi\in \rL^2(\Omega).
\end{alignat}
Then the solution 
is updated 
\begin{gather*}
    c_p^{k+1} = c_p^k + \delta c_p^{k+1},\quad c_l^{k+1} = c_l^k + \delta c_l^{k+1},\quad  \bu^{k+1} = \bu^k + \delta\bu^{k+1}, \\  
    p^{k+1} = p^k + \delta p^{k+1}, \quad  \phi_f^{k+1} = \phi_f^k + \delta \phi_f^{k+1},
\end{gather*}
until either the increments or the residuals drop below a given tolerance. In order to specify these variational forms, as usual we use the directional derivatives of all nonlinearities at a particular trial solution in the direction of 
the increments. After rearranging terms and omitting the superscript $(\cdot)^{n+1}$ whenever clear from the context, the bilinear forms and linear functionals in \eqref{eq:weak-linear} are defined as
 \begin{align}
   a_1(\delta c_p,w_p) & := \int_{\Omega} J^k\biggl(\frac{2\phi_f^k -\phi_f^{n}}{\Delta t} - \phi_f^k \frac{\partial r_p(c_p^k,c_l^k)}{\partial c_p^k} \biggr) \delta c_pw_p + \int_{\Omega} J^k \bF^{k,-1} \bD_p\bF^{k,-\tt t} \nabla\delta c_p \cdot \nabla w_p, \nonumber\\
  b_1^*(\delta c_l,w_p) &:= - \int_{\Omega} J^k  \frac{\partial r_p(\phi_f^k, c_p^k,c_l^k)}{\partial c_l^k}\delta c_l\,w_p,\quad F_1(w_p) := \int_{\Omega} \cR_1^{k,n} w_p, \quad F_2(w_l) := \int_{\Omega} \cR_2^{k,n} w_l,
 \nonumber \\
 b_2^*(\delta\bu,w_p) & :=\int_{\Omega}\! J^k\bF^{k,-1}\bigl([\bF^{k,-\tt t}\!:\!{\bnabla}\delta\bu]\bD_p
\! -\! {\bnabla}\delta\bu \bF^{k,-1} \bD_p - \bD_p \bF^{k,-\tt t} ({\bnabla}\delta\bu)^{\tt t}\bigr) \bF^{k,-\tt t} \nabla c^k_p \cdot \nabla w_p
\nonumber\\
& \qquad + 
 \int_{\Omega} J^k\bF^{k,-\tt t}:{\bnabla}\delta\bu \biggl(\frac{\phi_f^k -\phi_f^{n}}{\Delta t} c_p^k + \frac{c_p^k -c_p^{n}}{\Delta t}\phi_f^k-  r_p(\phi_f^k,c_p^k,c_l^k) \biggr)\,w_p,\nonumber\\
 b_3^*(\delta\phi_f,w_p) &:= \int_{\Omega}\! J^k\biggl(\frac{2c^k_p-c^{n}_p}{\Delta t}  - \frac{\partial r_p(\phi_f^k,c_p^k,c_l^k)}{\phi_f^k} \biggr)\delta\phi_f w_p ,
 \quad F_3(w_l) := \int_{\Omega} \bcR_{3,\Omega}^{k}\cdot \bv + \int_{\Sigma} \bcR_{3,\Sigma}^{k}\cdot \bv,\nonumber \\
   a_2(\delta c_l,w_l) & := \int_{\Omega} J^k\biggl(\frac{2\phi_f^k -\phi_f^{n}}{\Delta t} - \frac{\partial r_l(\phi_f^k,c_p^k,c_l^k)}{\partial c_l^k} \biggr) \delta c_lw_l + \int_{\Omega} J^k \bF^{k,-1} \bD_l\bF^{k,-\tt t} \nabla\delta c_l \cdot \nabla w_l\nonumber\\
   &\qquad - \int_{\Omega} \chi J^k\delta c_l \bF^{k,-1}\bF^{k,-\tt t} \nabla c_p^k\cdot\nabla w_l, \nonumber \\
   b_1(w_l,\delta c_p) &:= - \int_{\Omega} \chi J^k c_l^k\bF^{k,-1}\bF^{k,-\tt t} \nabla \delta c_p\cdot\nabla w_l
   - \int_{\Omega} J^k \frac{\partial r_l(\phi_f^k, c_p^k,c_l^k)}{\partial c_p^k}  \delta c_pw_l, \nonumber \\
   b^*_4(\delta \bu, w_l) & := \int_{\Omega} J^k\bF^{k,-1}\bigl([\bF^{k,-\tt t}:{\bnabla}\delta\bu]\bD_l
 - {\bnabla}\delta\bu \bF^{k,-1} \bD_l - \bD_l \bF^{k,-\tt t} ({\bnabla}\delta\bu)^{\tt t}\bigr) \bF^{k,-\tt t} \nabla c^k_l \cdot \nabla w_l
\nonumber\\
& \quad + \int_{\Omega} J^k\bF^{k,-\tt t}:{\bnabla}\delta\bu \biggl(\frac{\phi_f^k -\phi_f^{n}}{\Delta t} c_l^k + \frac{c_l^k -c_l^{n}}{\Delta t}\phi_f^k-  r_l(\phi_f^k,c_p^k,c_l^k) \biggr)\,w_l\nonumber\\
& \quad + \int_{\Omega} \chi c_l^k J^k \biggl( \bF^{k,-1} {\bnabla}\delta\bu \bF^{k,-1} + \bF^{k,-1} \bF^{k,-\tt t}({\bnabla}\delta\bu)^{\tt t}- \bF^{k,-\tt t}:{\bnabla}\delta\bu \bF^{k,-1}\biggr) \nonumber\\
&\quad \qquad \times  \bF^{k,-\tt t} \nabla c^k_p \cdot \nabla w_l,\nonumber \\
b^*_5(\delta \phi_f, w_l) & := \int_{\Omega} J^k \biggl(\frac{2c^k_l-c^{n}_l}{\Delta t} - \frac{\partial r_l(\phi_f^k,c_p^k,c_l^k)}{\phi_f^k} \biggr) \delta\phi_f w_l, \label{def:forms} \quad c_1^*(\delta p, \bv) = -  \int_{\Omega} \alpha \delta p J^k \bF^{k,-\tt t}:{\bnabla}\bv, \\
a_3(\delta \bu, \bv) & := \int_{\Omega} \biggl(\frac{\partial \bP_{\mathrm{eff}}(\bF^k)}{\partial \bF^k} {\bnabla}\delta\bu
+ \alpha p^k J^k [\bF^{k,-\tt t}:{\bnabla}\delta\bu\,\bI - \bF^{k,-\tt t}({\bnabla}\delta\bu)^{\tt t}]\bF^{k,-\tt t}\biggr):{\bnabla}\bv \nonumber \\
& \qquad -\int_\Sigma J^k [\bF^{k,-\tt t}:{\bnabla}\delta\bu\,\bI - \bF^{k,-\tt t}({\bnabla}\delta\bu)^{\tt t}]\bF^{k,-\tt t}\bt_\Sigma\cdot\bv,\nonumber\\
c_1(q,\delta \bu) & := \int_{\Omega} {\phi_f^k}J^k\bF^{k,-1}\biggl(\bF^{k,-\tt t}:{\bnabla}\delta\bu\frac{\bkappa(\phi_f^k)}{\mu_f} 
- {\bnabla}\delta\bu \bF^{k,-1}\frac{\bkappa(\phi_f^k)}{\mu_f} - \frac{\bkappa(\phi_f^k)}{\mu_f} \bF^{k,-\tt t} ({\bnabla}\delta\bu)^{\tt t}\biggr) \nonumber \\
&\qquad \qquad \times \bF^{k,-\tt t} \nabla p^k \cdot\nabla q  + \int_{\Omega}\rho_f J^k \bF^{k,-\tt t}:{\bnabla}\delta\bu\, \biggl(\frac{\phi_f^k-\phi_f^n}{\Delta t} - \ell(p^k,c_p^k)\biggr) q \nonumber \\
a_4(\delta p,q) & := \int_{\Omega}\! {\phi_f^k}J^k \bF^{k,-1} \frac{\bkappa(\phi_f^k)}{\mu_f} \bF^{k,-\tt t} \nabla \delta p \cdot\nabla q
- \!\!\int_{\Omega}\!\rho_f J^k \frac{\partial \ell(p^k,c_p^k)}{\partial p^k} \delta p\, q,\  a_5(\delta \phi_f,\psi)  := \int_{\Omega}\! \delta\phi_f \psi,   \nonumber\\
c_2(q,\delta c_p) & := - \int_{\Omega} \rho_fJ^k \frac{\partial \ell(p^k,c_p^k)}{\partial c_p^k} \delta c_p\, q, \quad F_4(q) := \int_{\Omega} \cR_4^{k,n} q,\quad F_5(\psi) := \int_{\Omega} \cR_5^{k} \psi,\nonumber\\
c_3^*(\delta \phi_f,q) & := \int_{\Omega}\! \frac{J^k}{\mu_f} \bF^{k,-1} \frac{\partial {[\bkappa(\phi_f^k)\cdot \phi_f^k ]}}{\partial \phi_f^k}\delta \phi_f \bF^{k,-\tt t} \nabla p^k \cdot\nabla q + \int_{\Omega} \!\!\rho_f J^k \frac{\delta\phi_f}{\Delta t} q, \nonumber\\
 c_4(\psi,\delta\bu)&  := \int_{\Omega} J^k \bF^{k,-\tt t}:{\bnabla}\delta\bu\, \psi,\nonumber
\end{align} 
where  $\bF^k=\bI + {\bnabla}\bu^k$, $J^{k} = \det\bF^k$, $\bF^{k,-1} = (\bF^k)^{-1}$, $\bF^{k,-\tt t} = (\bF^k)^{-\tt t}$,  and $\bcR_{3,\Omega}^{k},\bcR_{3,\Sigma}^k, \cR_5^{k}$ are vector and scalar residuals coming from the linearisation of \eqref{eq:u} and \eqref{eq:phi} and containing terms associated to the previous Newton step $k$ as well as body and traction loads. The terms $\cR_1^{k,n},\cR_2^{k,n}, \cR_4^{k,4}$ are scalar residuals from the linearisation of \eqref{eq:cp}, \eqref{eq:cl}, and \eqref{eq:p}, respectively; which in addition contain contributions from the previous time step $n$. 
\section{Stability and solvability of the linearised problem} \label{sec:stability}
\subsection{Definition and preliminaries}
For this section we draw inspiration from the stability analysis 
of a system similar to \eqref{eq:poroe}, recently 
proved in \cite{barnafi20}; and from the study of the linearised hyperelasticity equations, 
recently studied in \cite{farrell21}. 
We restrict the description to Neo-Hookean poroelastic solids with material parameters $\mu_s,\lambda_s$, for which the effective first Piola-Kirchhoff stress tensor is
\[{\bP_{\mathrm{eff}}= \mu_s(\bF - \bF^{- \tt t}) + \lambda_s \ln(J)\, \bF^{-\tt t}}.\] 
Note that a simplified version of the time-discrete tangent system \eqref{eq:weak-linear} is readily obtained by focusing on the first Newton-Raphson iterate when starting from 
{the following} initial guess 
\[{ \bu^k = \cero, \quad p^k = p_0, \quad  
\phi_f^k = \phi_0, \quad  c_p^k = c_l^k = 0,}
\] 
for displacement, fluid pressure, nominal porosity and species concentrations, which gives $D_t \phi_f^k = 0$, $\bF^k =\bI$, $J^k = 1$, $\nabla p^k = \cero$.   In addition, the residuals can be conveniently rearranged and we rescale the system in an appropriate manner using the parameters $\rho_f$, $\Delta t$, $\alpha$, $\mu_s$, $\phi_0$, and $\lambda_s$. Next we introduce the total pressure 
\begin{equation}\label{eq:total-pressure}
    \zeta := \alpha p - \lambda_s \vdiv \bu,
\end{equation}
and the system can be recast in terms of $\bu,\zeta,p,c_p,c_l$ (that is, using $\zeta$ instead of $\phi_f$), as follows 
\begin{alignat}{5}
&  \tilde{a}_1(c_p,w_p)+ a_1(c_p,w_p)   &   & & + \beta_1 b^*_3(p,w_p) &+ \beta_2 b^*_3(\zeta,w_p) &=& \tilde{F}_1(w_p)\ &\forall w_p \in \rH^1(\Omega),\nonumber\\
&   & \tilde{a}_2( c_l,w_l)+a_2( c_l,w_l)  &  & + \beta_1 b_5^*(p,w_l) &+ \beta_2 b_5^*(\zeta,w_l) &=& \tilde{F}_2(w_l) &\forall w_l \in \rH^1(\Omega),\nonumber\\
\label{eq:weak-linear-simple}&&&a_3(\bu,\bv) & &+ b_1(\zeta,\bv)   &=& F_3(\bv)  &\forall \bv \in\bH^1_{\Gamma}(\Omega),\\
  \nonumber
&b_4(c_p,q)&&  & \tilde{a}_4(p,q) + a_4(p,q) &-\frac{1}{\alpha} b_2(\delta_t \zeta,q)  &= &\tilde{F}_4(q)  &\forall q \in \rH^1_\Sigma(\Omega),\\
  \nonumber
&&& b_1(\psi,\bu) &+ b_2(\psi, p) &- a_5(\zeta,\psi)  &= &0 & \forall \psi\in \rL^2(\Omega),
\end{alignat}
where for a function scalar $X^{n+1}$, the expression $\delta_t X^{n+1}$ stands for $\frac{X^{n+1}-X^n}{\Delta t}$, and where
we have renamed the bilinear forms, which are now defined as 
\begin{gather}
   \tilde{a}_1(c_p,w_p)  = \int_{\Omega} \delta_t c_pw_p,\;\; a_1(c_p,w_p)= \int_{\Omega} \bD_p \nabla  c_p \cdot \nabla w_p, \;\; b_1(\psi,\bv) = -\int_{\Omega}\vdiv\bv\, \psi,  \;\; b_3^*(\zeta,w_p) = -\frac{1}{\Delta t} \int_{\Omega}\zeta w_p ,
 \nonumber \\
 \label{def:forms-simple}  
   \tilde{a}_2(c_l,w_l)  = \int_{\Omega}  \delta_t  c_lw_l, \quad a_2(c_l,w_l)  = \int_{\Omega} \bD_l  \nabla c_l \cdot \nabla w_l,\quad a_5(\zeta,\psi)  = \frac{1}{\lambda_s}\int_{\Omega} \zeta \psi, \quad \tilde{F}_4(q)  =  \int_{\Omega} \bigl(\frac{1}{\Delta t}\tilde{R}_4^n  - R_4\bigr) q,\nonumber\\
b^*_5(\zeta, w_l)  = -\frac{1}{\Delta t} \int_{\Omega}\zeta w_l,\quad b_2(\psi, q) = \frac{\alpha}{\lambda_s} \int_{\Omega}   q \psi, \quad
b_4(w_p,q) = \frac{-1}{\Delta t}\int_{\Omega} w_p q, \quad a_3(\bu, \bv)  = 2\mu_s \int_{\Omega}  {\beps}( \bu)
:{\beps}(\bv), \nonumber \\
\tilde{a}_4(  p,q)  = \bigl(1+ \frac{\alpha}{\lambda_s}\bigr) \int_{\Omega}\delta_t pq, \;\; a_4( p,q)  =     \int_{\Omega}  \frac{\bkappa}{\mu_f}   \nabla  p \cdot\nabla q
, \;\;  
c_3^*(\zeta,q)  = -\frac{\rho_f}{\Delta t}\int_{\Omega}\zeta q, 
\;\;  F_3(\bv) = \int_\Omega (\rho_s\bb +\mathbf{R}_3)\cdot \bv, \nonumber\\
    \tilde{F}_1(w_p) = \int_{\Omega} (\frac{1}{\Delta t}\tilde{R}_1^n + R_1) w_p, \quad \tilde{F}_2(w_l) = \int_{\Omega} (\frac{1}{\Delta t}\tilde{R}_2^n + R_2) w_l.
\end{gather} 
Here we have also set $\bt_\Sigma = \cero$. 
Note that the linear variational problem \eqref{eq:weak-linear-simple} corresponds to the linear and 
coupled reaction-diffusion / three-field Biot equations 
\begin{align}
\delta_t c_p-\frac{1}{\Delta t}(\beta_1 p+ \beta_2\zeta) - \vdiv(\bD_p \nabla c_p) & = \frac{1}{\Delta t}\tilde{R}_1^n + R_1,\nonumber\\
\delta_t c_l-\frac{1}{\Delta t}(\beta_1 p+ \beta_2\zeta) - \vdiv(\bD_l \nabla c_l) & = \frac{1}{\Delta t}\tilde{R}_2^n + R_2,\nonumber\\ 
-\bdiv(2\mu_s \beps(\bu) - \zeta \bI ) & = \rho_s\bb + \mathbf{R}_3, \label{eq:strong-linear}\\
-\frac{1}{\Delta t} c_p + \Big(1+\frac{\alpha}{ \lambda_s}\Big) \delta_t p - \frac{1}{\lambda_s}\delta_t \zeta - \vdiv \biggl(\frac{\bkappa}{\mu_f} \nabla p\biggr) & = \frac{1}{\Delta t}\tilde{R}_4^n  - R_4,\nonumber\\
\frac{\alpha}{\lambda_s} p  - \frac{1}{\lambda_s} \zeta - \vdiv\bu & = 0,\nonumber
\end{align} 
where the $R_i$'s and $\tilde{R}_i$'s are scalar and vector residuals, and the $\beta_i>0$ are additional constants arising from the linearisation of the reaction terms and from the change of variables \eqref{eq:total-pressure}.   
A variant of system \eqref{eq:strong-linear} including advection and nonlinear reaction, has been studied in \cite{verma21}. We follow that reference and in the remainder of this section we present the solvability, stability, and convergence analysis for the in-time problem associated with \eqref{eq:weak-linear-simple}. 

We will assume that the anisotropic permeability $\boldsymbol{\kappa}$ and the diffusion matrices $D_p,D_l$
are uniformly bounded and positive definite in $\Omega$. The latter means that, there exist positive constants $\kappa_1, \kappa_2$, and $D_i^{\min}, D_i^{\max}$, $i\in \{p,l\}$, such that
$$\kappa_1|\bw|^2 \leq \bw^{\mathrm{t}}\kappa\bw \leq \kappa_2|\bw|^2,\quad \mathrm{and}\quad D_i^{\min}|\bw|^2 \leq \bw^{\mathrm{t}}D_i\bw \leq D_i^{\max} |\bw|^2  \quad \forall\bw \in\mathbb{R}^d.$$
{We also recall the following version of Korn's inequality 
$$C_{k,1}\Vert\bu\Vert^2_{1,\Omega}\leq \Vert\boldsymbol{\varepsilon}(\bu)\Vert^2_{0,\Omega},$$
where $C_{k,1}$ is a positive constant (see, e.g., \cite{brenner}).  
This bound, together with the assumptions above on the permeability and diffusivity, imply the} following coercivity and positivity properties
\begin{gather}
a_1(w_p,w_p) \ge D_p^{\min} |w_p|^2_{1,\Omega}, \quad a_2(w_l,w_l) \ge D_l^{\min} |w_l|^2_{1,\Omega},  \\
a_3(\bv,\bv)  \ge 2 \mu_s C_{k,1}\|\bv\|_{1,\Omega}^2 , \quad a_4(q,q) \ge \frac{\kappa_1}{\mu_f} \|q^f\|_{1,\Omega}^2, \quad
a_5(\zeta,\zeta)  = \lambda_s^{-1}\|\zeta\|_{0,\Omega}^2,\nonumber   \label{elip_a_1}
\end{gather}
for all $\bv \in \mathrm{\bH}^1_\Gamma(\Omega)$, $\zeta \in \mathrm{\mathrm{L}}^2(\Omega)$, $w_p,w_l \in \mathrm{H}^1(\Omega)$, $q\in \mathrm{H}^1_\Sigma(\Omega)$. Moreover, {thanks to the structure of the formulation, the bilinear form $b_1$ (which coincides with the usual non-diagonal bilinear form in the Stokes equations) satisfies the following} inf-sup condition (see, e.g.,  \cite{girault79}):
For every $\zeta \in \mathrm{L}^2(\Omega)$, there exists $\beta >0$ \ such that
\begin{equation}\label{inf-sup}
\sup_{\bv \in \bH^1_\Gamma(\Omega)} \frac{b_1(\bv,\zeta)}{\|\bv\|_{1,\Omega}}  \ge \beta \|\zeta\|_{0,\Omega}.
\end{equation} Finally, we recall an important {discrete-in-time} identity and introduce the discrete-in-time norm
\begin{equation*}
\int_\Omega X^{n+1}\delta_t X^{n+1} = \frac{1}{2}\delta_t \Vert X^{n+1} \Vert{^2} +\frac{1}{2}\Delta t\Vert\delta_t X^{n+1}\Vert{^2},\qquad \Vert X\Vert^2_{\ell^2(V)}:=\Delta t\sum_{m=0}^{n}\Vert X^{m+1}\Vert_V^2,
\end{equation*}
respectively, which will be useful for the subsequent analysis.

\subsection{Fixed-point scheme and unique solvability of the decoupled problems} With the aim to facilitate the comprehension and clarity of the forthcoming analysis, through this section and Sections \ref{sec:coupled} - \ref{sec:stab}, we take up again the notation $(\cdot)^{n+1}$ given in Section \ref{sec:linearisation} to denote the current time step of the involved variables. Moreover, in order to apply a fixed-point scheme, we modify the variational formulation \eqref{eq:weak-linear-simple} in an equivalent one: find $(c^{n+1}_p,c^{n+1}_l,\bu^{n+1}, p^{n+1}, \zeta^{n+1})\in \mathrm{H}^1(\Omega)\times \mathrm{H}^1(\Omega)\times \mathbf{H}_\Gamma^1(\Omega)\times\mathrm{H}_\Sigma^1(\Omega)\times \mathrm{L}^2(\Omega)$ such that 
\begin{subequations}
\begin{alignat}{5}
&  \tilde{a}_1(c^{n+1}_p,w_p)+ a_1(c^{n+1}_p,w_p)   &   & & &&=& F_{1,p^{n+1},\zeta^{n+1}}(w_p), \label{eq:DR_1}\\
&    \tilde{a}_2( c^{n+1}_l,w_l)+a_2( c^{n+1}_l,w_l) & &  & & &=& F_{2,p^{n+1},\zeta^{n+1}}(w_l), \label{eq:DR_2}\\
&&&a_3(\bu^{n+1},\bv) & &+ b_1(\zeta^{n+1},\bv)   &=& F_3(\bv) , \label{eq:poro_1}\\
&&&  & \tilde{a}_4(p^{n+1},q) + a_4(p^{n+1},q) &-\frac{1}{\alpha} b_2(\delta_t \zeta^{n+1},q)  &= &F_{4,c^{n+1}_p}(q)  \label{eq:poro_2},\\
&&& b_1(\psi,\bu^{n+1}) &+ b_2(\psi, p^{n+1}) &- a_5(\zeta^{n+1},\psi)  &= &0,\label{eq:poro_3}
\end{alignat}
for each $(w_p,w_l,\bv,q,\psi)\in \mathrm{H}^1(\Omega)\times \mathrm{H}^1(\Omega)\times\mathbf{H}^1_\Gamma(\Omega)\times \mathrm{H}^1_\Sigma(\Omega)\times \mathrm{L}^2(\Omega)$, where the functionals $F_{1,p^{n+1},\zeta^{n+1}}$, $F_{2,p^{n+1},\zeta^{n+1}}$, $F_{4,c_p^{n+1}}$ are defined as
\end{subequations}
\begin{equation*}
	F_{1,q,\psi}(w_p):=\tilde{F}_1(w_p)-\beta_1b^*_3(q,w_p)-\beta_2b^*_3(\psi,w_p),\quad F_{2,q,\psi}(w_l):=\tilde{F}_2(w_l)-\beta_1b^*_3(q,w_l)-\beta_2b^*_3(\psi,w_l),
\end{equation*}
\begin{equation}\label{eq:functionals}
F_{4,w_p}(q):=\tilde{F}_4(q)-b_4(w_p,q),\end{equation}
respectively. 

We then start with the fixed-point strategy. Let us define the operator $\mathbf{S}:\mathrm{H}^1(\Omega)\rightarrow \mathrm{H}^1_\Sigma(\Omega)\times \mathrm{L}^2(\Omega)$ as
\begin{equation}\label{eq:def_S}
\mathbf{S}(c^{n+1}_p):=(\mathbf{S}_1(c^{n+1}_p),\mathbf{S}_2(c^{n+1}_p)):=(p^{n+1},\zeta^{n+1})\quad \forall\,c^{n+1}_p\in \mathrm{H}^1(\Omega),
\end{equation}
where $p^{n+1}\in \mathrm{H}^1_\Sigma(\Omega)$ and $\zeta^{n+1}\in \mathrm{L}^2(\Omega)$ are the second and third components of the solution of problem: Find $(\bu^{n+1},p^{n+1},\zeta^{n+1})\in \mathbf{H}^1_\Gamma(\Omega)\times\mathrm{H}^1_\Sigma(\Omega)\times \mathrm{L}^2(\Omega) $ such that 
\begin{subequations}
\begin{alignat}{5}
&&&a_3(\bu^{n+1},\bv) & &+ b_1(\zeta^{n+1},\bv)   &\;=\;& F_3(\bv)  &\quad\forall \bv \in\bH^1_{\Gamma}(\Omega),\label{eq:uncoupled1} \\
\label{eq:uncoupled2}
&&&  & \tilde{a}_4(p^{n+1},q) + a_4(p^{n+1},q) &-\frac{1}{\alpha} b_2(\delta_t \zeta^{n+1},q)  &\;=\; &F_{4,{c}^{n+1}_p}(q)  &\quad\forall q \in \rH^1_\Sigma(\Omega),\\
&&& b_1(\psi,\bu^{n+1}) &+ b_2(\psi, p^{n+1}) &- a_5(\zeta^{n+1},\psi)  &\;=\; &0 & \quad\forall \psi\in \rL^2(\Omega),\label{eq:uncoupled3}
\end{alignat} 
\end{subequations}
for a given $c^{n+1}_p\in \mathrm{H}^1(\Omega)$. In turn, let $\tilde{\mathbf{S}}:\mathrm{H}^1_\Sigma(\Omega)\times \mathrm{L}^2(\Omega)\rightarrow\mathrm{H}^1(\Omega)$ be the operator defined by 
\begin{equation}\label{eq:def_TS}
\tilde{\mathbf{S}}(p^{n+1},\zeta^{n+1}):=c_p^{n+1}\quad \forall\,(p^{n+1},\zeta^{n+1})\in \mathrm{H}^1_\Sigma(\Omega)\times \mathrm{L}^2(\Omega),
\end{equation}
where $c_p^{n+1}\in\mathrm{H}^1(\Omega)$ is the first component of the solution of the problem: Find $(c^{n+1}_p,c^{n+1}_l)\in [\mathrm{H}^1(\Omega)]^2$ such that 
\begin{subequations}
\begin{alignat}{5}
\tilde{a}_1(c^{n+1}_p,w_p) &\;+ &\; a_1(c^{n+1}_p,w_p) &  &  &  = &\; F_{1,{p}^{n+1},{\zeta}^{n+1}}(w_p) &\quad\forall w_p \in \mathrm{H}^1(\Omega), \label{eq:DR_1_FP}\\
\tilde{a}_2(c^{n+1}_l,w_l) &\;+ &\; a_2(c^{n+1}_l,w_l) &  &  &  = &\; F_{2,{p}^{n+1},{\zeta}^{n+1}}(w_l) &\quad\forall w_l \in \mathrm{H}^1(\Omega) \label{eq:DR_2_FP}, 
\end{alignat}
\end{subequations} for a given pair $(p^{n+1},\zeta^{n+1})\in \mathrm{H}^1_\Sigma(\Omega)\times \mathrm{L}^2(\Omega)$. Therefore, we define the map $\mathbf{T}:\mathrm{H}^1(\Omega)\rightarrow \mathrm{H}^1(\Omega)$ as
\begin{equation}\label{eq:def_T}
	\mathbf{T}(c^{n+1}_p):=\tilde{\mathbf{S}}(\mathbf{S}_1(c^{n+1}_p),\mathbf{S}_2(c^{n+1}_p))\quad \forall\,c_p^{n+1}\in \mathrm{H}^1(\Omega),
\end{equation}
and one readily realises that solving \eqref{eq:DR_1} - \eqref{eq:poro_3} is equivalent to seeking a fixed point of the solution operator $\mathbf{T}$.
 
Now, in order to prove that the operator $\mathbf{T}$ is well-defined, we first need to investigate that the uncoupled problems defined by the operators $\mathbf{S}$ and $\mathbf{\tilde{S}}$ are well-posed. The fact that the poroelastic problem is well-posed for a given ${c}^{n+1}_p$ is given next. Its proof can be deduced by employing the Fredholm  alternative approach and classical tools frequently used for showing the well-posedness of elliptic/parabolic equations. We refer to \cite[Lemmas 2.1, 2.2 and 2.3]{verma21} for further details.
\begin{lemma}\label{eq:stability-poro}
For a given ${c}^{n+1}_p\in \mathrm{H}^1(\Omega)$, the problem \eqref{eq:uncoupled1} - \eqref{eq:uncoupled3} has a unique solution $(\bu^{n+1}, p^{n+1}, \zeta^{n+1})$ $\in \mathbf{H}^1_\Gamma(\Omega)\times \mathrm{H}_\Sigma^1(\Omega)\times \mathrm{L}^2(\Omega)$. Moreover, there exists $C>0$, {independent of $\lambda_s$}, such that for each $n$,
\begin{align}\label{eq:CD1}
\begin{split}
&\Vert \bu^{n+1}\Vert^2_{1,\Omega} + \Vert p^{n+1}\Vert^2_{0,\Omega} + \Vert \zeta^{n+1}\Vert^2_{0,\Omega} + \Vert p\Vert^2_{\ell^2(H^1(\Omega))}\\
&\qquad  \leq {\exp}(C) \Big\{\Vert\bu^{0}\Vert^2_{1,\Omega} + \Vert p^{0}\Vert_{0,\Omega}^2+ \Vert \zeta^{0}\Vert^2_{0,\Omega}+ \Vert\mathbf{R}_3\Vert^2_{0,\Omega}+  \Vert\bb\Vert_{0,\Omega}^2 + \Delta t\Vert R_4\Vert^2_{0,\Omega}\Big.\\
&\qquad \qquad \qquad \quad \Big.+\sum_{m=0}^{n}\Vert \tilde{R}^m_4\Vert^2_{0,\Omega}+\sum_{m=0}^{n}\Vert {c}^{m+1}_{p}\Vert^2_{0,\Omega} \Big\}.
\end{split}
\end{align}
\end{lemma}
In turn, the well-posedness of the uncoupled diffusion-reaction problem is given next.
\begin{lemma} \label{uncoupled-ADR}
	For fixed ${p}^{n+1} \in \mathrm{H}^1_\Sigma(\Omega)$ and ${\zeta}^{n+1}\in \mathrm{L}^2(\Omega)$, the system of nonlinear parabolic equations \eqref{eq:DR_1_FP} - \eqref{eq:DR_2_FP}
	has a unique solution $(c_p^{n+1},c^{n+1}_l)\in [\mathrm{H}^1(\Omega)]^2$, that is continuously dependent on data, that is, there exists 
	$C > 0$, such that for each $n$, 
	\begin{align}
	&\Vert c^{n+1}_p\Vert^2_{0,\Omega}+\Vert c^{n+1}_l\Vert^2_{0,\Omega}+ \Vert \nabla c_p \Vert^2_{\ell^2(L^2(\Omega))}+ \Vert \nabla c_l\Vert^2_{\ell^2(L^2(\Omega))}\nonumber\\
	&\leq \exp(C)\Big\{\Vert c_p^{0}\Vert^2_{0,\Omega} + \Vert c_l^0\Vert^2_{0,\Omega}+\sum_{m=0}^{n}(\Vert \tilde{R}^m_1\Vert^2_{0,\Omega}+\Vert \tilde{R}^m_2\Vert^2_{0,\Omega})+ \Delta t\Vert R_1\Vert^2_{0,\Omega}+\Delta t\Vert R_2\Vert^2_{0,\Omega}\label{bound-w1h-w2h-stability}\\
	&\qquad \qquad \quad + \sum_{m=0}^{n}(\Vert{\zeta}^{m+1}\Vert^2_{0,\Omega}+\Vert{p}^{m+1}\Vert^2_{0,\Omega})\Big \}.\nonumber
	\end{align}
\end{lemma}
\begin{proof}
	For each n, the unique solvability of \eqref{eq:DR_1_FP} - \eqref{eq:DR_2_FP} can be obtained by using the coercivity of the diffusivities $\mathbf{D}_{i} \;\; i\in \{p,l\}$, and classical tools for time-discrete approximation schemes of parabolic equations (see for instance \cite{quarteroni94}). Now, for the stability estimation, we begin by obtaining the corresponding result for $c_p^{n+1}$. Consider $w_{p}=c_{p}^{n+1}$ in \eqref{eq:DR_1_FP}, to get
	\begin{equation*}
	\int_\Omega  \delta_t  c_{p}^{n+1}c_{p}^{n+1} + \int_\Omega \mathbf{D}_p\nabla c_{p}^{n+1}\cdot \nabla c_{p}^{n+1} =\int_{\Omega} (\frac{1}{\Delta t}\tilde{R}_1^n + R_1) c_p^{n+1}+ \beta_1\frac{1}{\Delta t} \int_{\Omega}{\zeta}^{n+1} c_p^{n+1} + \beta_2\frac{1}{\Delta t} \int_{\Omega}{p}^{n+1} c_p^{n+1} ,
	\end{equation*}
	and then, applying the property \eqref{elip_a_1}, and the classical Cauchy-Schwarz inequality, we obtain
	\begin{align*}
	\begin{split}
	&\frac{1}{2}\delta_t\Vert c^{n+1}_p\Vert^2_{0,\Omega}+ \frac{1}{2}\Delta t \Vert \delta_t c^{n+1}_p\Vert^2_{0,\Omega}+  D_p^{\min} \Vert \nabla c^{n+1}_p \Vert^2_{0,\Omega}\\
	&\quad \qquad \leq\frac{1}{\Delta t}\Vert \tilde{R}^n_1\Vert_{0,\Omega}\Vert c^{n+1}_p \Vert_{0,\Omega} + \Vert R_1\Vert_{0,\Omega}\Vert c^{n+1}_p \Vert_{0,\Omega} + \frac{1}{\Delta t}\Vert\beta_1{\zeta}^{n+1} +\beta_2{p}^{n+1}\Vert_{0,\Omega}\Vert c_p^{n+1}\Vert_{0,\Omega}.
	\end{split}
	\end{align*}
	Applying Young's inequality, it is possible to get 
	\begin{align*}
	\begin{split}
	&\frac{1}{2}\delta_t\Vert c^{n+1}_p\Vert^2_{0,\Omega}+ \frac{1}{2}\Delta t \Vert \delta_t c^{n+1}_p\Vert^2_{0,\Omega}+ D_p^{\min}\Vert \nabla c^{n+1}_p\Vert^2_{0,\Omega}\\
	&\qquad\qquad\qquad \leq \frac{1}{2\Delta t}\Vert \tilde{R}^n_1\Vert^2_{0,\Omega}+(\frac{1}{\Delta t}+\frac{1}{2})\Vert c^{n+1}_p \Vert_{0,\Omega}^2 + \frac{1}{2}\Vert R_1\Vert^2_{0,\Omega}+ \frac{\beta_1}{2\Delta t}\Vert{\zeta}^{n+1}\Vert^2_{0,\Omega} +\frac{\beta_2}{2\Delta t}\Vert{p}^{n+1}\Vert_{0,\Omega}^2,
	\end{split}
	\end{align*}
	and summing over $n-1$ and multiplying by $\Delta t$, we finally deduce that 
	\begin{align}\label{bound-cp}
	\begin{split}
	&\frac{1}{2}\Vert c_p^{n}\Vert^2_{0,\Omega}+ \frac{1}{2}\Delta t^2\sum_{m=0}^{n-1} \Vert \delta_t c_p^{m+1}\Vert^2_{0,\Omega}+ D_p^{\min}\Delta t \sum_{m=0}^{n-1}\Vert \nabla c_p^{m+1}\Vert^2_{0,\Omega}\\
	&\quad \qquad  \leq \Vert c_p^{0}\Vert_{0,\Omega} +\frac{1}{2}\sum_{m=0}^{n-1}\Vert \tilde{R}^m_1\Vert^2_{0,\Omega}+(1+\frac{\Delta t}{2}) \sum_{m=0}^{n-1} \Vert c_p^{m+1} \Vert_{0,\Omega}^2 + \frac{\Delta t}{2}\sum_{m=0}^{n-1}\Vert R_1\Vert^2_{0,\Omega}\\
	&\qquad \qquad + \frac{\beta_1}{2}\sum_{m=0}^{n-1}\Vert{\zeta}^{m+1}\Vert^2_{0,\Omega} +\frac{\beta_2}{2}\sum_{m=0}^{n-1}\Vert{p}^{m+1}\Vert_{0,\Omega}^2.
	\end{split}
	\end{align}
	The same result can be derived for $c_l^{n}$. Finally, the estimate \eqref{bound-w1h-w2h-stability} follows by adding the aforementioned result to \eqref{bound-cp} and then, applying Gronwall's inequality, and using the resulting approach for $n+1$.
\end{proof}
\subsection{Existence of a weak solution to the coupled linearised system}\label{sec:coupled}
 Let us define a closed subset of the Banach space $\mathrm{L}^2(\Omega)$ as 
 \begin{equation}\label{eq:def_K_1}
 \mathcal{K}:= \lbrace c^{n+1}_p \in \mathrm{L}^2(\Omega): \quad \Vert c^{n+1}_p\Vert_{1,\Omega}\leq r\rbrace,
 \end{equation}
 for a given $r>0$. We then restrict the space in \eqref{eq:def_T} to the ball \eqref{eq:def_K_1}, and therefore, we look for a $c_p^{n+1}\in \mathcal{K}$ such that $\mathbf{T}({c}^{n+1}_p) = c^{n+1}_p$. In what follows, we prove the existence of that $c_p^{n+1}$ by means of the Schauder fixed-point approach. We start with a previous result.
 \begin{lemma}
For a given ${c}^{n+1}_p\in \mathrm{H}^1(\Omega)$, the problem defined  by the operator $\mathbf{S}$ (cf. \eqref{eq:def_S}) satisfies the following estimate 
\begin{align}\label{eq:SN}
\begin{split}
&\Vert\mathbf{S}(c_p^{n+1})\Vert_{\boldsymbol{0,\Omega}}\leq \Vert p^{n+1}\Vert_{1,\Omega} + \Vert \zeta^{n+1}\Vert_{0,\Omega}\\
&\qquad  \leq C \Big\{\Vert\mathbf{R}_3\Vert_{0,\Omega}+\Vert\bb^{m+1}\Vert_{0,\Omega} + \Vert R_4\Vert_{0,\Omega}+\Vert \tilde{R}^n_4\Vert_{0,\Omega}+\Vert\zeta^{n}\Vert_{0,\Omega}\Big\} + \tilde{C}\Vert {c}_{p}^{n+1}\Vert_{0,\Omega},
\end{split}
\end{align}
where $C$ and $\tilde{C}$ are positive constants depending on $\Delta t, \mu_f, \mu_s,\kappa_1, \beta, \alpha$,  with
 \begin{equation}\label{def_TC}
 \tilde{C}:=(3\alpha\mu_f((2\Delta t^2\kappa_1)\min\{\mu_s/\Delta t,\alpha \kappa_1/2\mu_f,\beta^2\})^{-1})^{1/2}.
 \end{equation} 
 \end{lemma}
\begin{proof}
	The proof starts by multiplying the equations \eqref{eq:uncoupled1}, \eqref{eq:uncoupled2} and \eqref{eq:uncoupled3} by $\frac{1}{\Delta t}, \alpha$, and $-\frac{1}{\Delta t}$, respectively. Then, the estimate \eqref{eq:SN} follows after adding these resulting equations with $\bv:=\bu^{n+1}$, $q:=p^{n+1}$ and $\psi:=\zeta^{n+1}$, respectively, using {Cauchy--Schwarz} and Young inequalities. We omit further details. 
\end{proof}

 Now, we are in a position to establish the existence of a fixed point of the operator $\mathbf{T}$. This result is abridged in the following two lemmas.
\begin{lemma}\label{eq:bound_T}
	Assume that $\tilde{C}\hat{C}\leq\frac{1}{4}$, and
	\begin{equation*}
		\hat{C}\left\{\Vert c_p^n\Vert_{0,\Omega}+\Vert \tilde{R}_1^n\Vert_{0,\Omega} + \Vert R_1\Vert_{0,\Omega}+ C \Big\{\Vert\mathbf{R}_3\Vert_{0,\Omega}+\Vert\bb^{m+1}\Vert_{0,\Omega} + \Vert R_4\Vert_{0,\Omega}+\Vert \tilde{R}^n_4\Vert_{0,\Omega}+\Vert\zeta^{n}\Vert_{0,\Omega}\Big\} \right\}\leq \frac{r}{2},
	\end{equation*} where {$\tilde{C}, \hat{C}$ and $C$ are the constants specified in \eqref{def_TC}, \eqref{hatC} and \eqref{eq:SN}, respectively.} Then, the fixed-point operator $\mathbf{T}$ maps from $\mathcal{K}$ into itself.
\end{lemma}
\begin{proof}
	We begin by obtaining an $\mathrm{H}^1(\Omega)$-norm for the problem defined by $\mathbf{\tilde{S}}$ (cf. \eqref{eq:def_TS}). In fact, by taking  $w_p=c_p^{n+1}$ in \eqref{eq:DR_1_FP}, we readily see that
	\begin{align*}
	\begin{split}
	&\frac{1}{\Delta t} \Vert c_{p}^{n+1}\Vert^2_{0,\Omega} + \mathbf{D}^{min}_p\Vert\nabla c_{p}^{n+1}\Vert^2_{0,\Omega}\\
	&\qquad \leq \frac{1}{\Delta t}\int_{\Omega}c_p^{n+1}c_p^n+\int_{\Omega} (\frac{1}{\Delta t}\tilde{R}_1^n + R_1) c^{n+1}_p+ \beta_1\frac{1}{\Delta t} \int_{\Omega}{\zeta}^{n+1} c^{n+1}_p + \beta_2\frac{1}{\Delta t} \int_{\Omega}{p}^{n+1} c^{n+1}_p,
	\end{split}
	\end{align*}
	for a given pair $({p}^{n+1},{\zeta}^{n+1})\in  \mathrm{H}^1_\Sigma(\Omega)\times \mathrm{L}^2(\Omega)$, from which, we straightforwardly obtain 
	\begin{equation}\label{eq:bound_H1}
\Vert \mathbf{\widetilde{S}}({p}^{n+1},{\zeta}^{n+1})\Vert_{1,\Omega}=\Vert c_{p}^{n+1}\Vert_{1,\Omega}\leq \hat{C}\left\{\Vert c_p^n\Vert_{0,\Omega}+\Vert \tilde{R}_1^n\Vert_{0,\Omega} + \Vert R_1\Vert_{0,\Omega}+\Vert p^{n+1}\Vert_{0,\Omega}+\Vert \zeta^{n+1}\Vert_{0,\Omega}\right\},
	\end{equation} 
	and therefore, by using the definition of $\mathbf{T}$ (cf. \eqref{eq:def_T}), and applying the result given by \eqref{eq:bound_H1}, we get 
	\begin{align}\label{eq:prel_NT}
	\begin{split}
		&\Vert\mathbf{T}(c_p^{n+1})\Vert_{{1,\Omega}}:=\Vert\mathbf{\widetilde{S}}(\mathbf{S}_1(c_p^{n+1}),\mathbf{S}_2(c_p^{n+1}))\Vert_{1,\Omega}\\
		&\quad \leq \hat{C}\left\{\Vert c_p^n\Vert_{0,\Omega}+\Vert \tilde{R}_1^n\Vert_{0,\Omega} + \Vert R_1\Vert_{0,\Omega}+\Vert \mathbf{S}_1(c_p^{n+1})\Vert_{0,\Omega}+\Vert \mathbf{S}_2(c_p^{n+1})\Vert_{0,\Omega}\right\},
		\end{split}
	\end{align}
	{where
	\begin{equation}\label{hatC}
	    \hat{C}:=(\Delta t)^{-1}\max\{1,(\Delta t)^{-1},\beta_1,\beta_2\}\cdot \min\{(\Delta t)^{-1},\mathbf{D}^{min}_p\}^{-1}.
	\end{equation}}
	Finally, the desired result follows after substituting the estimate \eqref{eq:SN} into \eqref{eq:prel_NT}, and then, applying the assumption on data given at the statement
	of the present lemma.
\end{proof}

{We remark that there is no inconvenience with the second assumption on data given in Lemma \ref{eq:bound_T} since it depends on a constant $r$ that can be appropriately chosen (cf. \eqref{eq:def_K_1}). Moreover, the first smallness assumption given in Lemma \ref{eq:bound_T} is merely a theoretical condition to guarantee solvability of the continuous problem. However,  even though for the constants $\tilde{C}$ and $\hat{C}$ defined by \eqref{def_TC} and \eqref{hatC}, respectively, the condition might not always be satisfied (especially for extreme values of the model constants), the numerical experiments obtained in Section \ref{sec:results} show satisfactory results even without the theoretical constraints. The present analysis could be improved by using a different approach that circumvents the current restrictions on data, but we are not aware of such a strategy being applicable directly to this context. 
} 

\begin{lemma}\label{cor1}
 	The map $\mathbf{T}:\mathcal{K}\rightarrow\mathcal{K}$ is {continuous} and $\mathbf{T}(\mathcal{K})$ is relatively compact  in $\mathrm{L}^2(\Omega)$.
\end{lemma}
\begin{proof}
We notice from the previous lemma, that $\mathbf{T}(\mathcal{K})$ is bounded in $\mathrm{H}^1(\Omega)$. In this way, the compact embedding of $\mathrm{H}^1(\Omega)$ into $\mathrm{L}^2(\Omega)$ together with boundedness of $\mathbf{T}(\mathcal{K})$  conclude that  $\mathbf{T}(\mathcal{K})$ is relatively compact in $\mathrm{L}^2(\Omega)$. On the other hand, for the continuity property, we let $\{{c}^{n+1}_{p,k}\}_{k} \in \mathcal{K}$ be a sequence such that ${c}^{n+1}_{p,k} \rightarrow {c}^{n+1}_{p}$ in $\mathrm{L}^2(\Omega)$ as $k\to\infty$, and as is known from \eqref{eq:def_T}, $c_{p,k}^{n+1}=\mathbf{T}({c}^{n+1}_{p,k})$. In this way,  proceeding as in \cite[Lemma 2.8]{verma21}, we obtain that $c_{p,k}^{n+1}\rightarrow \mathbf{T}({c}^{n+1}_p)$ in $\mathrm{L}^2(\Omega)$ as $k\to\infty$, concluding the proof.
\end{proof}

Finally, the following lemma concerning the existence of solution for problem \eqref{eq:DR_1}-\eqref{eq:poro_3}, is merely an application of Lemmas \ref{eq:bound_T} and \ref{cor1}, and the Schauder fixed-point theorem. 
\begin{lemma}\label{eq:existence_result}
The  semi-discrete in-time formulation \eqref{eq:DR_1}- \eqref{eq:poro_3} possesses at least one solution.
\end{lemma}
\subsection{Uniqueness of weak solutions}\label{sec:uniq}
We establish the uniqueness of weak solutions through the following result.
\begin{lemma}\label{uniqueness-poro}
		The semi-discrete weak formulation  \eqref{eq:DR_1}-\eqref{eq:poro_3} has a unique solution.
\end{lemma}
\begin{proof}
	We begin by defining two solutions $(\bu_1^{n+1},p_1^{n+1}, \zeta_1^{n+1}, c_{p,1}^{n+1},c_{1,l}^{n+1})$ and $(\bu^{n+1}_2,p^{n+1}_2, \zeta^{n+1}_2, c^{n+1}_{p,2}, c^{n+1}_{l,2})$ associated with initial data $\bb_1, \bu^{0}_1, p^{0}_1, \zeta^0_1$, $c^{0}_{p,1}, c^0_{l,1}, {R}_{1,1}, \tilde{R}^n_{1,1}, {R}_{2,1}, \tilde{R}^n_{2,1},\mathbf{R}_{3,1}, {R}_{4,1}, \tilde{R}^n_{4,1}$ and $\bb_2, \bu^{0}_2,$ $ p^{0}_2$, $\zeta^0_2, c^{0}_{p,2},$ $c^0_{l,2}, {R}_{1,2}$, $\tilde{R}^n_{1,2}, {R}_{2,2}, \tilde{R}^n_{2,2}, \mathbf{R}_{3,2}, \mathrm{R}_{4,2}, \tilde{R}_{4,2}$, respectively, and then
	\begin{gather*}
	\mathcal{U}^{n+1}= \bu^{n+1}_1- \bu^{n+1}_2,\;\; \mathcal{P}^{n+1} = p^{n+1}_1 - p^{n+1}_2, \;\; \bigchi^{n+1} = \zeta_{1}^{n+1} - \zeta_{2}^{n+1},\\
	\mathcal{C}_{p}^{n+1} = c_{p,1}^{n+1} - c_{p,2}^{n+1}, \;\; \mathcal{C}_l^{n+1} = c_{l,1}^{n+1} - c_{l,2}^{n+1}.\end{gather*}
	Therefore, taking advantage of the linearity of the involved forms, choosing $\bv=\delta_t\mathcal{U}^{n+1}$ in \eqref{eq:poro_1}, and then applying Cauchy-Schwarz and Young inequalities, and multiplying the resulting inequality by $\Delta t$ and summing over $n-1$, we finally obtain 
	\begin{align}\label{stability-uh-final}
	\begin{split}
	&\mu_s C_{k,1}\Vert \mathcal{U}^{n}\Vert^2_{1,\Omega} + \frac{\mu_s C_{k,1}\Delta t^2}{4}\sum_{m=0}^{n-1}\Vert\delta_t \mathcal{U}^{m+1}\Vert^2_{1,\Omega} \\
	&\qquad\qquad  \leq  C\Big\{\Vert\mathcal{U}^{0}\Vert^2_{1,\Omega} + \sum_{m=0}^{n-1}\Vert\bigchi^{m+1}\Vert_{0,\Omega}^2+\sum_{m=0}^{n-1} \Vert\bb_1-\bb_2\Vert_{0,\Omega}^2+\sum_{m=0}^{n-1} \Vert\mathbf{R}_{3,1}-\mathbf{R}_{3,2}\Vert_{0,\Omega}^2\Big\},
	\end{split}
	\end{align}
	where $C$ is a constant independent of $\Delta t$. On the other hand, by taking $q=\mathcal{P}^{n+1}$ and $\psi=\bigchi^{n+1}$ in \eqref{eq:poro_2} and \eqref{eq:poro_3}, respectively, and then, applying again the linearity of the bilinear forms and proceeding as in \cite[Lemma 2.2]{verma21}, we arrive at the following estimate
	\begin{align}\label{prelim_poro_1}
	\begin{split}
	&\frac{1}{2\lambda_s}\Vert\bigchi^{n}\Vert^2_{0,\Omega} + \frac{1}{2}\Big(1+\frac{\alpha}{\lambda_s}\Big)\Big(\Vert \mathcal{P}^{n}\Vert^2_{0,\Omega} + \Delta t^2 \sum_{m=0}^{n-1}\Vert \delta_t \mathcal{P}^{m+1}\Vert^2_{0,\Omega}\Big)+\frac{\kappa_1\Delta t}{2\mu_f}\sum_{m=0}^{n-1}\Vert \mathcal{P}^{m+1}\Vert^2_{1,\Omega}\\
	& \leq \frac{1}{2\lambda_s}\Vert\bigchi^{0}\Vert^2_{0,\Omega}+ \frac{1}{2}\Big(1+\frac{\alpha}{\lambda_s}\Big)\Vert \mathcal{P}^{0}\Vert^2_{0,\Omega}+ \left(\frac{(1+\alpha)^2}{2\lambda}+1\right) \sum_{m=0}^{n-1}\Vert \mathcal{P}^{m+1}\Vert^2_{0,\Omega}+\frac{1}{2}\sum_{m=0}^{n-1}\Vert \tilde{R}^m_{4,1}-\tilde{R}^m_{4,2}\Vert^2_{0,\Omega}\\
	&\;\;\;\;+\frac{\mu_f\Delta t}{2\kappa_1}\sum_{m=0}^{n-1}\Vert {R}_{4,1}-{R}_{4,2}\Vert^2_{0,\Omega}+\frac{1}{2}\sum_{m=0}^{n-1} \Vert\mathcal{C}_p^{m+1}\Vert^2_{0,\Omega}
	 + \frac{1}{2\mu_s C_{k,1}}\Vert\bigchi^{n+1}\Vert^2_{0,\Omega}+ \frac{\mu_s C_{k,1}}{2}\Vert\mathcal{U}^{n+1}\Vert^2_{1,\Omega}\\
	 &\;\;\;\;+\frac{1}{\mu C_{k,1}}\sum_{m=0}^{n-1}\Vert\bigchi^{m+1}\Vert^2_{0,\Omega}+\frac{\mu_s C_{k,1}\Delta t^2}{4}\sum_{m=0}^{n-1}\Vert\delta_t \mathcal{U}^{m+1}\Vert^2_{1,\Omega}.
	 \end{split}
	\end{align} 
	Now, for the diffusion-reaction problem, we apply the linearity of the bilinear forms in \eqref{eq:DR_1}-\eqref{eq:DR_2}, with test functions $ \mathcal{C}_{p}$ and $\mathcal{C}_l$, respectively, and then, proceeding similar to Lemma \ref{eq:stability-poro}, we get 
	\begin{align}\label{prelim_DR}
	\begin{split}
	&\| \mathcal{C}_p^{n} \|_{0, \Omega}^2 + \| \mathcal{C}_l^{n} \|_{0, \Omega}^2 + \Delta t^2 \sum_{m=0}^{n-1} (\| \delta_{t}\mathcal{C}_p^{m+1} \|_{0, \Omega}^2+ \| \delta_{t}\mathcal{C}_l^{m+1} \|_{0, \Omega}^2)  + D^{\min} \Delta t \sum_{m=0}^{n-1}  (|\mathcal{C}_p^{m+1}|^2_{1,\Omega} + |\mathcal{C}_l^{m+1}|^2_{1,\Omega} )\\
	&\;\;  \leq \;C \bigg( \| \mathcal{C}_p^{0} \|_{0, \Omega}^2 + \| \mathcal{C}_l^{0} \|_{0, \Omega}^2+\sum_{m=0}^{n-1}\Big(\Vert \tilde{R}^m_{1,1}-\tilde{R}^m_{1,2}\Vert_{0, \Omega}^2+\Vert \tilde{R}^m_{2,1}-\tilde{R}^m_{2,2}\Vert_{0, \Omega}^2\Big)+ \frac{\Delta t}{2}\sum_{m=0}^{n-1}\Big(\Vert R_{1,1}-R_{1,2}\Vert^2\Big.   \\
	&\quad\;\;+\Vert R_{2,1}-R_{2,2}\Vert_{0, \Omega}^2 \Big)+  (1 + \frac{\Delta t}{2}) \sum_{m=0}^{n-1} \Big(\|\mathcal{C}_p^{m+1}\|_{0,\Omega}^2 + \|\mathcal{C}_l^{m+1}\|_{0,\Omega}^2\Big) + \sum_{m=0}^{n-1}\Vert\bigchi^{m+1}\Vert^2_{0,\Omega} +\sum_{m=0}^{n-1}\Vert\mathcal{P}^{m+1}\Vert_{0,\Omega}^2\bigg). \end{split}
	\end{align}
where $D^{\mathrm{min}}:=\min\{D_p^{\mathrm{min}},D_l^{\mathrm{min}}\}$. 
Finally, thanks to the inf-sup condition \eqref{inf-sup}, it is possible to obtain a bound for by $\| \bigchi^{n+1}\|_{0,\Omega}$ {independent of $\lambda_s$}, and then, combining  \eqref{stability-uh-final} - \eqref{prelim_DR}, and then using Gronwall's lemma, it can be deduced that 
\begin{align*}
&\| \mathcal{U}^{n+1}\|_{1,\Omega} + \| \mathcal{P}^{n+1}\|_{0,\Omega} + \| \bigchi^{n+1}\|_{0,\Omega}
+ \| \mathcal{C}_p^{n+1} \|_{0, \Omega} + \| \mathcal{C}_l^{n+1} \|_{0, \Omega} +  \| \mathcal{P}\|_{l^2(H^1(\Omega))} +  \| \nabla \mathcal{C}_p \|_{l^2(L^2(\Omega))}  \\
&+ \| \nabla \mathcal{C}_l \|_{l^2(L^2(\Omega))} \le C\sqrt{\exp(C_1)}\biggl( \|\mathcal{U}^0 \|_{1,\Omega} + \|\mathcal{P}^0 \|_{0,\Omega} + \|\bigchi^0 \|_{0,\Omega} +	\| \mathcal{C}_p^0 \|_{0,\Omega} + \| \mathcal{C}_l^0 \|_{0,\Omega} + \sum_{m=0}^n \| \bb_1 - \bb_2\|_{0, \Omega}  \\
&+  \sum_{m=0}^{n}\Big(\Vert \tilde{R}^m_{1,1}-\tilde{R}^m_{1,2}\Vert^2+\Vert \tilde{R}^m_{2,1}-\tilde{R}^m_{2,2}\Vert^2\Big)+ \Delta t\Vert R_{1,1}-R_{1,2}\Vert_{0,\Omega}^2+\Delta t\Vert R_{2,1}-R_{2,2}\Vert^2_{0,\Omega} \\
& + \Vert\mathbf{R}_{3,1}-\mathbf{R}_{3,2}\Vert_{0,\Omega}^2+ \sum_{m=0}^{n}\Vert \tilde{R}^m_{4,1}-\tilde{R}^m_{4,2}\Vert^2_{0,\Omega}+\Delta t\Vert {R}_{4,1}-{R}_{4,2}\Vert^2_{0,\Omega}\biggr),
\end{align*}
from which, we can ensure the existence of at most one weak solution to the system  \eqref{eq:DR_1} - \eqref{eq:poro_3}.
\end{proof}
\subsection{Stability  of the linearised coupled problem}\label{sec:stab}
The following lemma establishes the continuous dependence on data for problem \eqref{eq:DR_1} - \eqref{eq:poro_3}. Its proof follows similar arguments to those used in Lemmas \ref{eq:stability-poro}, \ref{uncoupled-ADR} and \ref{uniqueness-poro}, and therefore we omit it here. 
\begin{lemma}\label{eq:stability}
	The solution $(c_p^{n+1},c_l^{n+1},\bu^{n+1}, p^{n+1}, \zeta^{n+1})\in \mathrm{H}^1(\Omega)\times \mathrm{H}^1(\Omega)\times \mathbf{H}_\Gamma^1(\Omega)\times\mathrm{H}_\Sigma^1(\Omega)\times \mathrm{L}^2(\Omega)$ of
	problem \eqref{eq:DR_1} - \eqref{eq:poro_3}  satisfies
	\begin{align*}
	\begin{split}
	&\Vert \bu^{n+1}\Vert_{1,\Omega} + \Vert p^{n+1}\Vert_{0,\Omega} + \Vert \zeta^{n+1}\Vert_{0,\Omega} + \Vert p\Vert_{\ell^2(H^1(\Omega))}+ \Vert c_{p}^{n+1}\Vert_{0,\Omega} + \Vert c^{n+1}_{l}\Vert_{0,\Omega}+\Vert \nabla c_p \Vert_{\ell^2(L^2(\Omega))}\\
	&\quad + \Vert \nabla c_l\Vert_{\ell^2(L^2(\Omega))}\leq \sqrt{\exp(C)} \Big\{ \Vert\bu^{0}\Vert_{1,\Omega} + \Vert p^{0}\Vert_{0,\Omega}+ \Vert \zeta^{0}\Vert_{0,\Omega}+\Vert c_{p}^{0}\Vert_{0,\Omega}+\Vert c_{l}^{0}\Vert_{0,\Omega}+\Vert\bb\Vert_{0,\Omega}\\
	&\quad +  \sum_{m=0}^{n}\Big(\Vert \tilde{R}^m_{1}\Vert_{0,\Omega}+\Vert \tilde{R}^m_{2}\Vert_{0,\Omega}+\Vert \tilde{R}^m_{4}\Vert_{0,\Omega}\Big)+\Vert\mathbf{R}_3\Vert_{0,\Omega}+ \Delta t\Vert R_{1}\Vert_{0,\Omega}+\Delta t\Vert R_{2}\Vert_{0,\Omega}+\Delta t\Vert R_{4}\Vert_{0,\Omega} \Big\},
	\end{split}
	\end{align*}
	where $C>0$ is a constant {independent of $\lambda_s$}.
\end{lemma}
\begin{remark}\label{remark}
\hspace{1cm}
	\begin{itemize}
\item Although showing the solvability analysis of the linearised fully-continuous problem goes beyond the scope of this work, it is possible to extend the ideas presented here to establish that well-posedness  using appropriate choices of Sobolev spaces, and performing a passage to the limit adapting to our problem the results from, e.g., \cite{anaya18} (which focus on reaction--diffusion--Brinkman problems). 
\item The solvability analysis of the fully-discrete problem associated with \eqref{eq:DR_1}-\eqref{eq:poro_3} can be established similarly to the semidiscrete-in-time case. More precisely, with the finite elements spaces specified in \eqref{fe-spaces}, we can define an appropriate fixed-point operator and prove that it is well-defined thanks to the well-posedness of each uncoupled problem (see a general form in, e.g., \cite{quarteroni94}). Then, by employing the continuous dependence in the  fully-discrete case (which follows exactly as in Lemma \ref{eq:stability}) in combination with the analysis from \cite[Section 5.3]{anaya18}, we can prove the continuity of the fixed-point operator described above. Finally, the desired result follows from an application of   Brouwer's fixed--point theorem.
\item Given the approximation properties of the finite element spaces specified in \eqref{fe-spaces} (and recalled in, e.g., \cite{oyarzua16}), and following the steps given by \cite[Section 4]{verma21}, it is possible to derive an asymptotic $\mathcal{O}(h)$ convergence for
the proposed method. This is also verified for the nonlinear case, as shown numerically in Section~\ref{sec:results}. 
\item {The convergence properties are also robust with respect to incompressibility of the solid phase}. 
\item We stress that the arguments presented in this section are unfortunately not readily generalised to arbitrary hyperelastic materials, nor a generic initial guess for the Newton--Raphson loop. {In all of the numerical tests performed below, this has not shown to be a problem in practice. Still, the lack of a global existence result makes it is possible for the model to be increasingly difficult in the simulation of other scenarios of interest, such as pathological ones.}
\end{itemize}
\end{remark}

\section{{A} finite element formulation}\label{sec:fem}
\subsection{Discretisation of the nonlinear problem}
In order to define a Galerkin finite element method we denote by $\{\cT_{h}\}_{h>0}$ a shape-regular
family of partitions of 
$\bar\Omega$, conformed by tetrahedra (or triangles 
in 2D) $K$ of diameter $h_K$, with mesh size
$h:=\max\{h_K:\; K\in\cT_{h}\}$. Given an integer $k\ge1$ and a subset
$S$ of $\mathbb{R}^d$, $d=2,3$, by $\mathbb{P}_k(S)$ we will denote the 
space of polynomial functions defined locally in $S$ and being of total degree up to $k$. Let us also denote by $b_K:= \varphi_1\varphi_2\varphi_3$  a $\mathbb{P}_3$ bubble function in $K$, where $\varphi_1,\,\varphi_2\,,\varphi_3$ are the barycentric coordinates of the triangle $K$ (in 3D the bubble is a quartic polynomial). 
Then the finite-dimensional subspaces for the pathogen and leukocyte concentrations $\rW_h\subseteq \rH^1(\Omega)$, displacement $\bV_h\subseteq \bH_\Gamma^1(\Omega)$, 
porous fluid pressure  $\rQ_h\subseteq \rH_\Sigma^1(\Omega)$, and nominal porosity  $\Phi_h\subseteq \rL^2(\Omega)$ are defined, respectively, as follows 
\begin{align}	
  \nonumber	\rW_h&:=\{w_h\in C(\overline{\Omega}): w_h|_K\in\mathbb{P}_{1}(K)^d,\ \forall K\in\cT_{h}\},\\
  \nonumber	\bV_h&:=\{\bv_h\in \boldsymbol{C}(\overline{\Omega}): \bv_h|_K
  \in [\mathbb{P}_1(K)\oplus {\rm span}\{b_K\}]^d\ \forall K\in\cT_{h},\ \bv_h|_\Gamma = \cero\},\\
\label{fe-spaces}\rQ_h&:=\{q_{h}\in C(\overline{\Omega}): q_{h}|_K\in\mathbb{P}_{1}(K),\ \forall K\in\cT_{h},\ q_h|_\Sigma = 0\}, \\
\nonumber  \Phi_h&:=\{\psi_h\in C(\overline{\Omega}): \psi_h|_K\in\mathbb{P}_1(K),\ \forall K\in\cT_{h}\}.
\end{align}
The pair $(\bV_h,\Phi_h)$ is the well-known MINI-element, which is inf-sup stable
in the context of saddle-point Stokes equations in their velocity-pressure formulation \cite{cion19}.

Then the fully discrete problem arises from \eqref{eq:weak-linear} and for each 
time step $n$, we perform  inner Newton-Raphson iterations from $k=0,\ldots$ 
seeking 
 $(\delta c_{p,h}^{k+1},\delta c_{l,h}^{k+1},\delta\bu^{k+1}_h, \allowbreak \delta p^{k+1}_{h},\delta\phi^{k+1}_{f,h})\in \rW_h\times\rW_h\times\bV_h\times \rQ_h\times\Phi_h=:\mathbf{H}_h$ solutions 
 to the unsymmetric matrix system  
\begin{equation}\label{eq:fully-discrete}
\left[\begin{array}{ccccc}
 \cA_1&\cB_1' & \cB_2' & \cero & \cB_3' \\[1.5ex]
\cB_1 &  \cA_2 & \cB_4' & \cero & \cB_5' \\[1.5ex]
\cero &  \cero  &\cA_3 & \cC_1' &\cero \\[1.5ex]
\cC_2 &  \cero  & \cC_1 & \cA_4 & \cC_3'\\[1.5ex]
\cero &  \cero & \cC_4 & \cero & - \cA_5
\end{array}\right] 
\left[\begin{array}{c}
\delta C_{p}^{k+1}\\[1ex]
\delta C_{l}^{k+1}\\[1ex]
\delta U^{k+1}\\[1ex]
\delta P^{k+1}\\[1ex]
\delta \Phi^{k+1}_{f}
\end{array}\right]=
\left[\begin{array}{c}
\cF_{1,h}^{k,n} \\[1ex]  
\cF_{2,h}^{k,n} \\[1ex]  
\cF_{3,h}^{k} \\[1ex] 
\cF_{4,h}^{k,n} \\[1ex] 
\cF_{5,h}^{k} \end{array}
\right],
\end{equation}
where the entries $\delta C_{p}^{k+1}$, $\delta C_{l}^{k+1}$, $\delta U^{k+1}$, $\delta P^{k+1}$ and $\delta \Phi_f^{k+1}$ in the independent vector variable  are the vectors containing all internal degrees of freedom associated with the discrete incremental solutions  
for all fields, and 
the operators in calligraphic letters appearing in the coefficient matrix and load vector from \eqref{eq:fully-discrete} are 
induced by the corresponding bilinear forms and linear functionals in \eqref{def:forms}. The functionals with superscript $n$ 
on the right-hand side vector indicate that they also receive contributions from the backward Euler 
time-discretisation.  

\subsection{Schur complement based robust preconditioner}\label{sec:preconditioner}
The numerical solution of problem \eqref{eq:fully-discrete} through direct methods is not feasible for large systems, which is a natural consequence of considering finer meshes for 3D geometries, needed to better capture anatomical details. This is the standard scenario in which Krylov space methods are the most useful, more specifically a GMRES method due to the indefinite and non-symmetric nature of the problem \cite{saad1986}. Still, scalable solvers require the construction of a robust preconditioner that allows for the Krylov iterations to remain roughly constant whenever (i) the amount of processors is increased and (ii) the mesh is refined. In this section we propose a preconditioner based on a two-level nested Schur complement inspired by the multiphysics nature of the model. If we consider a block matrix $\mathbf M$ given by the general structure 
    \[ \mathbf M = \begin{bmatrix} \mathbf A & \mathbf B_1 \\ \mathbf B_2 & \mathbf C \end{bmatrix} ,\]
with $\mathbf A$ invertible, a Gauss elimination procedure yields
    \begin{equation}\label{eq:schur}
    \mathbf M = \begin{bmatrix} I & \mathbf 0 \\ \mathbf B_2\mathbf A^{-1} \end{bmatrix} \begin{bmatrix} \mathbf A & \mathbf 0 \\ \mathbf 0 & \mathbf C -\mathbf B_2\mathbf A^{-1}\mathbf B_1  \end{bmatrix} \begin{bmatrix} \mathbf I & \mathbf A^{-1}\mathbf B_1 \\ \mathbf 0 & \mathbf I \end{bmatrix}. \end{equation}
We note that if $\mathbf C$ is invertible, the same procedure can be applied with respect to it. Schur complement based preconditioners enjoy excellent theoretical properties, as the preconditioned system possesses at most 3 distinct eigenvalues \cite{murphy2000note}, implying that it converges in at most 3 iterations of a Krylov subspace method, but this is seldom true in practice as the Schur complement block $\mathbf S = \mathbf C -\mathbf B_2\mathbf A^{-1}\mathbf B_1$ is computationally expensive to compute. Two complementary strategies to circumvent this problem are (i) to consider the full/lower triangular/upper triangular/diagonal part of \eqref{eq:schur} and (ii) to approximate $\mathbf S$. In turn, two standard approaches for approximating $\mathbf S$ are the SIMPLE preconditioner given by $\mathbf S\approx \mathbf C -\mathbf B_2\text{diag}\,(\mathbf A)^{-1}\mathbf B_1$ and the block diagonal approximation $\mathbf S\approx \mathbf C$ (for other possibilities and their comparison see, e.g., the comprehensive review \cite{elman2008taxonomy}). 

On the first level of the proposed preconditioner, we split the variables into poroelastic and chemotaxis, denoted by $\text{poro}=(\bu, p, \phi)$ and $\text{chem}=(c_p, c_l)$ respectively, and use this to write the linearised problem as     
    \begin{equation}\label{eq:prec poro-chem} 
\mathbf J = \begin{bmatrix} \mathbf J_\text{poro} & \mathbf J_\text{poro, chem} \\ \mathbf J_\text{chem, poro} & \mathbf J_\text{chem} \end{bmatrix}, 
    \end{equation}
where each block is given by 
    \begin{align*}
        \mathbf J_\text{poro} &= \left[\begin{array}{ccccc}
            \cA_3 & \cC_1' &\cero \\[1.5ex]
            \cC_1 & \cA_4 & \cC_3'\\[1.5ex]
            \cC_4 & \cero & - \cA_5
            \end{array}\right], \quad 
        \mathbf J_\text{chem} = \left[\begin{array}{ccccc}
             \cA_1&\cB_1' \\[1.5ex]
            \cB_1 &  \cA_2 
            \end{array}\right], \quad 
        \mathbf J_\text{poro,chem} = \left[\begin{array}{ccccc}
            \cero &  \cero   \\[1.5ex]
            \cC_2 &  \cero  \\[1.5ex]
            \cero &  \cero 
            \end{array}\right], \quad 
        \mathbf J_\text{chem,poro} = \left[\begin{array}{ccccc}
            \cB_2' & \cero & \cB_3' \\[1.5ex]
            \cB_4' & \cero & \cB_5'
            \end{array}\right].
    \end{align*}
    
Using \eqref{eq:schur} we can write the inverse of \eqref{eq:prec poro-chem} as
    $$\mathbf J^{-1} = 
    \begin{bmatrix} \mathbf I & -\mathbf J_\text{poro}^{-1}\mathbf J_\text{poro,chem} \\ \mathbf 0 & \mathbf I \end{bmatrix}
     \begin{bmatrix} \mathbf J_\text{poro}^{-1} & \mathbf 0 \\ \mathbf 0 & (\mathbf J_\text{chem} -\mathbf J_\text{chem,poro}\mathbf J_\text{poro}^{-1}\mathbf J_\text{poro,chem})^{-1}  \end{bmatrix}
     \begin{bmatrix} I & \mathbf 0 \\ -\mathbf J_\text{chem,poro}\mathbf J_\text{poro}^{-1} & \mathbf 0 \end{bmatrix}. $$
Still, the application of this preconditioner requires applying $\mathbf J_\text{poro}^{-1}$, for which we consider an additional splitting between displacement and the pair (pressure, nominal porosity) denoted as $\bu-\Pi$, with $\Pi=(p,\phi_f)$. We can apply the same argument to the block $\mathbf J_\text{poro}$, now written as 
    \[ \mathbf J_\text{poro} = \begin{bmatrix} \mathbf J_{\bu} & \mathbf J_{\bu,\Pi} \\ \mathbf J_{\Pi, \bu} & \mathbf J_{\Pi} \end{bmatrix},\]
by making use of the invertibility of the displacement block $\mathbf J_{\bu}$, where the sub-blocks are given by
    \begin{align*}
        \mathbf J_{\bu} &= \cA_3, \qquad 
        \mathbf J_{\Pi} = \left[\begin{array}{ccccc}
            \cA_4 & \cC_3'\\[1.5ex]
            \cero & - \cA_5
            \end{array}\right], \qquad 
        \mathbf J_{\bu,\Pi} = \left[\begin{array}{ccccc}
            \cC_1' &\cero
            \end{array}\right], \qquad 
        \mathbf J_{\Pi,\bu} = \left[\begin{array}{ccccc}
            \cC_1 \\[1.5ex]
            \cC_4 
            \end{array}\right].
    \end{align*}
 This idea has been applied in \cite{deparis2016facsi} for an FSI problem, and in \cite{kirby2018solver} for the Boussinesq equations. The equations governing $(c_p,c_l)$ are of parabolic type and so their approximation does not pose a challenge, which indicates that the main difficulty lies in the poroelastic block approximation. Motivated by this, the proposed preconditioner is given by the following components:
\begin{itemize}
    \item On the first level we consider a lower-triangular Schur decomposition and on the second level we consider a full one.
    \item The ``chem'' block is approximated by the action of the HYPRE-BoomerAMG \cite{yang2002boomeramg} preconditioner. 
    \item Both blocks in the second level are approximated by an inexact GMRES solver preconditioned by an additive Schwarz method with a direct solver in each subdomain (using the MUMPS library \cite{amestoy2000mumps}).
\end{itemize}
 The inexactness is given by a relative tolerance of $0.1$, and we highlight that an inexact solver in the sub-blocks gives rise to a preconditioner that changes from one iteration to the next one. This requires the use of a flexible GMRES algorithm (fGMRES) \cite{saad1993flexible}.  The use of an additive Schwartz preconditioner, or in general a domain decomposition one, alleviates the deteriorating scalability of AMG preconditioners for higher-order elements, in this case required by the displacement to satisfy an appropriate discrete inf-sup condition. 
 
 We close this section by noting that the choice of the type of Schur decomposition is not arbitrary. For this, we consider exact, lower triangular Schur complement preconditioners at both levels. In this way, the preconditioner adopts the following form 
\[ \mathbf P = \begin{bmatrix}
    \mathbf P_\text{poro} & \mathbf 0 \\
    \mathbf J_\text{chem,poro} & \mathbf S_\text{chem}
    \end{bmatrix}, \]
where
    \begin{gather*}
        \mathbf P_\text{poro} = \begin{bmatrix}
        \mathbf J_{\bu} & \mathbf 0 \\
        \mathbf J_{\Pi,\bu} & \mathbf S_{\Pi} 
        \end{bmatrix},  \quad 
        \mathbf S_{\Pi} = \mathbf J_{\Pi} - \mathbf J_{\Pi,\bu}\mathbf J_{\bu}^{-1}\mathbf J_{\bu,\Pi}, \quad 
        \mathbf S_\text{chem} = \mathbf J_\text{chem} -\mathbf J_\text{chem,poro}\mathbf J_\text{poro}^{-1}\mathbf J_\text{poro,chem}. 
    \end{gather*}
After algebraic manipulations, it can be seen that the preconditioned system is given by 
\begin{equation}\label{eq:preconditioned system}
    \mathbf P^{-1}\mathbf J = \begin{bmatrix}
        \mathbf P_\text{poro}^{-1}\mathbf J_\text{poro} & \mathbf P^{-1}_\text{poro}\mathbf J_\text{poro,chem} \\
        \mathbf S^{-1}_\text{chem}\mathbf J_\text{chem,poro}(\mathbf I - \mathbf P^{-1}_\text{poro}\mathbf J_\text{poro}) & \mathbf S^{-1}_\text{chem}(\mathbf J_\text{chem} -\mathbf J_\text{chem,poro}\mathbf P_\text{poro}^{-1}\mathbf J_\text{poro,chem})
    \end{bmatrix},
\end{equation}
where
\[ \mathbf P_\text{poro}^{-1}\mathbf J_\text{poro} = \begin{bmatrix}
            \mathbf I & \mathbf J_{\bu}^{-1}\mathbf J_{\bu,\Pi} \\
            \mathbf 0 & \mathbf I
        \end{bmatrix}.\]
        
The preconditioned system does not exhibit the expected property of having a small number of eigenvalues. This is mainly due to the use of a  preconditioner instead of an exact inverse, where $\mathbf P_\text{poro} \neq \mathbf J_\text{poro}$. If such blocks coincide, then we would have

\begin{equation}\label{eq:preconditioned system exact}
    \mathbf P^{-1}\mathbf J = \begin{bmatrix}
        \mathbf I & \mathbf J^{-1}_\text{poro}\mathbf J_\text{poro,chem} \\
        \mathbf 0 & \mathbf I
    \end{bmatrix},
\end{equation}
which justified the use of an exact Schur preconditioner at the second level. This results in a preconditioned problem with only one eigenvalue.



\section{Computational experiments}\label{sec:results}
We now turn to the presentation of numerical examples serving to illustrate the performance of the {finite element} scheme and to examine further the main features of the model. All routines have been implemented using the open source finite element library \texttt{FEniCS} \cite{alnaes15}. A fixed tolerance of $10^{-6}$ is used on the residuals for the convergence criterion of the {Newton--Raphson} iterative algorithm. {We highlight that the error control of multiphysics problems is not trivial, for example in \cite{borregales2018robust} as an error they use the norm of the increment, whereas in \cite{white2019two} they use the norm of the residual. We prefer the norm of the residual, as having an increment going to zero might mean stagnation and not convergence. Still, this can be improved by checking the error in each physics separatedly.}

\subsection{Sensitivity analysis of model parameters}\label{sec:sensitivity}

This section presents the influence of some parameters on the dynamics of the proposed model to highlight the coupling mechanisms between the poroelastic {component}~\eqref{eq:poroe} and the advection-reaction-diffusion {component}~\eqref{eq:immune} describing pathogens and immune system dynamics, respectively. The analyses were performed in a one dimensional version of the {fully coupled model given by equations~\eqref{eq:poroe}-\eqref{eq:immune} considering a domain $\Omega \in [0,8]$ cm.} {An initial pathogen concentration of $c_{p,0} = 0.001$ was considered in the middle of the domain $x \in [3.95,4.05]$ to start the infection, whereas in the remaining domain $c_{p,0} = 0$ was used}. Throughout the domain we also consider {the following initial conditions:} $c_{l,0} = 0.003$, {$\phi_{0} = 0.2$,} and $p_0 = 0$. The following boundary conditions were considered: $\bu = \cero$ on the left {at $x=0$}, and $p = 0$ was prescribed at the right end {at $x=8$}\,cm of the domain. 

\begin{table}[!t]
	\centering
	\resizebox{\columnwidth}{!}{
	\begin{tabular}{lllc|lllc}
		\toprule
		\textbf{Parameter} & \textbf{Units} & \textbf{Description}     & \textbf{References}    & \textbf{Parameter} & \textbf{Units } & \textbf{Description}    & \textbf{References} \\
		\midrule
		$E=60$                                          & ${kg}/{cm\, s^2}$     & Young modulus      & estimated                   
		   & $\lambda_{lp} = 1.5$                            & ${1}/{d \ c}$      & Phagocytosis rate    & estimated                    \\
		$\nu = 0.35$                                    & $-$                   & Poisson coefficient  &  estimated              
		   & $\bar{\lambda}_{lp} = \lambda_{lp}/\phi_0$      & ${1}/{d \ c}$      & Relative phagocytosis rate   &                  \\ 		
		$\lambda_s$                                     & ${kg}/{cm\,s^2}$      & First Lam\'e parameter   &               
		   & $\lambda_{pl} = 7.1$                                       & ${1}/{d \ c}$      & Leukocytes migration rate          & estimated     \\ 
		$\mu_s$                                         & ${kg}/{cm\,s^2}$      & Shear module     &                    
		  & $\bar{\lambda}_{pl} = \lambda_{pl}/\phi_0$                            & ${1}/{d \ c}$             & Relative leukocytes migration      &      \\ 
		$\rho_f = 1$                                     & {${kg}/{cm^3}$}       & Fluid phase density  &  estimated               
		  & $\pi_i = 10$                                    & $mmHg$                & Interstitial oncotic pressure    & \cite{phipps11}       \\ 
		$\rho_s = 2e^{-3}$                              & ${kg}/{cm^3}$         & Solid phase density                 & \cite{barnafi2021multiscale}
		  & $\pi_c = 20$                                    & $mmHg$                & Capillary oncotic pressure    & \cite{phipps11}           \\ 
		$\alpha = 0.25$                                 & $-$                   & Biot modulus                       & estimated 
		  & $\sigma_0 = 0.91$                               & $-$                   & Osmotic reflection coefficient      & \cite{phipps11}    \\
		$\phi_0 = 0.2$                                  & $-$                   & Initial fluid phase                 & \cite{basser92}
		  & $L_{bp} = 1e4$                                  & ${1}/{c}$             & Pathogen influence on permeability   & estimated   \\
		$D_p = 1e^{-3}$                                 & ${cm^2}/{d}$          & Pathogen diffusion       & estimated
		  & $P_c = 20$                                      & $mmHg$                & Capillary pressure                & \cite{phipps11}      \\
		$\bar{D_p} = D_p/\phi_0$                        & ${cm^2}/{d}$          & Relative pathogen diffusion   &
		  & $L_{p0} = 3.6e^{-8}$                            & ${cm}/{s \ mmHg}$  & Hydraulic permeability             & \cite{phipps11}     \\ 
		$D_l = 5e^{-2}$                                 & ${cm^2}/{d}$          & Leukocyte diffusion           & estimated
		  & $\ell_0 = 6.82e^{-5}$                           & ${1}/{s}$             & Normal lymphatic flow             & estimated      \\ 
		$\bar{D_l} = D_l/\phi_0$                        & ${cm^2}/{d}$          & Relative leukocyte diffusion  &
		  & $k_m = 6.5$                                     & $mmHg$                & Half life of lymphatic flow         & estimated    \\ 
		$\mathcal{X} = 1e^{-2}$                         & ${cm^2}/{d \ c}$   & Chemotaxis                              & estimated
		  & $nHill = 1$                                     & $-$                   & Hill coefficient                   & estimated     \\ 
		$\bar{\mathcal{X}} = \mathcal{X}/\phi_0$        & ${cm^2}/{d \ c}$   & Relative chemotaxis  &                    
		  & $\mbox{V}_{max} = 200$                          & $-$                   & Max lymphatic flow                 & estimated     \\
		$\gamma_p = 1.2e^{-1}$                          & ${1}/{d}$             & Pathogen reproduction rate              & estimated
		  & $K = 2.5e^{-7}$                                 & ${cm^2}/{s \ mmHg}$& Permeability                     & \cite{phipps11}       \\	
		$\bar{\gamma_p} = \gamma_p/\phi_0$              & ${1}/{d}$             & Relative pathogen reproduction rate     & 
		  & $(S/V) = 174$                                   & $1/cm$                & Vessel area per volume unit           & \cite{phipps11}  \\
		\bottomrule
	\end{tabular}
	}
	\caption{Example 1: Reference parameters of the model to be used in Section~\ref{sec:sensitivity}.}\label{tab:parameters}
\end{table}

A simplified sensitivity analysis varying one parameter at a time with respect to the reference parameters reported in Table~\ref{tab:parameters} was performed. We limit the presentation of these results to only a few parameters and variables divided into three cases that highlight the couplings of the model. Young's modulus $E$, associated with the mechanical part, was investigated for the first case. In the second and third cases, associated with the immune response of the model, both phagocytosis rate, $\lambda_{lp}$, and pathogen's reproduction rate, $\gamma_P$, were studied. These parameters were chosen because they better highlighted the couplings in the model. They are also related to critical biological responses depending on their values. Changes in the value of $\gamma_P$, for example, can be associated with pathogens with distinct proliferation abilities; changes in the value of $\lambda_{lp}$ can describe the ability of leukocytes in dealing with the invading pathogen; whereas changes in $E$ represent tissues with different stiffness.

In all the three cases, the investigated parameter ($E$, $\lambda_{lp}$, or $\gamma_P$) had its value doubled and halved with respect to the reference value reported in Table~\ref{tab:parameters}. To study these three scenarios, we assume that pathogens are responsible for triggering the inflammatory response. For this reason, the instant in which pathogens' concentration reaches its peak is used as a reference to collect data from the other variables of interest. 

Figure~\ref{1d-analysis} shows the results of the sensitivity analysis. The time in which pathogens concentration reaches its maximum was distinct for each scenario only for the analysis of the influence of pathogen's reproduction rate (cf. Figure~\ref{1d-analysis}, third line), occurring after 12 days, six days, and 24 days, respectively, since the start of infection. For the analysis of $\lambda_{lp}$ and $\lambda_{pl}$, the instant of time at which the pathogens reach the maximum was $12$ days since the start of infection.

\begin{figure}[!t]
\begin{center}
\includegraphics[width=1.0\textwidth]{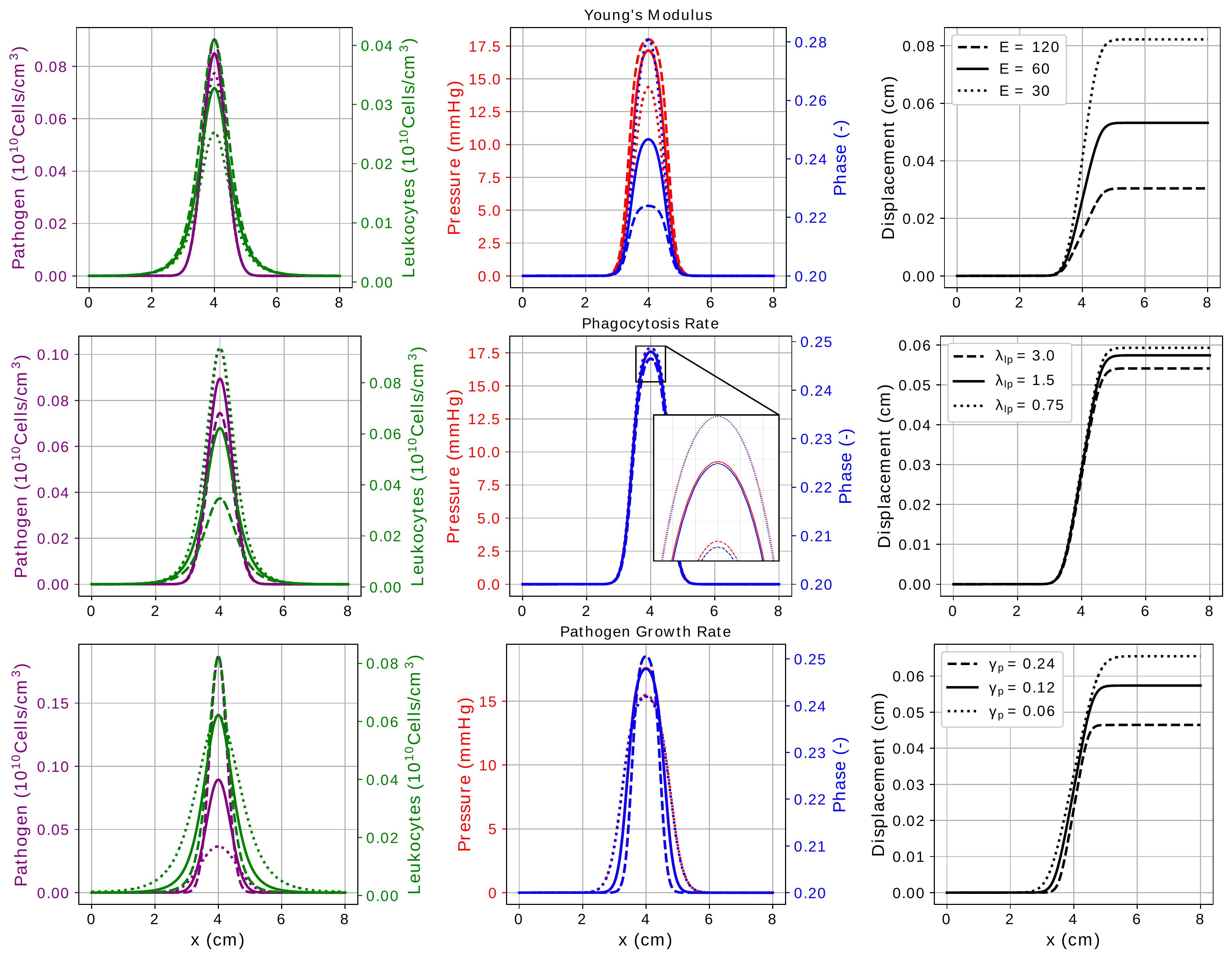}
\end{center}
\caption{Example 1: The interaction of immune response and mechanical parts of the model is evaluated using Young's modulus ($E$, first line), phagocytosis rate ($\lambda_{lp}$, second line) and pathogen reproduction rate ($\gamma_P$, last line) at the time instant in which the peak pathogen concentration is reached.  For each of these three parameters,  three scenarios are evaluated and plotted as distinct curves in each graphic: $E = 60$ (scenario 1,  solid line), $E = 120$ (scenario 2, dashed line) and $E = 30$ (scenario 3, dotted line); $\lambda_{lp} = 1.5$ (scenario 1,  solid line), $\lambda_{lp} = 3.0$ (scenario 2,  solid line) and $\lambda_{lp} = 0.75$ (scenario 3,  solid line); $\gamma_P = 0.12$ (scenario 1, solid line), $\gamma_P = 0.24$ (scenario 2, dashed line), $\gamma_P = 0.06$ (scenario 3, dotted line). The graphics in the left and in the middle {columns} are composed of two ordinate $y-$axes. {In the first column}, the left $y-$axis represents the concentration of pathogens and the right $y-$axis the concentration of leukocytes. {The second column} presents the pressure field and the fraction of fluid phase, respectively, on the left and right $y-$axes.  {The third column} represents the deformation field on a single $y-$axis. All figures represent, on the $x-$axis, the tissue size in centimetres.}\label{1d-analysis}
\end{figure} 

The plots in the first row of Figure~\ref{1d-analysis} presents the results for varying tissue stiffness. A stiffer tissue will imply a lower capacity to deform and accumulate fluid. With a smaller amount of fluid in the infected region, the concentration of pathogens and leukocytes in the tissue increases (Figure~\ref{1d-analysis}, {top left}), i.e., the same amount of pathogens in the region is divided by a smaller fluid volume.  

The second row of panels in Figure~\ref{1d-analysis} presents the results for varying the phagocytosis rate, $\lambda_{lp}$. A reduction in $\lambda_{lp}$ value can be related to a scenario in which the immune system is ineffective in dealing with the invading pathogen. This could be explained by immunodeficiency, for example. In this context, the concentration of pathogens will reach a higher and faster peak of infection for smaller values of $\lambda_{lp}$. The lower capacity of phagocytosis by leukocytes leads to a considerable increase in the concentration of pathogens (Figure~\ref{1d-analysis}, {centre left}). The increase in pathogens' concentration induces an increase in the capillary permeability, which allows a more significant accumulation of fluid in the infected region (Figure~\ref{1d-analysis} {centre middle}). The accumulation of fluid in the region causes the pressure to increase. As a consequence of these events, the deformation suffered by the tissue is greater than the one observed for higher values of $\lambda_{lp}$ (Figure~ \ref{1d-analysis} {centre right}).

Note that an increase in $\gamma_P$ can be interpreted as a pathogen that can reproduce faster due to its intrinsic biological characteristics or due to the presence of favourable conditions (abundance of nutrients or temperature, for example). In fact, as the bottom row of Figure~\ref{1d-analysis} shows, the peak of infection was higher and faster for larger $\gamma_P$ values, as expected. However, counter-intuitively, high values of $\gamma_P$ are associated with small deformations and small regions of oedema. As a result of more pathogens, the capillary permeability increases to allow more leukocytes to enter into the tissue (Figure~\ref{1d-analysis},  {bottom left}). The rapid increase in the concentration of leukocytes contained the infection quickly, preventing pathogens from remaining in the tissue long enough to spread. Since the time pathogens stay in the tissue is short, as well as the region in which they spread, the deformation suffered by the tissue also decreases (Figure~\ref{1d-analysis},  {bottom right}). 

In the third scenario, with the small value of $\gamma_P = 0.06$, the peak of infection only occurs 24 days after the invasion begins, which leads to a weak inflammatory response. Thus, the concentration of leukocytes is reduced. The low leukocyte concentration results in a long time to remove pathogens that can spread over a larger region. The weak immune response results in more deformation when compared to the cases with faster pathogen dynamics, which highlights the couplings and strong nonlinearities in the model. 
 
For this test we have also used the FEniCS environment along with the direct solver MUMPS.

\subsection{Example 2: Compression and drainage of a poroelastic sample}

\begin{figure}[t!]
\begin{center}
\includegraphics[width=0.325\textwidth]{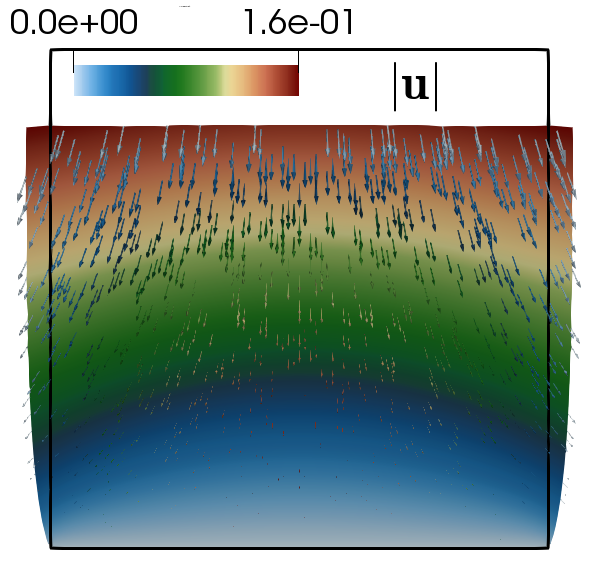}
\includegraphics[width=0.325\textwidth]{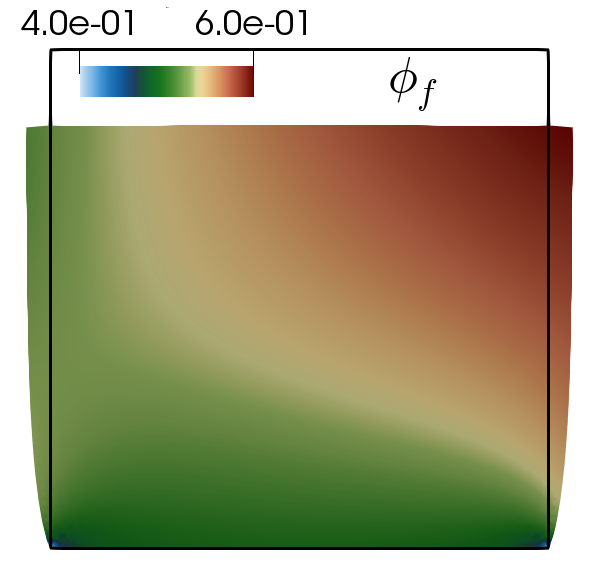}
\includegraphics[width=0.325\textwidth]{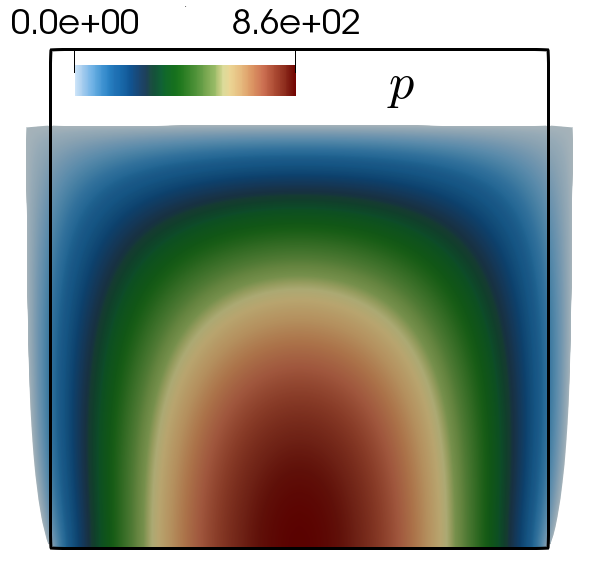}\\
\includegraphics[width=0.325\textwidth]{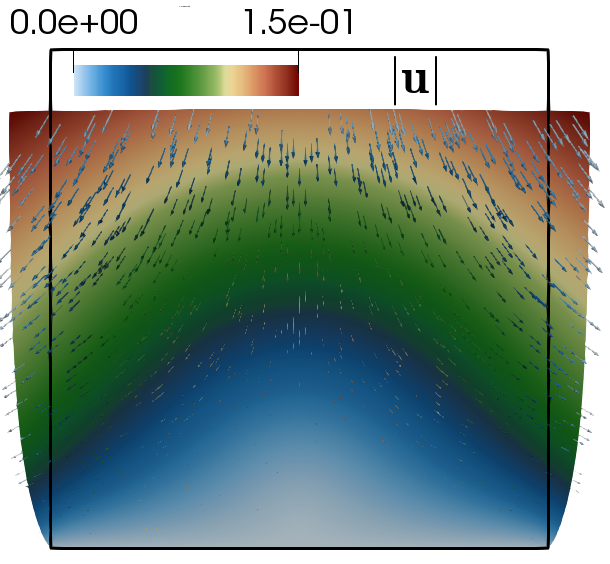}
\includegraphics[width=0.325\textwidth]{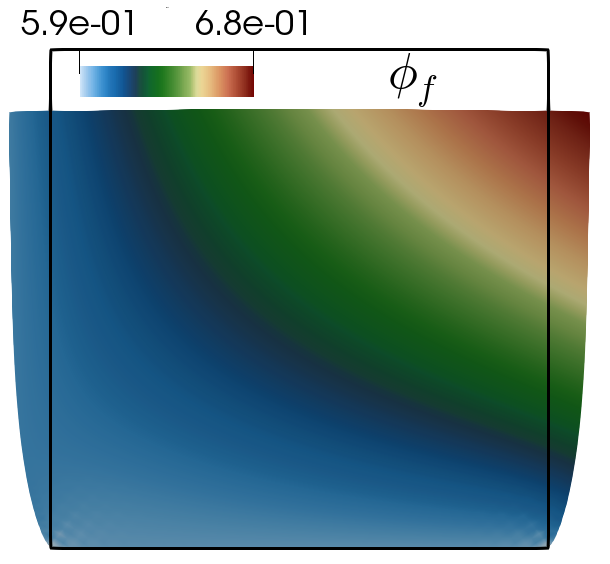}
\includegraphics[width=0.325\textwidth]{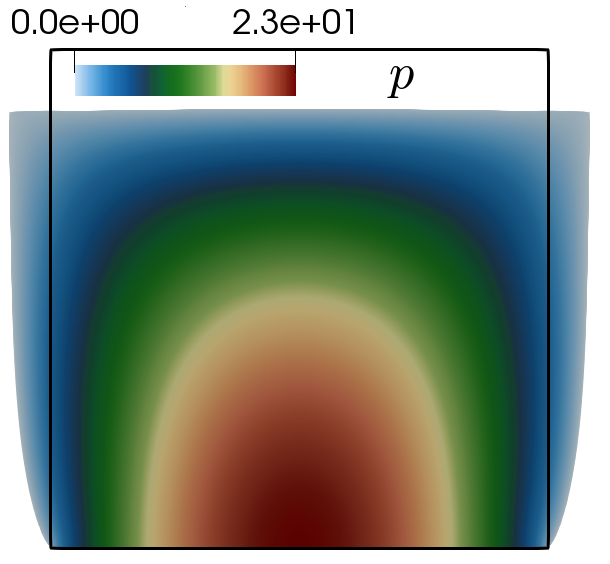}
\end{center}
\vspace{-3mm}
\caption{Example 2: Compression and drainage of a porous structure. Approximate displacement magnitude, Lagrangian porosity, and fluid pressure, on the deformed configuration and computed at $t = 0.5$ with $\nu = 0.33$ (top) and $\nu = 0.495$ (bottom).}
\label{fig:ex01a}
\end{figure} 

With the aim of illustrating the coupling between poroelastic deformation and fluid flow using the formulation \eqref{eq:u}-\eqref{eq:phi} and in the absence of the immune system chemotaxis, we simulate the incremental compression of a box $\Omega = (0,1)^2$\,m$^2$.  The bottom edge of the boundary constitutes $\Gamma$ (where according to \eqref{bc:Gamma}, the deformation is zero and we impose zero fluid flux), whereas $\Sigma$ is conformed by the top edge (where a sinusoidal-in-time distributed traction $\bt_\Sigma = 2000\sin(\pi t)\nn$ is applied and a constant fluid pressure $p_{\mathrm{in}} = 0.2$\,MPa is considered) as well as the vertical walls (where we prescribe zero traction and zero fluid pressure).   We employ a non-homogeneous initial Lagrangian porosity that is not symmetric with respect to the centre of mass of the domain, given by $\phi_{0} (\bx) = \frac{3}{5} + \frac{1}{10}\sin(x_1x_2)$, the isotropic power-law porosity-dependent permeability, a compressible {neo--Hookean} strain energy density yielding the nominal effective stress $\bS_{\mathrm{eff}} =  \mu_s(\bF^{\tt t} - \bF^{-1}) + \lambda_s\mathrm{ln}J\bF^{-1}$, and the parameters
\begin{gather*}
E = 10^4\, \text{Kg/ms}^2, \quad \nu \in [0.2,0.499999], \quad 
\mu_s = \frac{E}{2(1+\nu)}, \quad \lambda_s = \frac{E\nu}{(1+\nu)(1-2\nu)},\quad 
\bb = \cero, \quad \ell = 0 ,  \\
 {\rho_s = 2\cdot10^{-3}\, \text{Kg/m}^3, \quad \rho_f = 10^{-3}\, \text{Kg/m}^3}, \quad 
\alpha = 0.25, \quad \mu_f = 10^{-3}\,\text{m}^2/\text{s},\quad \kappa_0 = 10^{-5}\,\text{m}^2.
\end{gather*}
The scheme characterised by the finite element spaces in \eqref{fe-spaces} is used on a crossed-shaped uniform mesh with 10000 cells, and the coupled system is simulated until $t_{\mathrm{final}} = 1$\,s using a constant time step $\Delta t = 0.1$\,s. The system is considered initially at rest ($\bu(0) = \cero,\ p(0) = 0,\ \phi_f(0) = \phi_0$). This test also serves to assess the robustness of the method. We run a set of simulations varying the Poisson ratio from 0.2 to 0.499999. Regardless of the parameter regime, the numerical results do not show spurious oscillations of pressure, nor unexpectedly small displacements or nonphysical distortions or checker-board patterns in the porosity field. In Figure~\ref{fig:ex01a} we plot deformed geometries and field variables showing the progressive compression of the porous block over three time instants, and using $\nu = 0.33$ and $\nu =0.495$. As expected, the fluid pore pressure does not have a uniform distribution, and it moves toward $\Gamma$ on the bottom edge. 

\begin{figure}[t!]
\begin{center}
\includegraphics[width=0.244\textwidth]{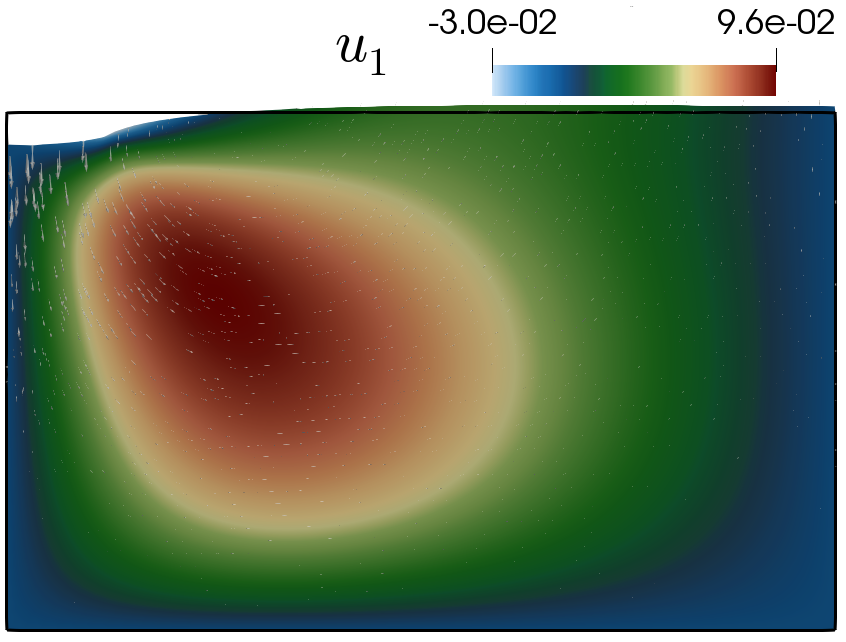}
\includegraphics[width=0.244\textwidth]{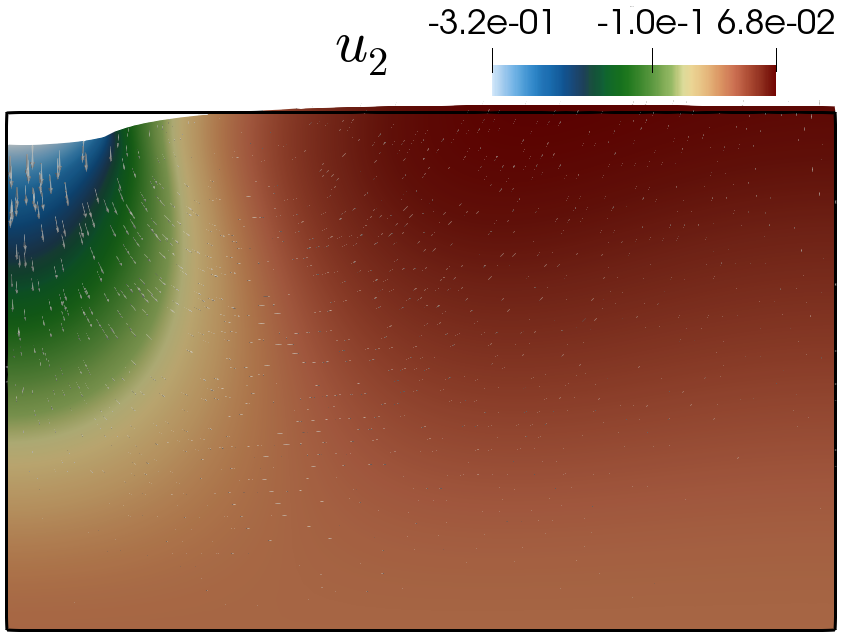}
\includegraphics[width=0.244\textwidth]{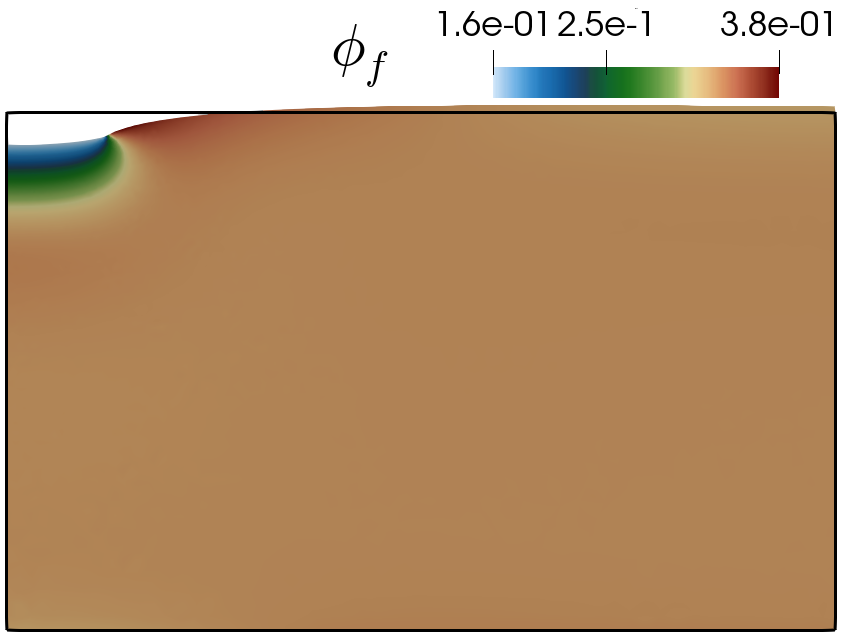}
\includegraphics[width=0.244\textwidth]{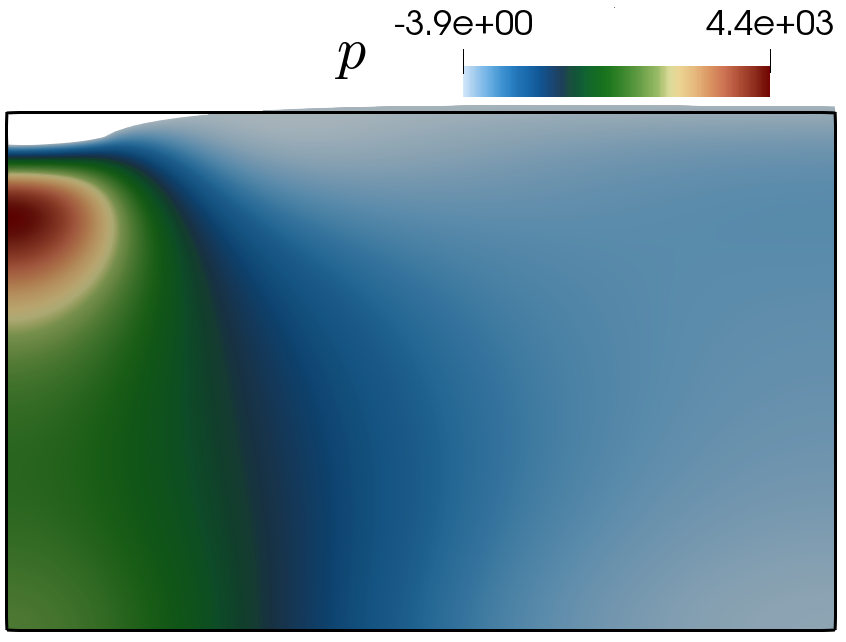}\\
\includegraphics[width=0.244\textwidth]{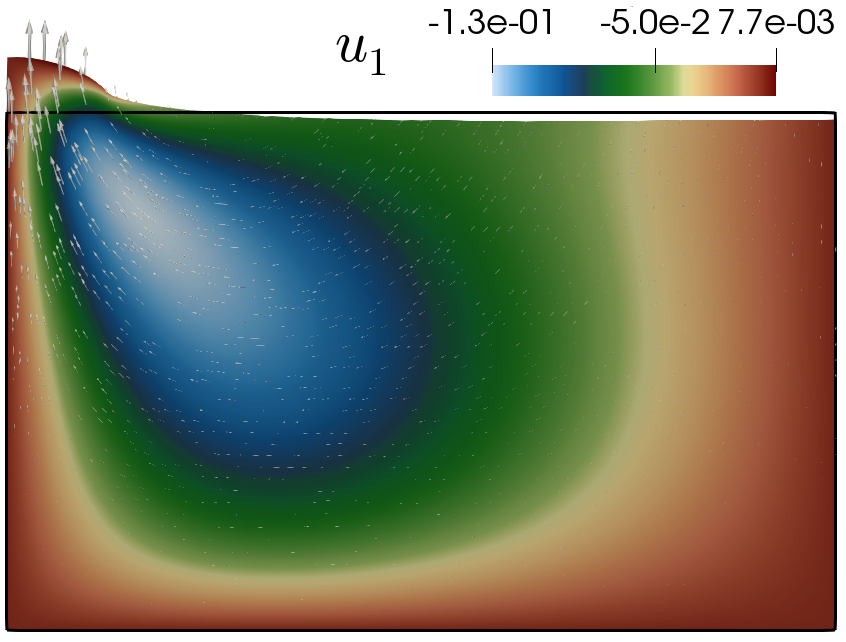}
\includegraphics[width=0.244\textwidth]{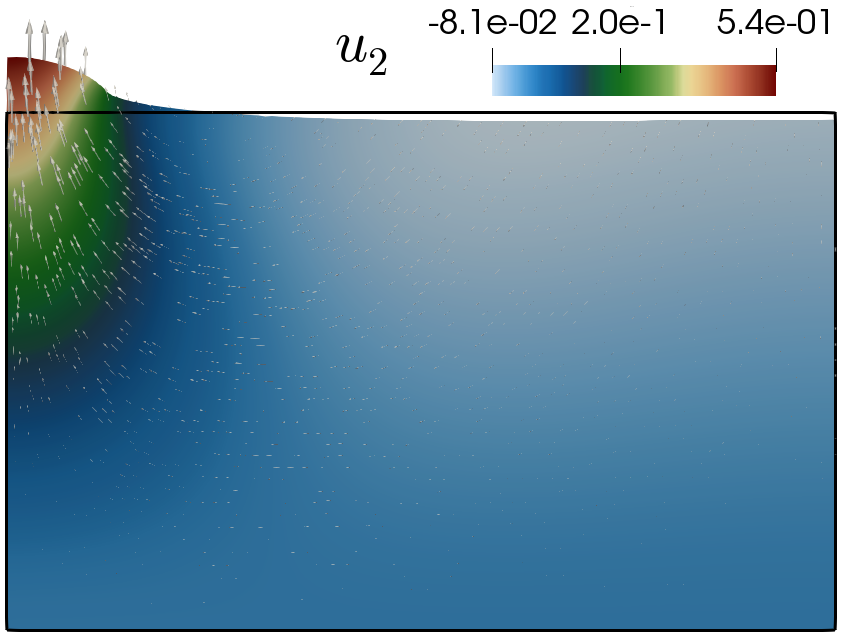}
\includegraphics[width=0.244\textwidth]{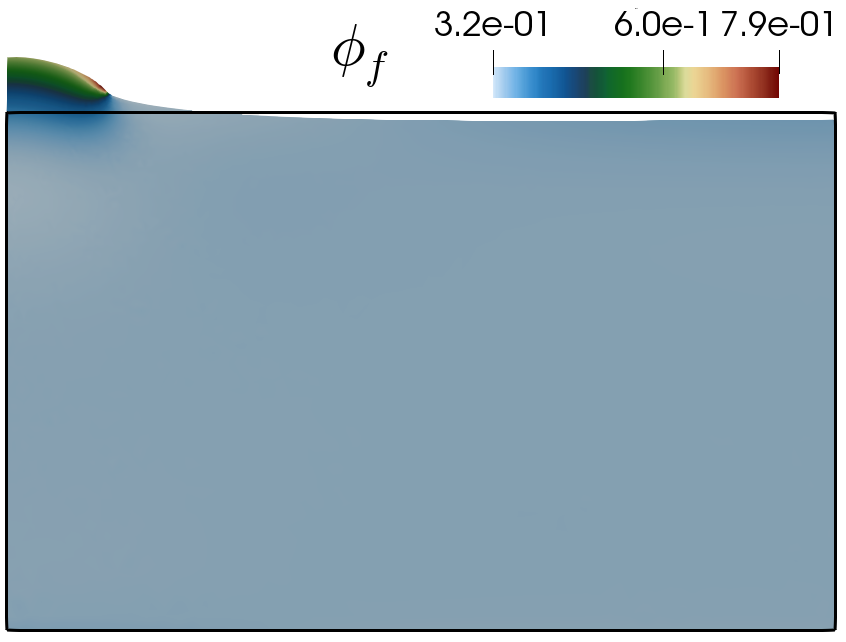}
\includegraphics[width=0.244\textwidth]{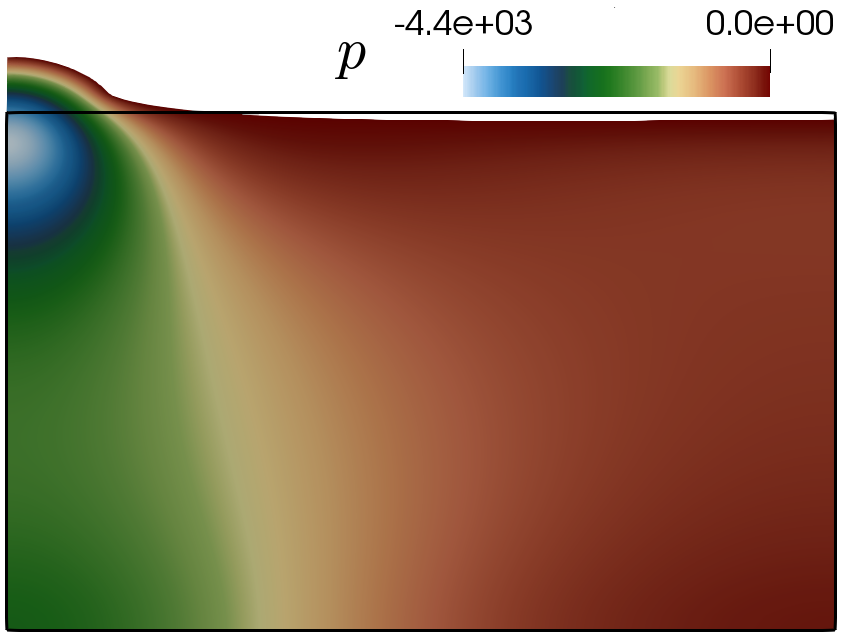}\\
\includegraphics[width=0.495\textwidth]{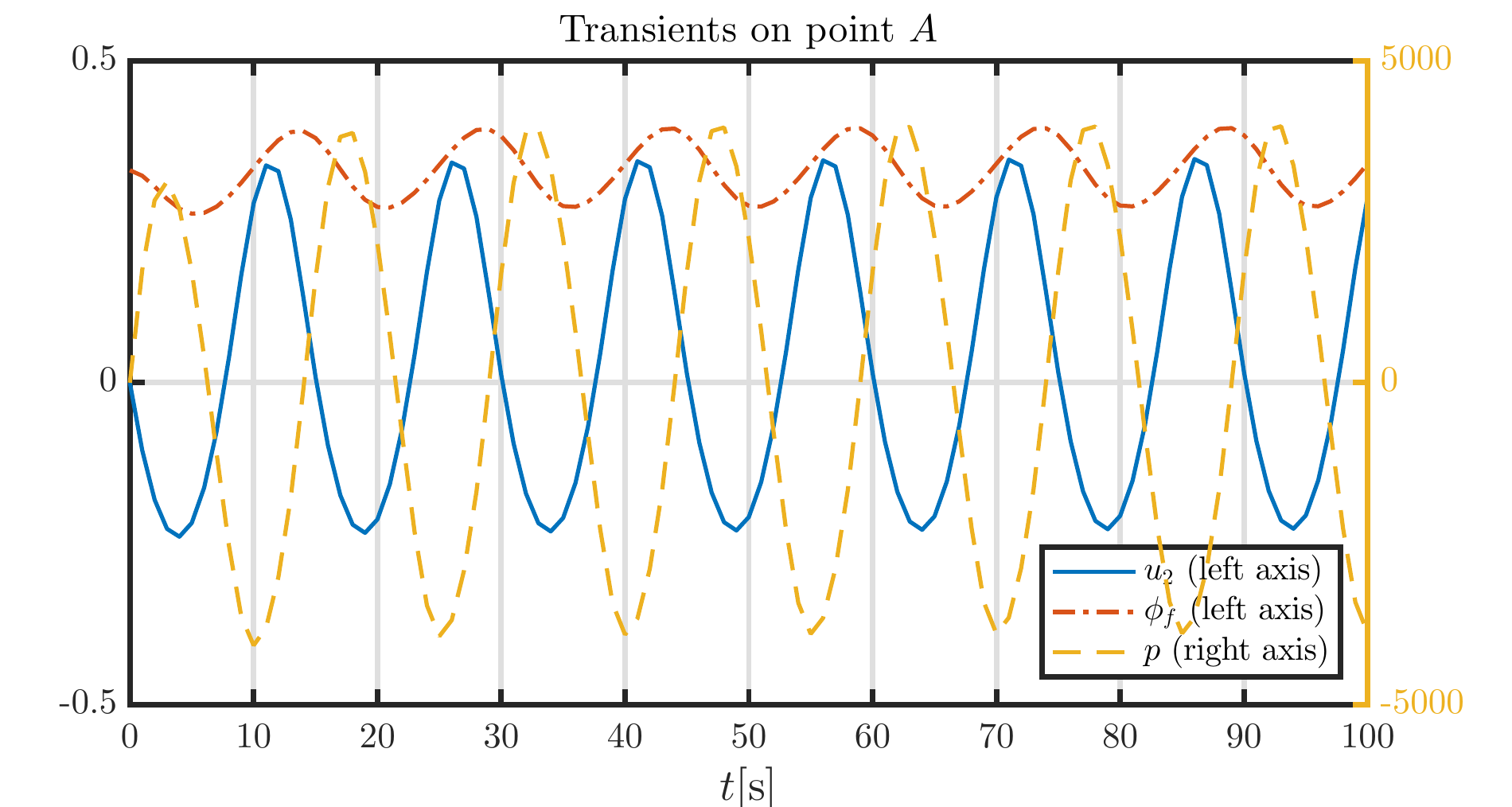}
\includegraphics[width=0.495\textwidth]{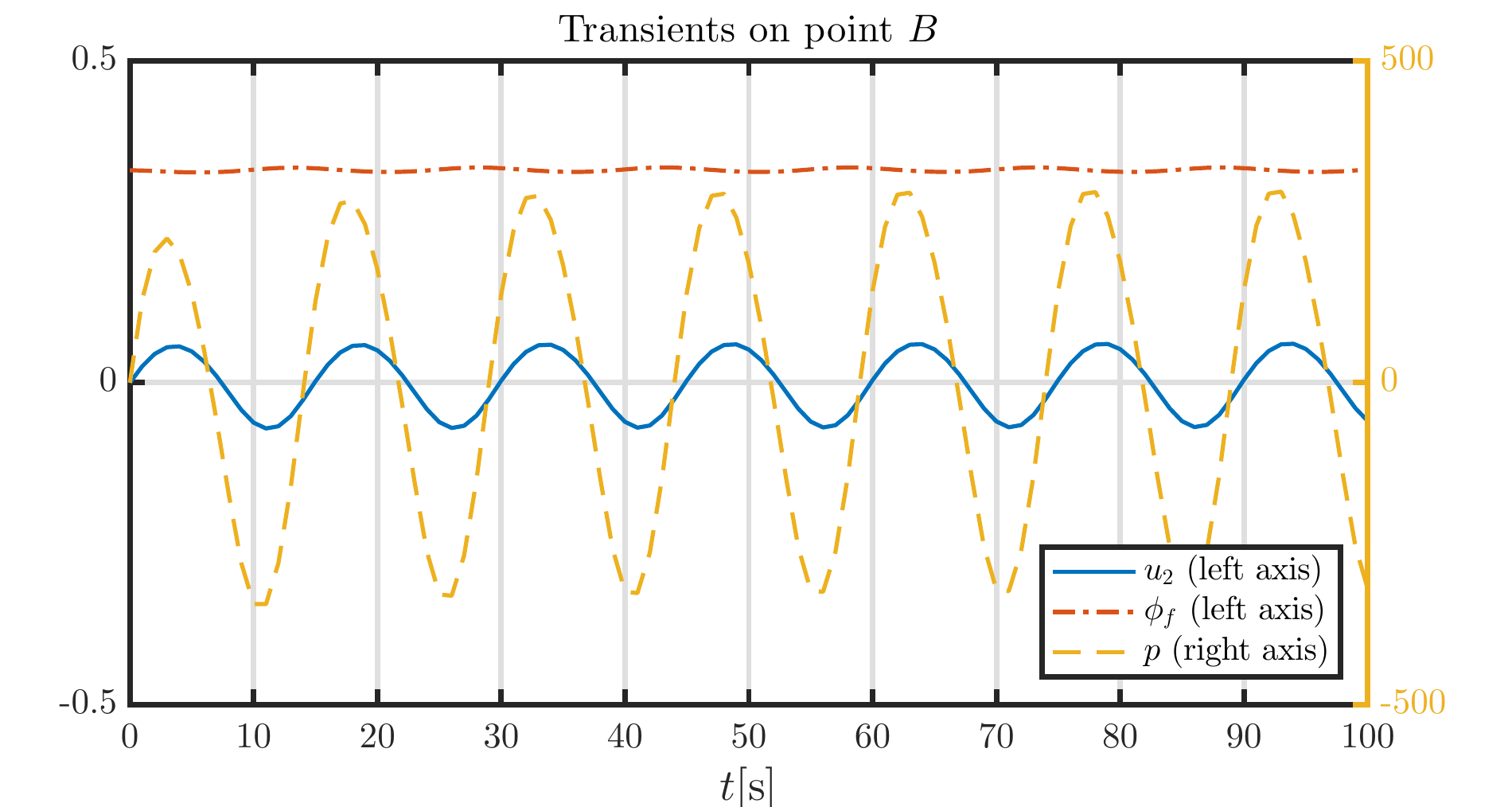}
\end{center}

\vspace{-1mm}
\caption{Example 2: Compression benchmark. Approximate displacement components, Lagrangian porosity, and fluid pressure, on the deformed configuration and computed at $t = 34$ (top) and $t = 41$ (middle), and transients of field variables at points $A$ and $B$ (bottom).}
\label{fig:ex01b}
\end{figure}

In addition, we perform a benchmark test where again a time-harmonic loading defined by $\bt_\Sigma = \frac{1}{5}(\lambda_s+2\mu_s)\sin(\frac{2\pi}{15} t)\nn$ is applied, now only on part of the top edge of the box $\Omega = (0,8)\times(0,5)$\,m$^2$ (on the segment $0\leq x_1\leq 1\,\text{m},\ x_2 = 5$\,m). The setup of this validation example proceeds similarly as in  \cite[Sect. 4.2]{korsawe06} (see also \cite[Sect. 4.2]{zheng20} and \cite{li04,rohan17}). The boundary conditions in these references permit drainage (fixing a zero fluid pressure) on the whole top lid, setting zero fluid flux elsewhere on the boundary, the bottom edge is clamped (imposing zero displacement) while the vertical walls are on  rollers (only the horizontal displacement is set to zero), and the remainder of the top edge is traction-free. The parameter values that change with respect to the previous case are 
\begin{gather*}
E = 3\cdot10^4\, \text{N/m}^2, \ \nu = 0.2, \ 
\alpha = 1, \ \kappa_0 = 10^{-8}\,\text{m}^2, \ \phi_0 = 0.33, \ t_{\mathrm{final}} = 100\,\text{s}, \ 
\Delta t = 1\,\text{s},
\end{gather*}
and we now use the normalised and isotropic {Kozeny--Carman} porosity-dependent permeability. For this test the mesh is unstructured and graded near the top-left corner. Figure~\ref{fig:ex01b} shows snapshots of the approximated field variables at times $t= 4,12,16$\,s and the bottom panels also portray transients of the vertical displacement and porosity recorded at two locations (node A$\approx$(0.5,4.5), directly below the centre of the footing strip and node B$\approx$(8.0,4.5) approximately at the same height but on the right wall). These plots show the same qualitative behaviour encountered in, e.g., \cite{li04}.

\subsection{Example 3: Errors with respect to manufactured solutions}
Since there is no analytical solution currently available for the coupled system \eqref{eq:u}-\eqref{eq:cl}, the accuracy of the finite element discretisation is investigated using the following closed-form manufactured solutions  
\begin{gather}
c_p = t[0.3\exp(x_1)+0.1\cos(\pi x_1)\cos(\pi x_2)],\ 	
c_l = t[0.3\exp(x_1)+0.1\sin(\pi x_1)\sin(\pi x_2)],\nonumber\\
\bu = 0.25t\begin{pmatrix}
	\sin(\pi x_1)\cos(\pi x_2)+\frac{x_1^2}{\lambda_s}\\
	-\cos(\pi x_1)\sin(\pi x_2)+\frac{x_2^2}{\lambda_s}
	\end{pmatrix}, \quad  
	p = t\sin(\pi x_1x_2)\cos(\pi x_1x_2), \label{eq:manuf}\quad  \phi_0 =0.6+0.1\sin(x_1x_2), 
\end{gather}
and $\phi_f = J - 1 + \phi_0$, defined on the unit square domain $\Omega = (0,1)^2$. Together with the strong form of the governing equations, these exact solutions are used to obtain the load and source right-hand side terms. For simplicity we impose Dirichlet boundary conditions for the concentrations $c_p,c_l$ whereas mixed boundary conditions are considered for the displacement and fluid pressure. On the left side of the boundary $\Gamma$, we prescribe the exact displacement from \eqref{eq:manuf} and an exact traction $\bP\nn$ on the remainder of the boundary, $\Sigma$. The fluid pressure is constrained to match the exact one on $\Sigma$ and an exact fluid pressure flux is imposed on $\Gamma$.

\begin{figure}[t!]
\begin{center}
\includegraphics[width=0.5\textwidth]{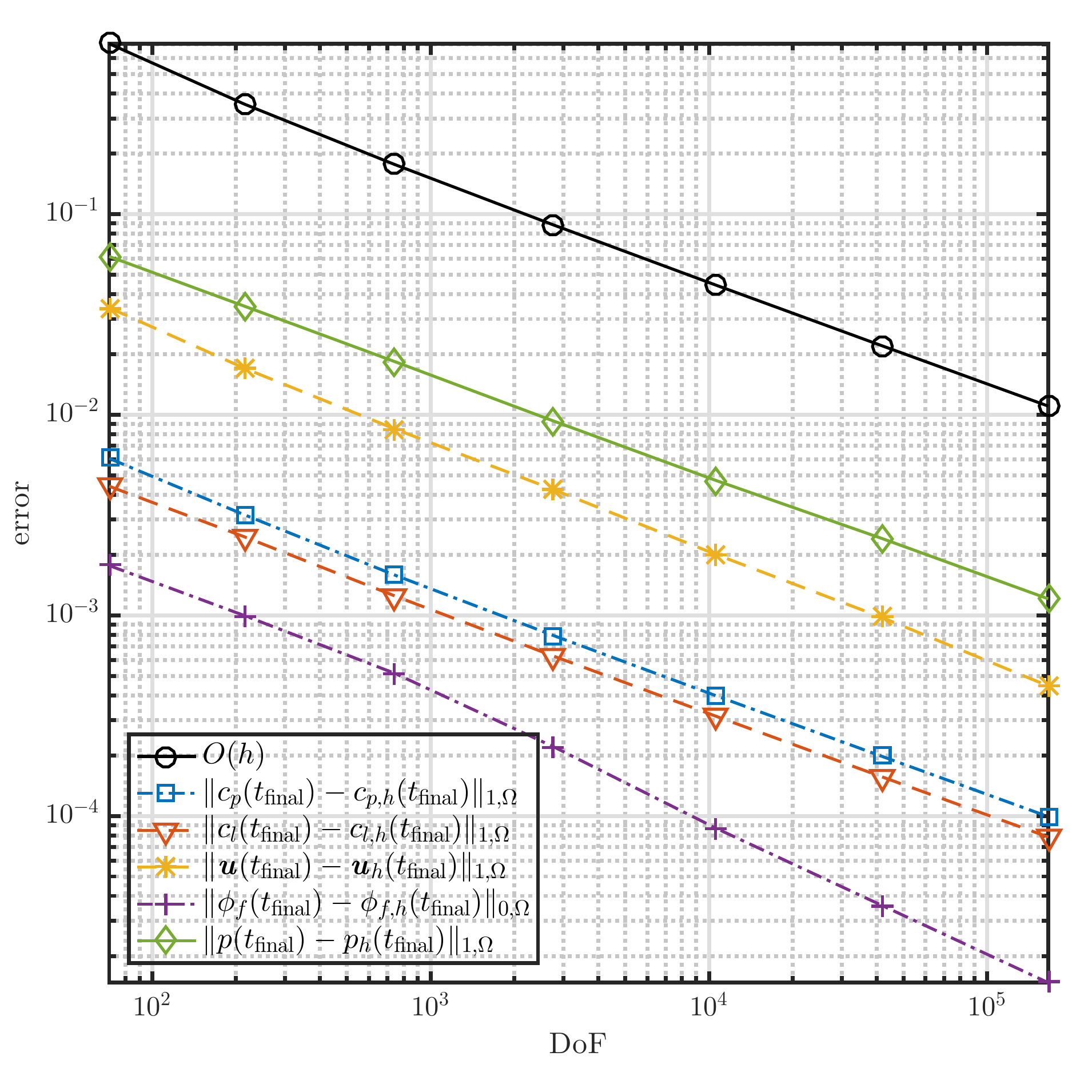}
\end{center}

\vspace{-1mm}
\caption{Example 3: Error decay with respect to the number of degrees of freedom, computed between approximate solutions computed with a first-order method and manufactured solutions \eqref{eq:manuf}, at the final adimensional time $t_{\mathrm{final}} = 0.03$.}\label{fig:ex02}
\end{figure}

We use the same neo--Hookean material law as in Example 1, and employ the isotropic Kozeny--Carman permeability $\bkappa^{\mathrm{KC}}_{\mathrm{iso}}$. The following parameter values are selected 
\begin{gather*}
\lambda_s = 36.4,\quad \mu_s = 22.1, \quad \rho_s = \rho_f = \mu_f =  \kappa_0 = L_{p0} = L_{bp} = L_{br} =v_{\max} = \gamma_p = \gamma_l = c_l^{\max} =1, \quad 
\alpha = 0.5, \\ \bD_p = 0.9\bI, \quad \bD_l =0.8\bI,  
  \quad \sigma_0 = \lambda_{lp} = \lambda_{pl} = \mu_l = l_0 = p_0 = \pi_i = k_m =\pi_c = \chi = p_c =  S/V = 1, \quad 
 n = 2,
\end{gather*}
they are all regarded adimensional and do not have physical relevance in this case, as we will be simply testing the convergence of the finite element solutions. 

We generate successively refined simplicial grids and use a sufficiently small (non dimensional) time step $\Delta t = 0.01$ and simulate a sufficiently short time horizon $t_{\mathrm{final}} = 3\Delta t$, to guarantee that the error produced by the time discretisation does not dominate. Errors between the approximate and exact solutions are plotted against the number of degrees of freedom in Figure~\ref{fig:ex02}. This error history plot confirms the optimal convergence of the finite element scheme (in this case, first order) given by the Remark \ref{remark}, for all variables in their respective norms, where a slightly better rate, of about 1.3, is seen for the porosity. In addition, the average iteration count for the Newton-Raphson method (for each refinement level and for all time steps) is five.

\subsection{Example 4: Localisation of oedema regions}
In the next test we present and discuss the coupling dynamics of the model~\eqref{eq:poroe}-\eqref{eq:immune}, considering a two-dimensional domain $\Omega = (0,4) \times (0,4)$\, {cm$^2$, and a time step size of $\Delta t = 0.1$ days}. The parameters described in the Table~\ref{tab:parameters} were used. Boundary conditions were applied as follows: $\bu = \textbf{0}$ on the left edge, and $p=0$ for the right, top and bottom edges of the domain, and the following initial conditions were used: $c_{p,0} = 0.001 \: \mbox{in the region} \: (x_1 - 2)^2 + (x_2 - 2)^2 \le 0.03$ and $c_{p,0} = 0.0$, otherwise. Also, we consider $c_{l,0} = 0.003$, $\phi_{0} = 0.2$ and $p_0 = 0$ in the entire domain. Figure~\ref{fig:2DPadrao} presents the results of a localised oedema formation. Each row presents snapshots of the variables $c_p$, $c_l$, $p$, $\bu$ and $\phi_f$, respectively, at three selected time instants.

The dynamics of the immune response can be observed in Figure~\ref{fig:2DPadrao} in terms of $c_p$ and $c_l$. After the pathogens appear in the tissue, they start to grow (Figure~\ref{fig:2DPadrao} at time $t = 8$ days) and reach their peak concentration (Figure~\ref{fig:2DPadrao} at time $t = 13$ days). However, in response to these events, leukocytes migrate to the infected site and their concentration also grows (Figure~\ref{fig:2DPadrao}, $c_l$ column). As a consequence of the presence of leukocytes, pathogen concentration starts to decrease, as Figure~\ref{fig:2DPadrao} shows for $c_l$ at time $t = 18$ days.

The presence of pathogens triggers a mechanical response due to a sequence of couplings in the model. The endothelium increases its permeability to allow leukocytes to leave the bloodstream and enter the tissue. This increase in the endothelium permeability in turn increases interstitial fluid and pressure (Figure~\ref{fig:2DPadrao} for $\phi_f$ and $p$ at time $t = 8$ days). When the concentration of pathogens reaches its peak (Figure~\ref{fig:2DPadrao} for $c_p$ at time $t = 13$ days), it is possible to observe that the amount of fluid phase $\phi_f$ in the tissue also reaches its maximum value. An increase of liquid in the region leads to an increase in both pressure and displacement fields (Figure~\ref{fig:2DPadrao} for $\bu$ and $p$ at time $t = 13$ days). Finally, when pathogens $c_p$ are almost eliminated by the leukocytes at time $t = 18$ days, the dynamics of pressure, fluid phase fraction and displacements, tend to return to their initial values.

\begin{figure}[!t] 
\begin{center}
\includegraphics[width=.95\textwidth]{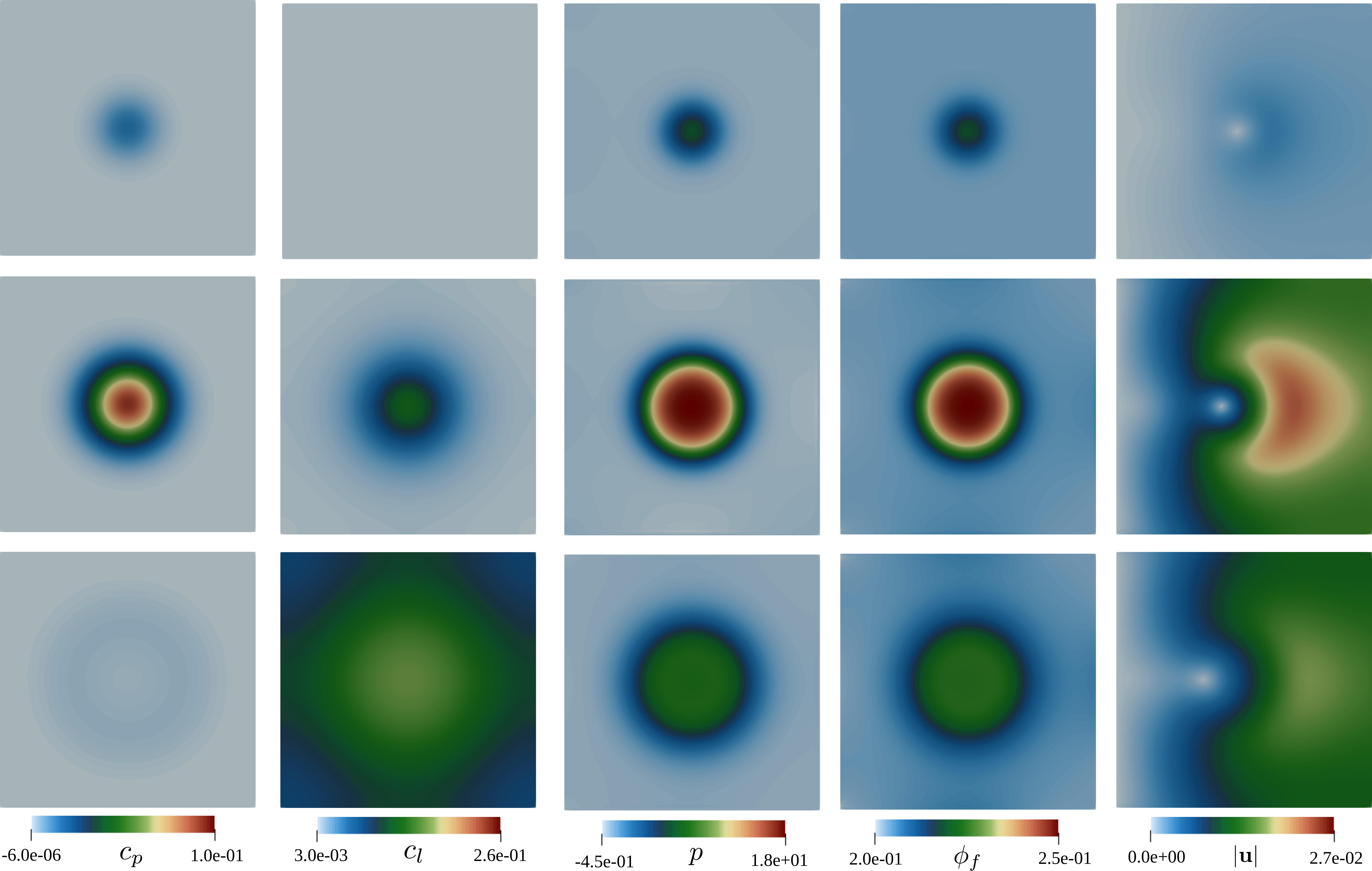}
\end{center}
\caption{Example 4. Behaviour of pathogen (first column) and leukocyte (second column) concentrations, pressure field (third column), fraction of fluid phase (fourth column) and displacement field (fifth column) at time $t = 8$ (top row), $t = 13$ (middle row) and $t = 18$ (bottom row), after the solution of the model \eqref{eq:u}-\eqref{eq:phi},\eqref{eq:cp}-\eqref{eq:cl}, considering the parameters given by Table~\ref{tab:parameters} and initial conditions given by $c_{p,0} = 0.001$, $c_{l,0} = 0.003$, $\phi_{0} = 0.2$, $\bu_0 = \cero$, and $p_0 = 0$.
}
\label{fig:2DPadrao}
\end{figure}

\subsection{Example 5: Immune system dynamics and poro-hyperelasticity in a left ventricle} \label{sec:ventr}
Next we conduct a series of tests using a ventricular geometry segmented from patient-specific images \cite{warriner18}, and where synthetic muscle fibre and collagen sheet directions are constructed using a Laplace-Dirichlet rule-based method \cite{rossi14}. On the basal surface we prescribe zero normal displacement and zero fluid pressure flux, the epicardium is considered stress-free and with a prescribed fluid pressure, and on the endocardium we impose a time-dependent traction $\bt_\Sigma = m_0\sin^2(\pi t)\nn$ with $m_0=0.1$\,N/cm$^2$, and an endocardial fluid pressure $p_\mathrm{endo} = p_0\sin^2(\pi t)$ having the same time period. No-flux conditions are used for the concentrations on the whole boundary.

\begin{figure}[t!]
\begin{center}
\includegraphics[width=0.244\textwidth]{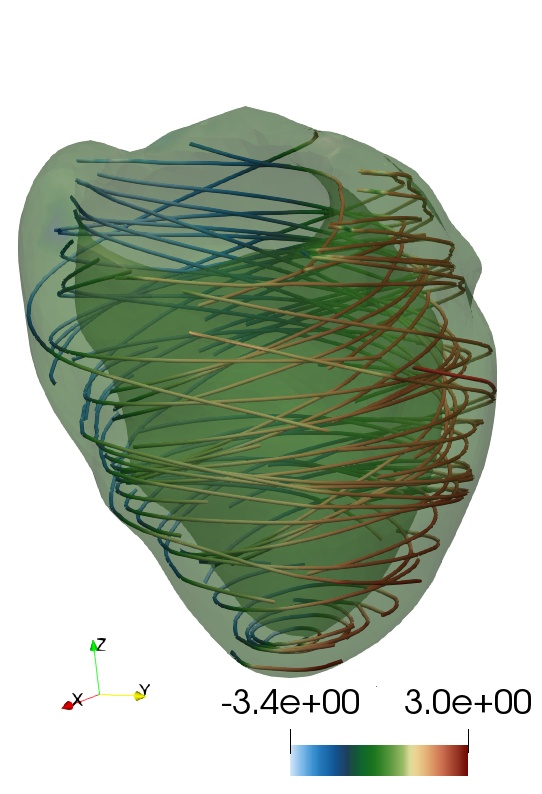}
\includegraphics[width=0.244\textwidth]{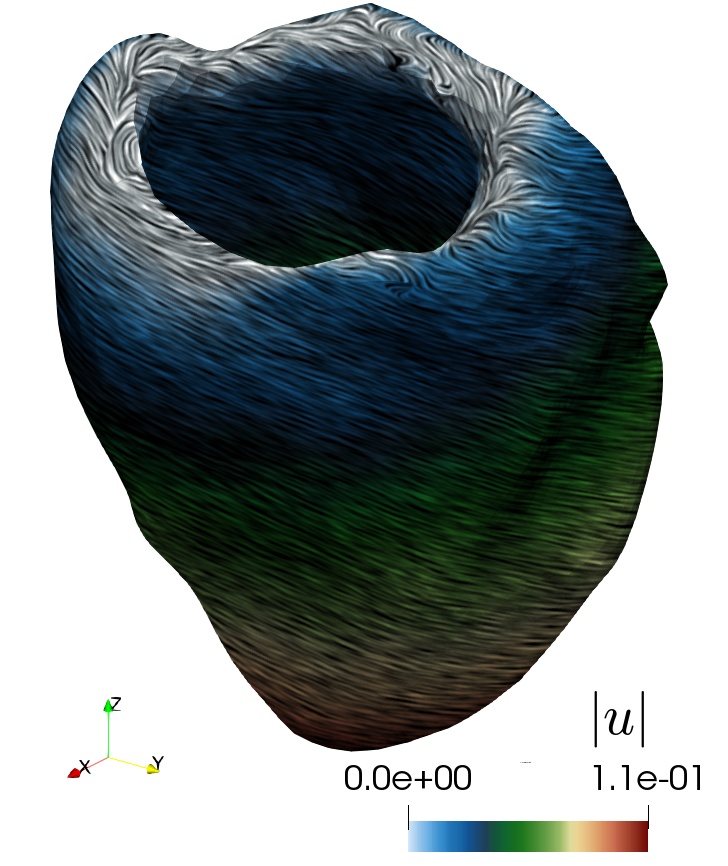}
\includegraphics[width=0.244\textwidth]{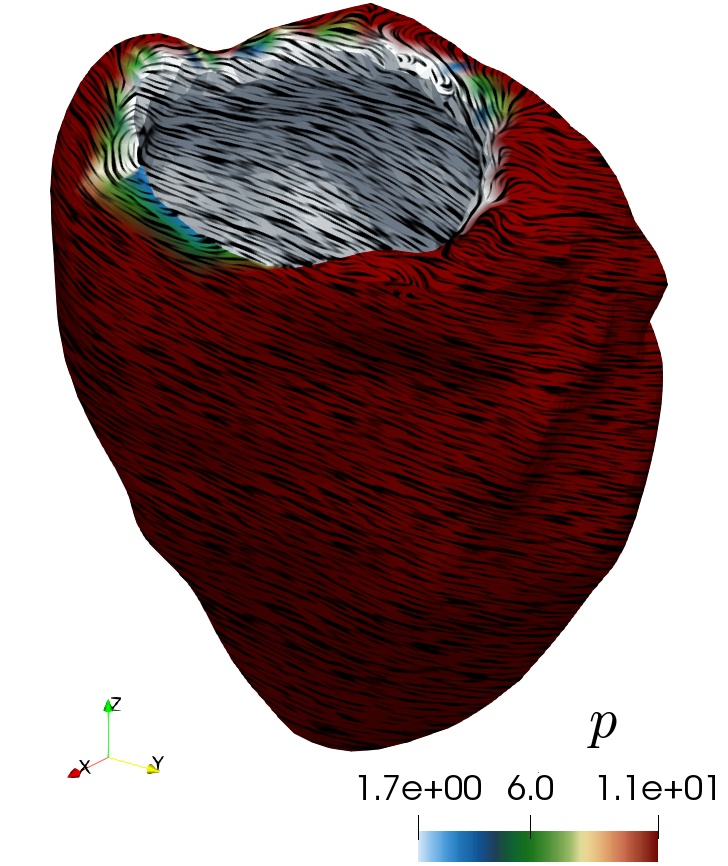}
\includegraphics[width=0.244\textwidth]{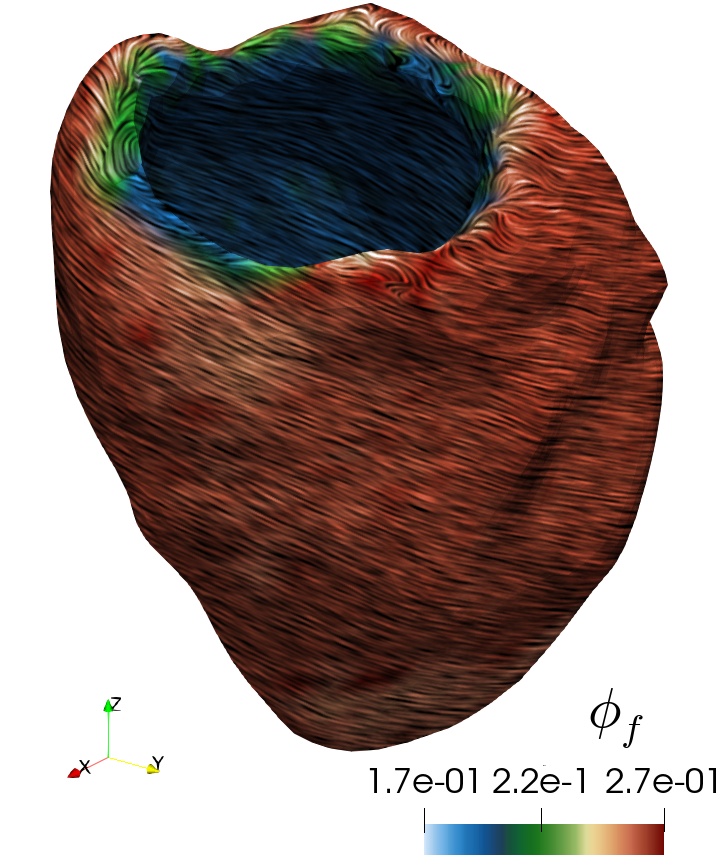}\\
\includegraphics[width=0.244\textwidth]{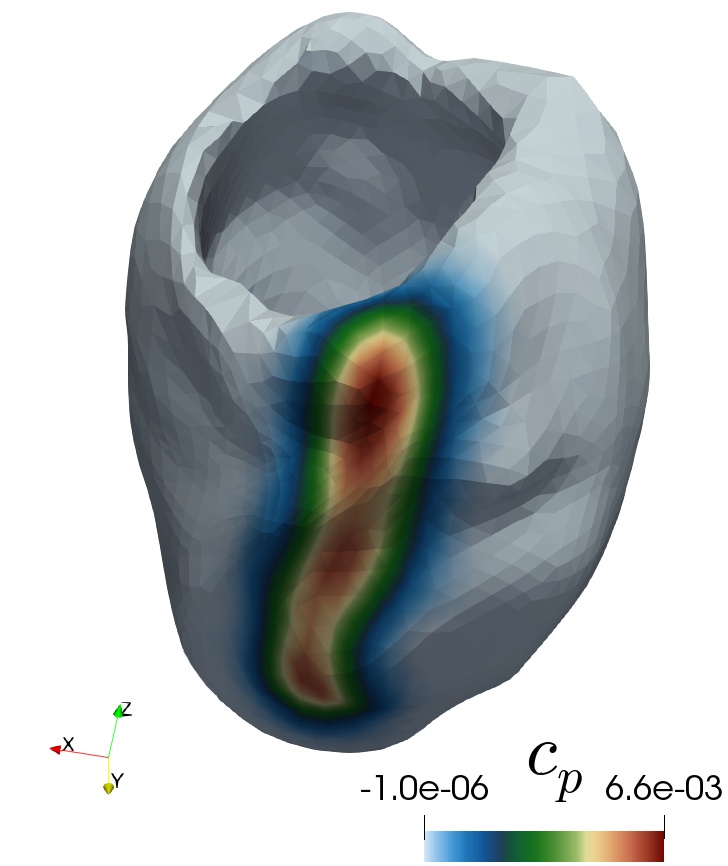}
\includegraphics[width=0.244\textwidth]{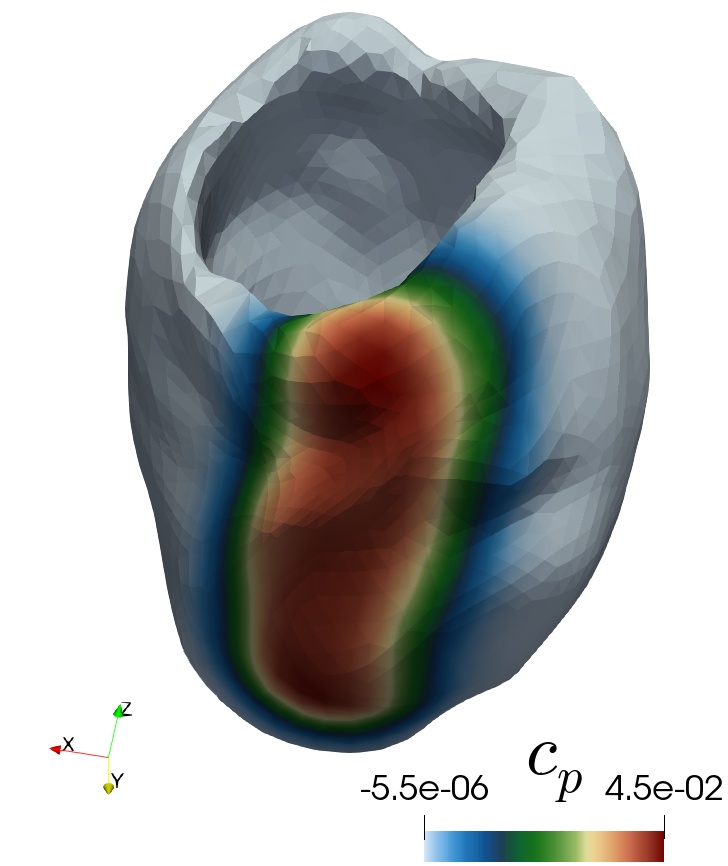}
\includegraphics[width=0.244\textwidth]{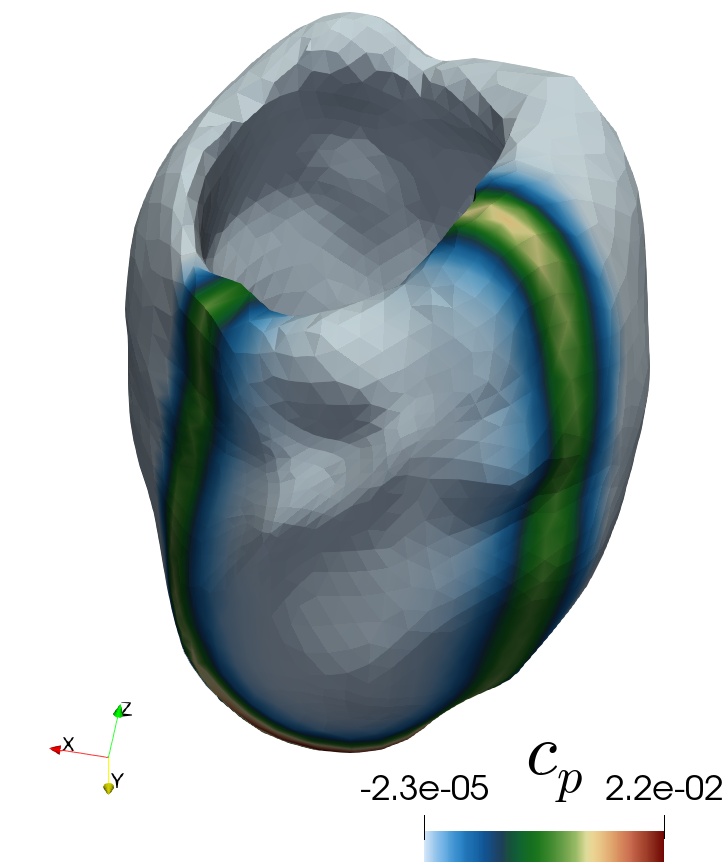}
\includegraphics[width=0.244\textwidth]{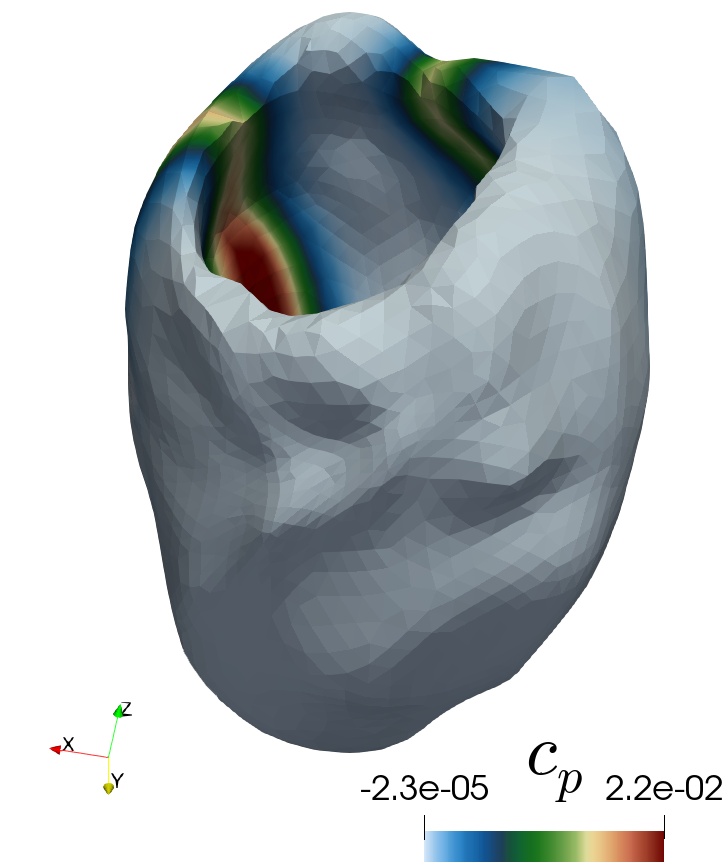}\\
\includegraphics[width=0.244\textwidth]{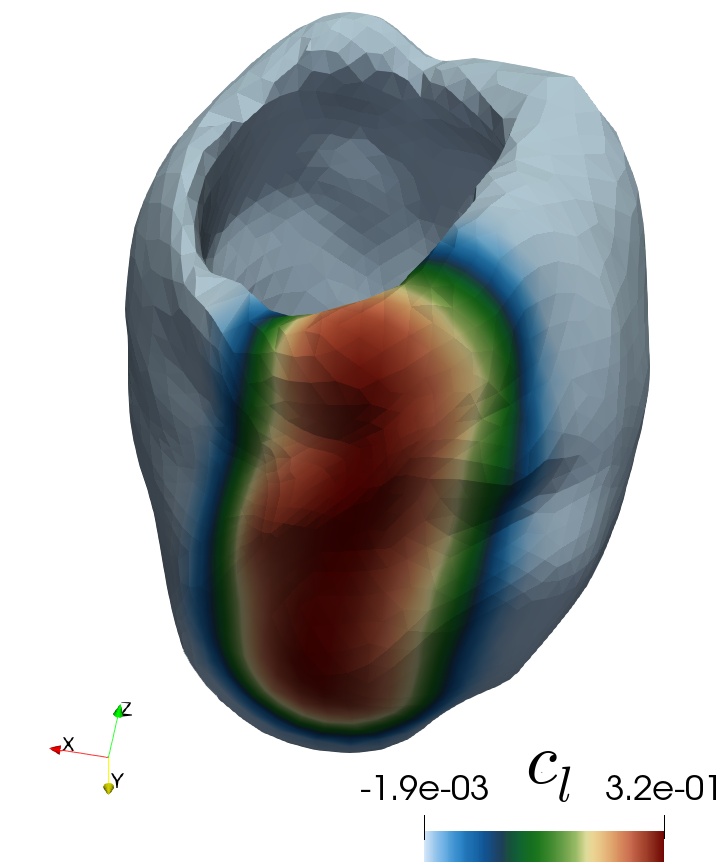}
\includegraphics[width=0.244\textwidth]{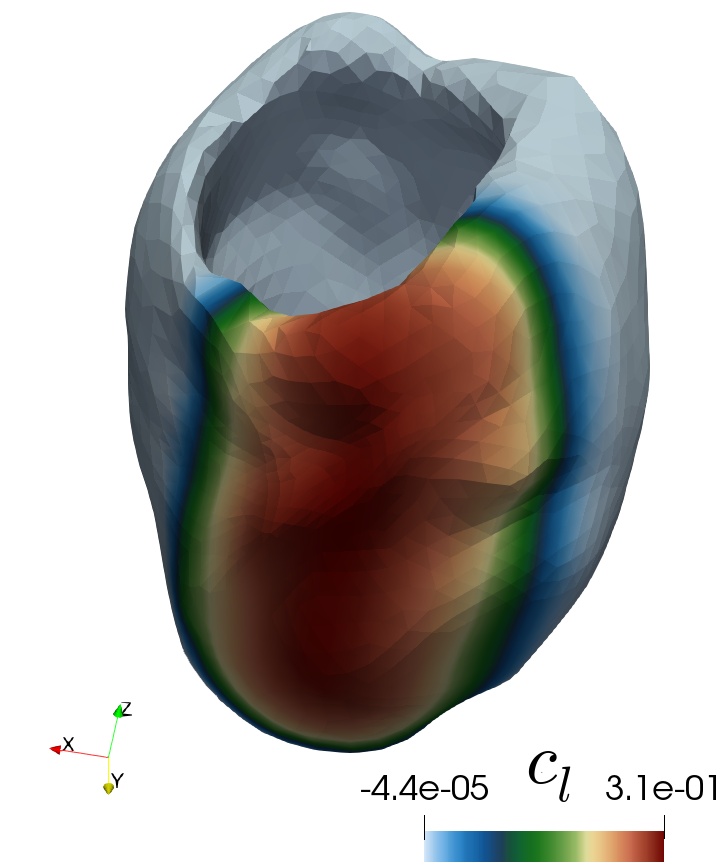}
\includegraphics[width=0.244\textwidth]{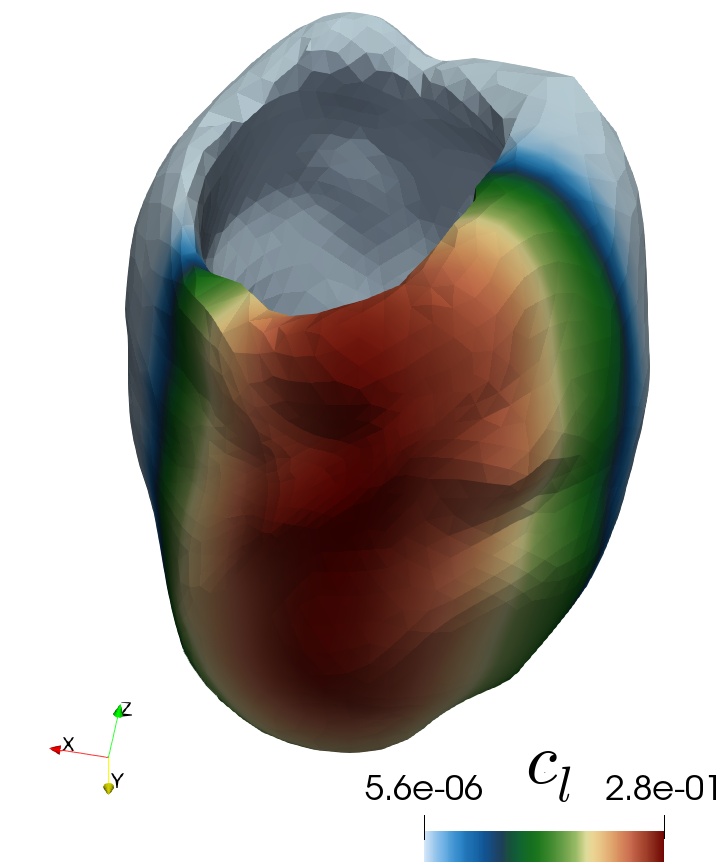}
\includegraphics[width=0.244\textwidth]{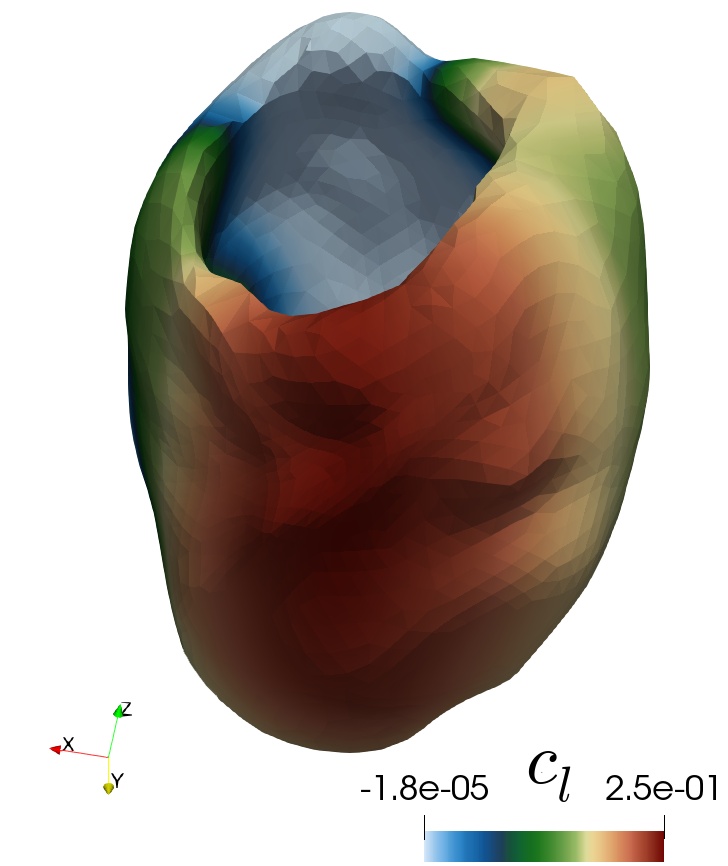}
\end{center}

\vspace{-1mm}
\caption{Example 5. Assigned fibre distribution streamlines, displacement magnitude, fluid pressure distribution, and porosity on a patient-specific left ventricular geometry at $t=14$ (top), and evolution of pathogens and leukocytes concentration (middle and bottom rows, respectively, seen from a slightly different angle).}
\label{fig:ex03a}
\end{figure}

For this example we use the Holzapfel-Ogden constitutive strain energy stated in \eqref{eq:HO}, and employ the anisotropic Kozeny--Carman permeability $\bkappa^{\mathrm{KC}}$ introduced in \eqref{eq:kappa}; while the remaining model parameters used for the 3D ventricular test assume the values 
\begin{gather*}
\lambda_s = 27.293\,\text{kPa},\quad 
\mu_s = 3.103\,\text{kPa}, \quad 
\bD_p = 3\times10^{-3}/\phi_0\,\text{cm$^2/$h}, \quad 
\pi_i = 10\,\text{mmHg},\quad n = 5, \quad \alpha = 0.5,\\ \bD_l = 5\times10^{-2}/\phi_0\,\text{cm$^2/$h}, \quad 
v_{\max} = 20, \quad
\sigma_0 = 0.91, \quad
\phi_0 = 0.2, \quad 
p_0 = 10.9\,\text{mmHg}, \quad  
k_m =6.5\,\text{mmHg}, \\ 
\pi_c = p_c = 20\,\text{mmHg}, \quad {S/V = 174}, \quad 
\chi = 10^{-4}\,\text{cm$^3$/(h.$10^{7}$cell)}, \quad  
\gamma_p = 0.13/\phi_0\,\text{cm$^3$/(h.$10^{7}$cell)}, \\ 
\lambda_{pl} = 7.1/\phi_0, \quad \lambda_{lp} = 1.8/\phi_0, \quad 
l_0 = 0\,\text{cm/s}, \quad  
L_{bp} = 5000, \quad L_{p0} = 3.6\times10^{-8}, \quad \mu_f = 10^{-3}\,\text{cm$^2$/s}, \\  
\kappa_0 = 2.5\times10^{-7}\,\text{cm$^2$},  \quad \rho_s = 2\times10^{-3}\,\text{Kg/cm$^3$}, \quad 
\rho_f = 10^{-3}\,\text{Kg/cm$^3$},\quad  
a = 0.496\,\text{N/cm$^2$}, \quad b = 0.041,\\ 
 a_f = 0.193\,\text{N/cm$^2$}, \quad b_f = 0.176, \quad  
a_s = 0.123\,\text{N/cm$^2$},\quad b_s = 0.209, \quad 
a_{fs}=0.162\,\text{N/cm$^2$}, \quad b_{fs} = 0.166.
\end{gather*}

An initial concentration of pathogens is considered on a transmural strip, and we simulate the dynamical behaviour of the coupled system over several minutes. The evolution of the pathogens and leukocytes  distribution on the deformed ventricular geometry is shown in Figure~\ref{fig:ex03a}, where we also show fibre distribution and mechanical fields at half-time. {One can observe that the presence of pathogens induces the local entry of leukocytes. Fluid also enters the interstitial space of the tissue, increasing pressure. Although leukocytes can remove part of the pathogens, some of them diffuse through almost the entire domain. As this pathogens' wave sweeps a significant part of the left ventricle, more parts of the heart are impacted by the inflammatory response and, as a consequence, diffuse oedema is formed. Numerical simulations show that the propagation of the front of pathogens likely depends on the diffusion and replication of pathogens and the wave tail likely depends on the diffusion and efficiency of the leukocytes\cite{lourenco22}. The numerical results indicate that, although fluid phase and pressure increase due to the presence of the pathogen, they result in small changes in displacement. In fact, new diagnostic tools have been studied to increase the non-invasive diagnostic accuracy of myocarditis and other myocardial pathologies using for this purpose the detection of increased water contents at some regions instead of an increase in tissue displacement. See, for example, \cite{kim2017,spieker2017t2,spieker2017abnormal}.} 





To conclude this section, for the ventricular geometry we proceed to briefly test the preconditioner described in Section \ref{sec:preconditioner}. To measure its performance, we look at the total and average number of Krylov iterations in two different meshes corresponding to the left ventricle geometry and with respect to $\mathbb P_1$ and $\mathbb P_1^b$ (bubble enriched element) conforming finite elements for the displacement to study the impact of the inf-sup stability in the performance. The results are shown in Table \ref{tab:preconditioner}, and as we have observed that numerically the most difficult time instant is the first one, we report the performance only on it.  

We highlight that the number of average Krylov iterations is roughly constant with an increasing number of cores. This can be expected due to the use of inexact solvers in the sub-blocks of the Schur complement preconditioner, which yield an adequate approximation of the exact Schur preconditioner, which has at most 3 distinct eigenvalues. It is particularly interesting to observe that this performance is obtained by using only the action of an AMG preconditioner for the chemotaxis block, meaning that (i) the coupling between the chemotaxis and poroelastic models is not very strong in passive cardiac simulations and that (ii) such block is not necessarily computationally challenging.  

\begin{table}[t!]
    \centering
    \begin{subtable}[h]{\textwidth}
        \centering
        \begin{tabular}{r | l l l l l l}
        \toprule DoFs & 1 cpu & 2 cpu & 4 cpu & 8 cpu & 16 cpu& 32 cpu \\
        \midrule
        28924 & 48 (6.86) & 88 (12.57) & 81 (11.57) & 87 (12.43) & 80 (11.43) & 90 (12.86)\\
        155351 & 74 (6.73) & 139 (12.64) & 122 (12.20) & 103 (12.88) & 89 (12.71) & 93 (13.29) \\ \bottomrule
       \end{tabular}
       \caption{Displacement approximated with $\mathbb P_1$ elements.}
       \label{tab:preconditioner-p1}
    \end{subtable}
    
    \begin{subtable}[h]{\textwidth}
        \centering
        \begin{tabular}{r | l l l l l l}
        \toprule DoFs & 1 cpu& 2 cpu& 4 cpu& 8 cpu& 16 cpu& 32 cpu \\
        \midrule
        72772 & 74 (6.73) & 139 (12.64) & 122 (12.20) & 103 (12.88) & 89 (12.71) & 93 (13.29)\\
        438158 & 50 (6.25) & 88 (11.00) & 83 (11.86) & 85 (10.62) & 77 (11.00) & 80 (11.43)\\ \bottomrule
       \end{tabular}
       \caption{Displacement approximated with $\mathbb P_1^b$ elements.}
       \label{tab:preconditioner-p2}
    \end{subtable}
    \caption{Example 5. Iteration counts, total and (average), for fGMRES with the proposed preconditioner, and for either continuous and piecewise linear approximations of displacement, or piecewise linear displacements with bubble enrichment. Note that an increase in the total iterations, for constant average iterations, occurs due to a different number of Newton iterations.}
    \label{tab:preconditioner}
\end{table}

\section{Final remarks}\label{sec:concl}
We have studied a general model capturing the phenomenological features of the interaction between chemotaxis of the immune system in saturated poroelastic media admitting large deformations. The essential properties of the new model include a formulation in terms of displacement, fluid pressure, and Lagrangian porosity, coupled with concentrations for leukocytes and pathogens. In particular, the poro-hyperelastic compartment of the model can be identified, after linearisation and adequately choosing the initial guess, with the three-field formulation for Biot's poroelasticity from \cite{oyarzua16,verma21}. We have proposed {a finite}  element method together with splitting strategies, and have studied their properties in terms of accuracy and iteration count. The realisation of the coupling is general enough to accommodate different types of model specifications, including diverse hyperelastic solid laws and showing the interplay of solid deformation, effective stress, and pore fluid pressure build-up.  The model also includes species transport through the total velocity, and the system is solved by means of a monolithic finite element method in saddle-point form. 

Some extensions in the development of this work include the study of long-term behaviour of the oedema formation, as well as performing a thorough exploration of different effects from introducing lymphatic sinks for $c_p$ and $c_l$. We also plan to incorporate active stress due to calcium release (in turn generated by the immune reaction-diffusion) and assess quantitative differences in the classical diffusion case. In a more applicative context, a more comprehensive sensitivity analysis is still required to determine the most relevant coupling mechanisms. Also, the validation and verification of the model against patient-specific data are still to be conducted. {We have expanded the application discussed in  Section~\ref{sec:ventr} to address further modelling and numerically oriented investigations in the recent work \cite{lourenco22}}. 
Regarding numerical schemes, a possible next step is to use mixed formulations for the immune system equations (following, e.g., \cite{Gatica2021}) to produce mass conservative discretisations. 

Further investigation is necessary, for instance, regarding the specific form of the anisotropic porosity as well as in designing new coupling mechanisms  that will contribute  to a better understanding of the formation and termination of myocarditis and myocardial oedema. 

\bigskip 
\noindent\textbf{Acknowledgements.} This work has been supported by Universidade Federal de Juiz de Fora (UFJF) through the scholarship ``Coordenação de Aperfeiçoamento de Pessoal de Nível Superior'' (CAPES) - Brazil - Finance Code 001; by Conselho Nacional de Desenvolvimento Científico e Tecnológico (CNPq) - Brazil {grant numbers 423278/2021-5, 308745/2021-3 and 310722/2021-7}; by Fundação de Amparo à Pesquisa do Estado de Minas Gerais (FAPEMIG) - Brazil CEX APQ 02830/17, TEC APQ 03213/17, {and TEC APQ 01340/18}; by the Monash Mathematics Research Fund S05802-3951284; {by the Australian Research Council through the Discovery Project grant DP220103160}; and by the Ministry of Science and Higher Education of the Russian Federation within the framework of state support for the creation and development of World-Class Research Centres ``Digital biodesign and personalised healthcare'' No. 075-15-2020-926.

\bigskip
\noindent\textbf{Declarations.} 
{The authors declare that the research was conducted in the absence of any commercial or financial 
relationships that could be construed as a potential conflict of interest.}

\bigskip
\noindent\textbf{Data availability.} 
{Part of the datasets and finite element implementations generated during and/or analysed during the current study are available from the \texttt{GitHub} repository} \url{https://github.com/ruizbaier/PoroelasticModelForAcuteMyocarditis}.

\bibliographystyle{siam}
\bibliography{Bibl.bib}

\begin{thebibliography}{10}

\bibitem{alnaes15}
{\sc M.~S. Aln{\ae}s, J.~Blechta, J.~Hake, A.~Johansson, B.~Kehlet, A.~Logg,
  C.~Richardson, J.~Ring, M.~E. Rognes, and G.~N. Wells}, {\em The {FE}ni{CS}
  project version 1.5}, Arch. Numer. Software, 3 (2015), pp.~9--23.

\bibitem{alves2019}
{\sc J.~R. Alves, R.~A.~B. de~Queiroz, M.~Bär, and R.~W. dos Santos}, {\em
  Simulation of the perfusion of contrast agent used in cardiac magnetic
  resonance: A step toward non-invasive cardiac perfusion quantification},
  Frontiers in Physiology, 10 (2019), p.~177.

\bibitem{alves16}
{\sc J.~R. Alves, R.~A.~B. de~Queiroz, and R.~Weber~dos Santos}, {\em
  Simulation of cardiac perfusion by contrast in the myocardium using a
  formulation of flow in porous media}, J. Comput. Appl. Math., 295 (2016),
  pp.~13--24.

\bibitem{amestoy2000mumps}
{\sc P.~R. Amestoy, I.~S. Duff, J.-Y. L’Excellent, and J.~Koster}, {\em
  {MUMPS}: a general purpose distributed memory sparse solver}, in
  International Workshop on Applied Parallel Computing, Springer, 2000,
  pp.~121--130.

\bibitem{anaya18}
{\sc V.~Anaya, M.~Bendahmane, D.~Mora, and R.~Ruiz-Baier}, {\em On a
  vorticity-based formulation for reaction-diffusion-brinkman systems},
  Networks \& Heterogeneous Media, 13 (2018), pp.~69--94.

\bibitem{Arnold1984}
{\sc D.~N. Arnold, F.~Brezzi, and M.~Fortin}, {\em A stable finite element for
  the {S}tokes equations}, Calcolo, 21 (1984), pp.~337--344.

\bibitem{ateshian10}
{\sc G.~A. Ateshian and J.~A. Weiss}, {\em Anisotropic hydraulic permeability
  under finite deformation}, J. Biomech. Engrg., 132 (2010), p.~111004(7).

\bibitem{Auricchio2005}
{\sc F.~Auricchio, L.~Beir\~{a}o~da Veiga, C.~Lovadina, and A.~Reali}, {\em A
  stability study of some mixed finite elements for large deformation
  elasticity problems}, Comput. Methods Appl. Mech. Engrg., 194 (2005),
  pp.~1075--1092.

\bibitem{baerland2017weakly}
{\sc T.~B{\ae}rland, J.~J. Lee, K.-A. Mardal, and R.~Winther}, {\em Weakly
  imposed symmetry and robust preconditioners for {B}iot’s consolidation
  model}, Computational Methods in Applied Mathematics, 17 (2017),
  pp.~377--396.

\bibitem{deoliveira20b}
{\sc N.~Barnafi, L.~M. De~Oliveira~Vilaca, M.~C. Milinkovitch, and
  R.~Ruiz-Baier}, {\em Coupling chemotaxis and poromechanics for the modelling
  of feather primordia patterning}, In preparation,  (2021), pp.~1--28.

\bibitem{barnafi2021multiscale}
{\sc N.~Barnafi, S.~Di~Gregorio, L.~Dede', P.~Zunino, C.~Vergara, and A.~M.
  Quarteroni}, {\em A multiscale poromechanics model integrating myocardial
  perfusion and systemic circulation}, MOX Reports,  (2021).

\bibitem{barnafi20}
{\sc N.~Barnafi, P.~Zunino, L.~Ded\`e, and A.~Quarteroni}, {\em Mathematical
  analysis and numerical approximation of a general linearized
  poro-hyperelastic model}, Comput. Math. Appl., 91 (2021), pp.~202--228.

\bibitem{basser92}
{\sc P.~J. Basser}, {\em Interstitial pressure, volume, and flow during
  infusion into brain tissue}, Microvascular research, 44 (1992), pp.~143--165.

\bibitem{berger16}
{\sc L.~Berger, R.~Bordas, K.~Burrowes, V.~Grau, S.~Tavener, and D.~Kay}, {\em
  A poroelastic model coupled to a fluid network with applications in lung
  modelling}, Int. J. Numer. Methods Biomed. Engrg., 32 (2016), p.~e02731.

\bibitem{berger17}
{\sc L.~Berger, R.~Bordas, D.~Kay, and S.~Tavener}, {\em A stabilized finite
  element method for finite-strain three-field poroelasticity}, Comput. Mech.,
  60 (2017), pp.~51--68.

\bibitem{borregales2018robust}
{\sc M.~Borregales, F.~A. Radu, K.~Kumar, and J.~M. Nordbotten}, {\em Robust
  iterative schemes for non-linear poromechanics}, Computational Geosciences,
  22 (2018), pp.~1021--1038.

\bibitem{Both2017}
{\sc J.~Both, M.~Borregales, J.~Nordbotten, K.~Kumar, and F.~Radu}, {\em Robust
  fixed stress splitting for {B}iot's equations in heterogeneous media},
  Applied Mathematics Letters, 68 (2017), pp.~101--108.

\bibitem{Both2019}
{\sc J.~{Both}, K.~{Kumar}, J.~{Nordbotten}, and F.~{Radu}}, {\em {The gradient
  flow structures of thermo-poro-visco-elastic processes in porous media}},
  arXiv e-prints,  (2019).

\bibitem{both20}
{\sc J.~W. Both, N.~A. Barnafi, F.~A. Radu, P.~Zunino, and A.~Quarteroni}, {\em
  Iterative splitting schemes for a soft material poromechanics model},
  Computer Methods in Applied Mechanics and Engineering, 388 (2022),
  p.~e114183.

\bibitem{brenner}
{\sc S.~C. Brenner and L.~R. Scott}, {\em The mathematical theory of finite
  element methods}, Texts in Applied Mathematics, Springer-Verlag, New York,
  2002.

\bibitem{burger97}
{\sc R.~L. Burger and K.~Belitz}, {\em Measurement of anisotropic hydraulic
  conductivity in unconsolidated sands: {A} case study from a shoreface
  deposit, {Oyster, Virginia}}, Water Res. Research, 33 (1997), pp.~1515--1522.

\bibitem{burtschell2017}
{\sc B.~Burtschell, D.~Chapelle, and P.~Moireau}, {\em Effective and
  energy-preserving time discretization for a general nonlinear poromechanical
  formulation}, Computers \& Structures, 182 (2017), pp.~313--324.

\bibitem{Chamberland2010}
{\sc E.~Chamberland, A.~Fortin, and M.~Fortin}, {\em Comparison of the
  performance of some finite element discretizations for large deformation
  elasticity problems}, Comput. Struct., 88 (2010), pp.~664--673.

\bibitem{chapelle10}
{\sc D.~Chapelle, J.-F. Gerbeau, J.~Sainte-Marie, and I.~E. Vignon-Clementel},
  {\em A poroelastic model valid in large strains with applications to
  perfusion in cardiac modeling}, Comput. Mech., 46 (2010), pp.~91--101.

\bibitem{Chapelle201482}
{\sc D.~Chapelle and P.~Moireau}, {\em General coupling of porous flows and
  hyperelastic formulations - from thermodynamics principles to energy balance
  and compatible time schemes}, European Journal of Mechanics, B/Fluids, 46
  (2014), pp.~82--96.

\bibitem{cherubini17}
{\sc C.~Cherubini, S.~Filippi, A.~Gizzi, and R.~Ruiz-Baier}, {\em A note on
  stress-driven anisotropic diffusion and its role in active deformable media},
  J. Theor. Biol., 430 (2017), pp.~221--228.

\bibitem{choo19}
{\sc J.~Choo}, {\em Large deformation poromechanics with local mass
  conservation: {A}n enriched {G}alerkin finite element framework}, Int. J.
  Numer. Methods Engrg., 116 (2019), pp.~66--90.

\bibitem{cion19}
{\sc A.~Cioncolini and D.~Boffi}, {\em The {MINI} mixed finite element for the
  {S}tokes problem: An experimental investigation}, Comput. Math. Appl., 77
  (2019), pp.~2432--2446.

\bibitem{colli14}
{\sc P.~Colli~Franzone, L.~F. Pavarino, and S.~Scacchi}, {\em {Mathematical
  Cardiac Electrophysiology}}, vol.~13, Springer International Publishing,
  2014.

\bibitem{cookson2012novel}
{\sc A.~Cookson, J.~Lee, C.~Michler, R.~Chabiniok, E.~Hyde, D.~Nordsletten,
  M.~Sinclair, M.~Siebes, and N.~Smith}, {\em A novel porous mechanical
  framework for modelling the interaction between coronary perfusion and
  myocardial mechanics}, Journal of Biomechanics, 45 (2012), pp.~850--855.

\bibitem{Coussy2004}
{\sc O.~Coussy}, {\em Poromechanics}, John Wiley \& Sons, 2004.

\bibitem{deoliveira20}
{\sc L.~M. De~Oliveira~Vilaca, B.~G{\'o}mez-Vargas, S.~Kumar, R.~Ruiz-Baier,
  and N.~Verma}, {\em Stability analysis for a new model of multi-species
  convection-diffusion-reaction in poroelastic tissue}, Appl. Math. Model., 84
  (2020), pp.~425--446.

\bibitem{deparis2016facsi}
{\sc S.~Deparis, D.~Forti, G.~Grandperrin, and A.~Quarteroni}, {\em {FaCSI: A}
  block parallel preconditioner for fluid--structure interaction in
  hemodynamics}, Journal of Computational Physics, 327 (2016), pp.~700--718.

\bibitem{ehret17}
{\sc A.~E. Ehret, K.~Bircher, A.~Stracuzzi, V.~Marina, M.~Z{\"u}ndel, and
  E.~Mazza}, {\em Inverse poroelasticity as a fundamental mechanism in
  biomechanics and mechanobiology}, Nature Comm., 10 (2017), p.~e1002.

\bibitem{elman2008taxonomy}
{\sc H.~Elman, V.~E. Howle, J.~Shadid, R.~Shuttleworth, and R.~Tuminaro}, {\em
  A taxonomy and comparison of parallel block multi-level preconditioners for
  the incompressible {N}avier--{S}tokes equations}, Journal of Computational
  Physics, 227 (2008), pp.~1790--1808.

\bibitem{farrell21}
{\sc P.~E. Farrell, L.~F. Gatica, B.~P. Lamichhane, R.~Oyarz\'ua, and
  R.~Ruiz-Baier}, {\em Mixed {K}irchhoff stress - displacement - pressure
  formulations for incompressible hyperelasticity}, Comput. Methods Appl. Mech.
  Engrg., 374 (2021), p.~e113562.

\bibitem{federico07}
{\sc S.~Federico and W.~Herzog}, {\em On the anisotropy and inhomogeneity of
  permeability in articular cartilage}, Biomech. Model. Mechanobiol., 7 (2007),
  pp.~367--378.

\bibitem{franceschini2021approximate}
{\sc A.~Franceschini, N.~Castelletto, and M.~Ferronato}, {\em Approximate
  inverse-based block preconditioners in poroelasticity}, Computational
  Geosciences, 25 (2021), pp.~701--714.

\bibitem{reis19b}
{\sc R.~Freitas~Reis, J.~L. Fernandes, T.~R. Schmal, B.~Martins~Rocha, R.~Weber
  Dos~Santos, and M.~Lobosco}, {\em A personalized computational model of edema
  formation in myocarditis based on long-axis biventricular {MRI} images}, BMC
  Bioinformatics, 20 (2019), p.~532(11).

\bibitem{reis19}
{\sc R.~Freitas~Reis, R.~Weber Dos~Santos, B.~Martins~Rocha, and M.~Lobosco},
  {\em On the mathematical modeling of inflammatory edema formation}, Comput.
  Math. Appl., 78 (2019), pp.~2994--3006.

\bibitem{friedrich2009}
{\sc M.~G. Friedrich, U.~Sechtem, J.~Schulz-Menger, G.~Holmvang, P.~Alakija,
  L.~T. Cooper, J.~A. White, H.~Abdel-Aty, M.~Gutberlet, S.~Prasad, A.~Aletras,
  J.-P. Laissy, I.~Paterson, N.~G. Filipchuk, A.~Kumar, M.~Pauschinger, and
  P.~Liu}, {\em Cardiovascular magnetic resonance in myocarditis: A {JACC}
  white paper}, Journal of the American College of Cardiology, 53 (2009),
  pp.~1475--1487.

\bibitem{frigo2019relaxed}
{\sc M.~Frigo, N.~Castelletto, and M.~Ferronato}, {\em A relaxed physical
  factorization preconditioner for mixed finite element coupled poromechanics},
  SIAM Journal on Scientific Computing, 41 (2019), pp.~B694--B720.

\bibitem{Gatica2021}
{\sc G.~N. Gatica, R.~Oyarz\'ua, R.~Ruiz-Baier, and Y.~D. Sobral}, {\em Banach
  spaces-based analysis of a fully-mixed finite element method for the
  steady-state model of fluidized beds}, Comput. Math. Appl., 84 (2021),
  pp.~244--276.

\bibitem{girault79}
{\sc V.~Girault and P.-A. Raviart}, {\em Finite Element Approximation of the
  Navier-Stokes Equation}, Lecture Notes in Math. 749, Springer-Verlag, Berlin,
  New York, 1979.

\bibitem{holzapfel09}
{\sc G.~A. Holzapfel and R.~W. Ogden}, {\em Constitutive modelling of passive
  myocardium: a structurally based framework for material characterization},
  Phil. Trans. Royal Soc. Lond. A, 367 (2009), pp.~3445--3475.

\bibitem{hong2019conservative}
{\sc Q.~Hong, J.~Kraus, M.~Lymbery, and F.~Philo}, {\em Conservative
  discretizations and parameter-robust preconditioners for biot and
  multiple-network flux-based poroelasticity models}, Numerical Linear Algebra
  with Applications, 26 (2019), p.~e2242.

\bibitem{kim2017}
{\sc P.~K. Kim, Y.~J. Hong, D.~J. Im, Y.~J. Suh, C.~H. Park, J.~Y. Kim,
  S.~Chang, H.-J. Lee, J.~Hur, Y.~J. Kim, et~al.}, {\em Myocardial t1 and t2
  mapping: techniques and clinical applications}, Korean journal of radiology,
  18 (2017), pp.~113--131.

\bibitem{kirby2018solver}
{\sc R.~C. Kirby and L.~Mitchell}, {\em Solver composition across the
  {PDE}/linear algebra barrier}, SIAM Journal on Scientific Computing, 40
  (2018), pp.~C76--C98.

\bibitem{korsawe06}
{\sc J.~Korsawe, G.~Starke, W.~Wang, and O.~Kolditz}, {\em Finite element
  analysis of poro-elastic consolidation in porous media: {S}tandard and mixed
  approaches}, Comput. Methods Appl. Mech. Engrg., 195 (2006), pp.~1096--1115.

\bibitem{lang16}
{\sc G.~E. Lang, D.~Vella, S.~L. Waters, and A.~Goriely}, {\em {Mathematical
  modelling of blood-brain barrier failure and oedema}}, Math. Med. Biol., 34
  (2016), pp.~391--414.

\bibitem{li04}
{\sc C.~Li, R.~I. Borja, and R.~A. Regueiro}, {\em Dynamics of porous media at
  finite strain}, Comput. Methods Appl. Mech. Engrg., 193 (2004),
  pp.~3837--3870.

\bibitem{lourenco22}
{\sc W.~d.~J. Louren\c{c}o, R.~F. Reis, R.~Ruiz-Baier, B.~M. Rocha,
  R.~Weber~dos Santos, and M.~Lobosco}, {\em A poroelastic approach for
  modelling myocardial oedema in acute myocarditis}, Frontiers in Physiology,
  13 (2022), pp.~e888515(1--14).

\bibitem{macminn16}
{\sc C.~W. MacMinn, E.~R. Dufresne, and J.~S. Wettlaufer}, {\em Large
  deformations of a soft porous material}, Phys. Rev. Appl., 5 (2016),
  p.~044020(30).

\bibitem{mauk03}
{\sc R.~T. Mauck, C.~T. Hung, and G.~A. Ateshian}, {\em Modelling of neutral
  solute transport in a dynamically loaded porous permeable gel: implications
  for articular cartilage biosynthesis and tissue engineering}, J. Biomech.
  Engrg., 125 (2003), pp.~602--614.

\bibitem{Mikelic2013convergence}
{\sc A.~Mikeli{\'c} and M.~F. Wheeler}, {\em Convergence of iterative coupling
  for coupled flow and geomechanics}, Computational Geosciences, 17 (2013),
  pp.~455--461.

\bibitem{moee13}
{\sc E.~Moeendarbary, L.~Valon, M.~Fritzsche, A.~R. Harris, D.~A. Mouling,
  A.~J. Thrasher, E.~Stride, L.~Mahadevan, and G.~T. Charras}, {\em The
  cytoplasm of living cells behaves as a poroelastic material}, Nature
  Materials, 12 (2013), p.~3517.

\bibitem{murphy2000note}
{\sc M.~F. Murphy, G.~H. Golub, and A.~J. Wathen}, {\em A note on
  preconditioning for indefinite linear systems}, SIAM Journal on Scientific
  Computing, 21 (2000), pp.~1969--1972.

\bibitem{nash00}
{\sc M.~P. Nash and P.~J. Hunter}, {\em Computational mechanics of the heart.
  {F}rom tissue structure to ventricular function}, J. Elasticity, 61 (2000),
  pp.~113--141.

\bibitem{nedjar13}
{\sc B.~Nedjar}, {\em Formulation of a nonlinear porosity law for fully
  saturated porous media at finite strains}, J. Mech. Phys. Solids, 61 (2013),
  pp.~537--556.

\bibitem{oyarzua16}
{\sc R.~Oyarz\'ua and R.~Ruiz-Baier}, {\em Locking-free finite element methods
  for poroelasticity}, SIAM J. Numer. Anal., 54 (2016), pp.~2951--2973.

\bibitem{phipps11}
{\sc C.~Phipps and M.~Kohandel}, {\em Mathematical model of the effect of
  interstitial fluid pressure on angiogenic behavior in solid tumors},
  Computational and Mathematical Methods in Medicine, 2011 (2011), p.~e.843765.

\bibitem{puntmann2018}
{\sc V.~O. Puntmann, A.~M. Zeiher, and E.~Nagel}, {\em {T1 and T2 mapping in
  myocarditis: seeing beyond the horizon of Lake Louise criteria and
  histopathology}}, Expert Review of Cardiovascular Therapy, 16 (2018),
  pp.~319--330.

\bibitem{quarteroni17}
{\sc A.~Quarteroni, T.~Lassila, S.~Rossi, and R.~Ruiz-Baier}, {\em Integrated
  heart -- coupled multiscale and multiphysics models for the simulation of the
  cardiac function}, Comput. Methods Appl. Mech. Engrg., 314 (2017),
  pp.~345--407.

\bibitem{quarteroni94}
{\sc A.~Quarteroni and A.~Valli}, {\em Numerical Approximation of Partial
  Differential Equations}, vol.~23, Springer Ser. Comput. Math. Springer-Verlag
  Berlin Heidelberg, 1994.

\bibitem{rohan17}
{\sc E.~Rohan and V.~Luke\v{s}}, {\em Modeling large-deforming fluid-saturated
  porous media using an {E}ulerian incremental formulation}, Adv. Engrg.
  Softw., 113 (2017), pp.~84--95.

\bibitem{rossi14}
{\sc S.~Rossi, T.~Lassila, R.~Ruiz-Baier, A.~Sequeira, and A.~Quarteroni}, {\em
  Thermodynamically consistent orthotropic activation model capturing
  ventricular systolic wall thickening in cardiac electromechanics}, European
  Journal of Mechanics: A/Solids, 48 (2014), pp.~129--142.

\bibitem{ruiz15}
{\sc R.~Ruiz-Baier}, {\em Primal-mixed formulations for reaction-diffusion
  systems on deforming domains}, J. Comput. Phys., 299 (2015), pp.~320--338.

\bibitem{saad1993flexible}
{\sc Y.~Saad}, {\em A flexible inner-outer preconditioned {GMRES} algorithm},
  SIAM Journal on Scientific Computing, 14 (1993), pp.~461--469.

\bibitem{saad1986}
{\sc Y.~Saad and M.~H. Schultz}, {\em {GMRES: A} generalized minimal residual
  algorithm for solving nonsymmetric linear systems}, SIAM Journal on
  Scientific and Statistical Computing, 7 (1986), pp.~856--869.

\bibitem{sacco17}
{\sc R.~Sacco, P.~Causin, C.~Lelli, and M.~T. Raimondi}, {\em A poroelastic
  mixture model of mechanobiological processes in biomass growth: theory and
  application to tissue engineering}, Meccanica, 52 (2017), pp.~3273--3297.

\bibitem{showalter05}
{\sc R.~E. Showalter}, {\em Poroelastic filtration coupled to {S}tokes flow},
  in Control Theory of Partial Differential Equations, O.~Imanuvilov,
  G.~Leugering, R.~Triggiani, and B.-Y. Zhang, eds., vol.~242 of Lecture Notes
  in Pure and Applied Mathematics, Chapman \& Hall, Boca Raton, 2005,
  pp.~229--241.

\bibitem{spieker2017abnormal}
{\sc M.~Spieker, S.~Haberkorn, M.~Gastl, P.~Behm, S.~Katsianos, P.~Horn,
  C.~Jacoby, B.~Schnackenburg, P.~Reinecke, M.~Kelm, et~al.}, {\em Abnormal t2
  mapping cardiovascular magnetic resonance correlates with adverse clinical
  outcome in patients with suspected acute myocarditis}, Journal of
  Cardiovascular Magnetic Resonance, 19 (2017), p.~38.

\bibitem{spieker2017t2}
{\sc M.~Spieker, E.~Katsianos, M.~Gastl, P.~Behm, P.~Horn, C.~Jacoby,
  B.~Schnackenburg, P.~Reinecke, M.~Kelm, R.~Westenfeld, et~al.}, {\em T2
  mapping cardiovascular magnetic resonance identifies the presence of
  myocardial inflammation in patients with dilated cardiomyopathy as compared
  to endomyocardial biopsy}, European Heart Journal-Cardiovascular Imaging, 19
  (2017), pp.~574--582.

\bibitem{sundnes07}
{\sc J.~Sundnes, G.~T. Lines, X.~Cai, B.~F. Nielsen, K.-A. Mardal, and
  A.~Tveito}, {\em Computing the Electrical Activity in the Heart}, vol.~1,
  Springer-Verlag, Berlin Heidelberg, 2007.

\bibitem{thompson19}
{\sc T.~B. Thompson, B.~M. Rivi{\`e}re, and M.~G. Knepley}, {\em An implicit
  discontinuous {G}alerkin method for modeling acute edema and resuscitation in
  the small intestine}, Math. Med. Biol., 36 (2019), pp.~513--548.

\bibitem{vanFurth1968}
{\sc R.~van Furth and Z.~A. Cohn}, {\em The origin and kinetics of mononuclear
  phagocytes}, Journal of Experimental Medicine, 128 (1968), pp.~415--435.

\bibitem{verma21}
{\sc N.~Verma, B.~G{\'o}mez-Vargas, L.~M. De~Oliveira~Vilaca, S.~Kumar, and
  R.~Ruiz-Baier}, {\em Well-posedness and discrete analysis for
  advection-diffusion-reaction in poroelastic media}, Applicable Analysis, in
  press (2021).

\bibitem{warriner18}
{\sc D.~R. Warriner, T.~Jackson, E.~Zacur, E.~Sammut, P.~Sheridan, D.~R. Hose,
  P.~Lawford, R.~Razavi, S.~A. Niederer, C.~A. Rinaldi, and P.~Lamata}, {\em An
  asymmetric wall-thickening pattern predicts response to cardiac
  resynchronization therapy}, JACC: Cardiovasc. Imag., 11 (2018),
  pp.~1545--1546.

\bibitem{White201655}
{\sc J.~White, N.~Castelletto, and H.~Tchelepi}, {\em Block-partitioned solvers
  for coupled poromechanics: A unified framework}, Computer Methods in Applied
  Mechanics and Engineering, 303 (2016), pp.~55--74.

\bibitem{white2019two}
{\sc J.~A. White, N.~Castelletto, S.~Klevtsov, Q.~M. Bui, D.~Osei-Kuffuor, and
  H.~A. Tchelepi}, {\em A two-stage preconditioner for multiphase poromechanics
  in reservoir simulation}, Computer Methods in Applied Mechanics and
  Engineering, 357 (2019), p.~112575.

\bibitem{yang2002boomeramg}
{\sc U.~M. Yang and V.~E. Henson}, {\em {BoomerAMG}: a parallel algebraic
  multigrid solver and preconditioner}, Applied Numerical Mathematics, 41
  (2002), pp.~155--177.

\bibitem{young12}
{\sc J.~Young, B.~M. Rivi{\`e}re, C.~S. Cox, and K.~Uray}, {\em A mathematical
  model of intestinal oedema formation}, Math. Med. Biol., 31 (2012),
  pp.~1--15.

\bibitem{yu19}
{\sc C.~Yu, K.~Malakpoor, and J.~M. Huyghe}, {\em A mixed hybrid finite element
  framework for the simulation of swelling ionized hydrogels}, Comput. Mech.,
  63 (2019), pp.~835--852.

\bibitem{zheng20}
{\sc P.~Zheng, G.~Jiao, and K.~Zhang}, {\em A mixed stabilized finite element
  formulation for finite deformation of a poroelastic solid saturated with a
  compressible fluid}, Arch. Appl. Mech., in press (2020), pp.~1--19.

\bibitem{zheng19}
{\sc P.~Zheng and K.~Zhang}, {\em On the effective stress law and its
  application to finite deformation problems in a poroelastic solid}, Int. J.
  Mech. Sci., 161-162 (2019), p.~e105074.

\end{thebibliography}
\end{document}